\documentclass[12pt]{amsart}
 \usepackage{amssymb}
 \usepackage{latexsym}
 \usepackage{graphicx}
 \usepackage{multicol}
 \usepackage{mathrsfs}
 \usepackage{pstricks, pst-node}
 \usepackage{young}
 \usepackage[vcentermath,enableskew,stdtext]{youngtab}
 \usepackage[left=1.18in,right=1.18in,top=1.1in,bottom=1.14in]{geometry}
 \usepackage[all]{xy}
    \SelectTips{cm}{10}     %use the nicer arrowheads
    \everyxy={<2.5em,0em>:} % Sets the scale I like
 \usepackage{fancyhdr}      %%%%%%%%%%%% Pagestyle stuff %%%%%%%%%%%%%%%
 \usepackage{multicol}
 \usepackage{multirow}
 \usepackage{enumitem}
 \usepackage{stmaryrd}

\newcounter{ctr}
\theoremstyle{plain}
\newtheorem{theorem}{Theorem}[section]

\newtheorem*{lemma*}{Lemma}
\newtheorem{lemma}[theorem]{Lemma}
\newtheorem{corollary}[theorem]{Corollary}
\newtheorem{proposition}[theorem]{Proposition}
\newtheorem{conjecture}[theorem]{Conjecture}

\newtheorem{propdef}[theorem]{Proposition-Definition}

\theoremstyle{definition}
\newtheorem{definition}[theorem]{Definition}

\newtheorem{remark}[theorem]{Remark}
\newtheorem{example}[theorem]{Example}
\newtheorem{algorithm}[theorem]{Algorithm}
\renewcommand{\AA}{\ensuremath{\mathbb{A}}}

\newcommand{\ann}{\text{\rm Ann}\,}

\newcommand{\CC}{\ensuremath{\mathbb{C}}}

\newcommand{\gl}{\ensuremath{\mathfrak{gl}}}

\newcommand{\gr}{\text{\rm gr}}
\newcommand{\grin}{\text{\rm in}}
\renewcommand{\H}{\ensuremath{\mathscr{H}}}

\renewcommand{\hom}{\text{\rm Hom}}

\newcommand{\idelm}{\ensuremath{id}}

\newcommand{\I}[2]{\ensuremath{\mathcal{I}^{\text{#1}}_{#2}}}
\newcommand{\Il}[1]{\ensuremath{\I{#1}{\lambda}}}
\newcommand{\ind}{\text{\rm Ind}}

\renewcommand{\L}{\ensuremath{\mathscr{L}}}

\newcommand{\Mod}{\ensuremath{\mathbf{Mod}}}

\newcommand{\R}{\ensuremath{\mathscr{R}}}

\newcommand{\sgn}{\text{\rm sgn}}
\newcommand{\spec}{\text{\rm Spec\,}}

\newcommand{\ZZ}{\ensuremath{\mathbb{Z}}}
\newcommand{\be}{\begin{equation}}
\newcommand{\ee}{\end{equation}}
\newcommand{\Res}{\text{\rm Res}}
\renewcommand{\S}{\ensuremath{\mathcal{S}}}
\newcommand{\tsr}{\ensuremath{\otimes}}
\newcommand{\C}{\ensuremath{C^{\prime}}} %canonical basis

\newcommand{\br}[1]{\ensuremath{\overline{#1}}}

\renewcommand{\u}{\ensuremath{u}}  %q^{1/2}
\newcommand{\ui}{\ensuremath{u^{-1}}} %q^{-1/2}

\newcommand{\leftexp}[2]{{\vphantom{#2}}^{#1}{#2}}
\newcommand{\leftsub}[2]{{\vphantom{#2}}_{#1}{#2}}

\newcommand{\Tab}{\ensuremath{\mathcal{T}}}
\newcommand{\word}{\ensuremath{\mathcal{W}}}
\newcommand{\rsd}[1]{\ensuremath{\hat{#1}}}
 \newcommand{\dual}[1]{\ensuremath{{#1}^\vee}}

\newcommand{\y}{\ensuremath{Y}} %the Bernstein basis
\newcommand{\eW}{\ensuremath{{W_e}}}
\newcommand{\pW}{\ensuremath{{W_e^+}}}
\newcommand{\eH}{\ensuremath{\widehat{\H}}} %extended affine hecke algebra of type An-1
\newcommand{\pH}{\ensuremath{\widehat{\H}^{+}}} %positive extended affine hecke alg
\newcommand{\pY}{\ensuremath{Y^+}}
\newcommand{\aH}{\ensuremath{\widetilde{\H}}}
\newcommand{\eS}{\widehat{\S}}
\newcommand{\pS}{\widehat{\S}^+}
\newcommand{\aW}{\ensuremath{{W_a}}}
\newcommand{\cs}{CSQ\ensuremath{(\pH)}}
\newcommand{\pPi}{\ensuremath{\Pi^+}}

\newcommand{\lC}{\ensuremath{\overleftarrow{C}^{\prime}}}
\newcommand{\rC}{\ensuremath{\overrightarrow{C}^\prime}}
\newcommand{\lP}{\ensuremath{\overleftarrow{P}}}

\newcommand{\lj}[2]{{#1}_{#2}}
\newcommand{\lJ}[2]{{#1}^{#2}}
\newcommand{\rj}[2]{\leftsub{#2}{#1}}
\newcommand{\rJ}[2]{\leftexp{#2}{#1}}

\newcommand{\klo}[1]{\ensuremath{\leq_{#1}}}
\newcommand{\klocov}[1]{\ensuremath{\xleftarrow[#1]{}}}
\newcommand{\cinvn}[1]{\ensuremath{\mathscr{R}_{1^{#1}}}}
\newcommand{\cinv}{\ensuremath{\cinvn{n}}}

\newcommand{\cc}{\ensuremath{\xrightarrow{\text{cc}}}}
\newcommand{\cat}{\text{rcat}}
\newcommand{\ccat}{\text{ccat}}

\DeclareMathOperator{\ccp}{CCP}
\newcommand{\charge}{\text{charge}}
\newcommand{\ccharge}{\text{cocharge}}
\newcommand{\cl}[1]{\ensuremath{{#1}^\text{cc}}}

\DeclareMathOperator{\ctype}{ctype}
\newcommand{\invlr}{\ensuremath{\Psi}}
\DeclareMathOperator{\rank}{rank}

\newcommand{\gd}{\ensuremath{\trianglerighteq}}
\newcommand{\ld}{\ensuremath{\trianglelefteq}}

\newcommand{\tto}{\ensuremath{\rightsquigarrow}}
\newcommand{\tleftright}{\ensuremath{\leftrightsquigarrow}}

\newcommand{\reading}{\text{\rm rowword}}
\newcommand{\creading}{\text{\rm colword}}
\newcommand{\sh}{\text{\rm sh}}

\DeclareMathOperator{\row}{row}
\DeclareMathOperator{\col}{column}
\newcommand{\skl}[3]{\ensuremath{#1 \xrightarrow{#2} #3}}

\newcommand{\ds}{\ensuremath{D^S}}
\newcommand{\dsw}{\ensuremath{D^S w_0}}

\newcommand{\chena}[2]{\ensuremath{\AA^\text{Chen}_{#1, #2}}} %Li-Chung Chen atom indexed by a pair of tableaux
\newcommand{\csqa}[2]{\ensuremath{\AA^\text{csq}_{#1, #2}}} % Cellular subquotient by a pair of tableaux
\newcommand{\moda}[2]{\ensuremath{\AA^\text{mod}_{#1, #2}}}
\newcommand{\swca}[2]{\ensuremath{\AA}^\text{SWc}_{#1,#2}}
\newcommand{\swra}[2]{\ensuremath{\AA}^\text{SWr}_{#1,#2}}
\newcommand{\gpda}[2]{\ensuremath{\AA}^{\text{GP}^\vee}_{#1,#2}}
\newcommand{\gpa}[2]{\ensuremath{\AA}^{\text{GP}}_{#1,#2}}
\newcommand{\sgnpa}[1]{\ensuremath{\AA}^{\text{sgn}}_{#1}}
\newcommand{\nnA}{\ensuremath{A_{\geq 0}}}
\newcommand{\posA}{\ensuremath{A_{>0}}}
\newcommand{\fxccp}{\ensuremath{F^\text{xccp}}}
\newcommand{\fmod}{\ensuremath{F^\text{mod}}}
\newcommand{\fsp}{\ensuremath{F^\text{ccp}}}
\newcommand{\lrcelllong}[1]{\ensuremath{\mathbf{c}_{#1}}}
\newcommand{\lrcell}{\ensuremath{\lrcelllong{\nu}}}
\newcommand{\sgnrec}[1]{\ensuremath{\textbf{sgnQ}_{#1}}}
\newcommand{\wposet}{\ensuremath{\mathscr{P}}}

\renewcommand{\ng}{\text{-}}

\DeclareMathOperator{\sgnp}{sgn_P}
\DeclareMathOperator{\sgnq}{sgn_Q}
\DeclareMathOperator{\partition}{Part}

\newlength{\cellsize}
\newcommand\tableau[1]{
\vcenter{
\let\\=\cr
\baselineskip=-16000pt
\lineskiplimit=16000pt
\lineskip=0pt
\halign{&\tableaucell{##}\cr#1\crcr}}}

\newcommand{\tableaucell}[1]{{%
\def \arg{#1}\def \void{}%
\ifx \void \arg
\vbox to \cellsize{\vfil \hrule width \cellsize height 0pt}%
\else
\unitlength=\cellsize
\begin{picture}(1,1)
\put(0,0){\makebox(1,1){$#1$}}
\put(0,0){\line(1,0){1}}
\put(0,1){\line(1,0){1}}
\put(0,0){\line(0,1){1}}
\put(1,0){\line(0,1){1}}
\end{picture}%
\fi}}

\begin{document}

\author{Jonah Blasiak$^\dagger$}
\title{Cyclage, catabolism, and the affine Hecke algebra}
\thanks{$^\dagger$Is currently an NSF postdoctoral fellow. This research was partially conducted during the period the author was employed by the Clay Mathematics Institute as a Liftoff Fellow.}
\address{The University of Chicago, Department of Computer Science, 1100 East 58th Street, Chicago, IL 60637, USA}
\email{jblasiak@gmail.com}

\begin{abstract}
We identify a  subalgebra $\pH_n$ of the extended affine Hecke algebra $\eH_n$ of type $A$.  The subalgebra $\pH_n$ is a $\u$-analogue of the  monoid algebra of $\S_n \ltimes \ZZ_{\geq 0}^n$ and inherits a canonical basis from that of $\eH_n$. We show that its left cells are naturally labeled by tableaux filled with positive integer entries having distinct residues mod $n$, which we term \emph{positive affine tableaux} (PAT).

We then exhibit a cellular subquotient $\cinv$ of $\pH_n$ that is a $\u$-analogue of the ring of coinvariants $\CC[y_1,\ldots,y_n]/(e_1, \ldots,e_n)$ with left cells labeled by PAT that are essentially standard Young tableaux with cocharge labels. Multiplying canonical basis elements by a certain element $\pi \in \pH_n$ corresponds to rotations of words, and on cells corresponds to cocyclage.
We further show that $\cinv$ has cellular quotients  $\R_\lambda$ that are $\u$-analogues of the Garsia-Procesi modules $R_\lambda$ with left cells labeled by (a PAT version of) the  $\lambda$-catabolizable tableaux.

We give a conjectural description of a cellular filtration of $\pH_n$, the subquotients of which are isomorphic to dual versions of $\R_\lambda$ under the perfect pairing on $\cinv$.  This turns out to be closely related to the combinatorics of the cells of  $\eH_n$ worked out by Shi, Lusztig, and Xi, and we state explicit conjectures along these lines.
We also conjecture that the $k$-atoms of Lascoux, Lapointe, and Morse \cite{LLM} and the $R$-catabolizable tableaux of Shimozono and Weyman \cite{SW} have cellular counterparts in $\pH_n$.  We extend the idea of atom copies from \cite{LLM} to positive affine tableaux and give descriptions, mostly conjectural, of some of these copies in terms of catabolizability.

\end{abstract}

\maketitle
\section{Introduction}
It is well-known that the ring of coinvariants
$R_{1^n} = \CC[y_1,\ldots,y_n]/(e_1, \ldots,e_n)$, thought of as a
$\CC\S_n$-module with $\S_n$ acting by permuting the variables, is a
graded version of the regular representation.  However, how a
decomposition of this module into irreducibles is compatible with
multiplication by the $y_i$ remains a mystery.

A precise question one can ask along these lines  goes as follows.
Let $E \subseteq R_d$ be an $\S_n$-irreducible, where $R_d$ is the $d$-th graded part of the  polynomial ring $R = \CC[y_1,\ldots,y_n]$.  Suppose that the isotypic component of $R_d$ containing $E$ is $E$ itself.  Then define
$I \subseteq R$ to be the sum of all homogeneous ideals $J \subseteq R$ that are left stable under the $\S_n$-action and satisfy $J \cap E = 0$.
The quotient $R/I$ contains $E$ as the unique $\S_n$-irreducible of top degree $d$.
It is natural to ask

\begin{center}What is the graded character of $R/I$?
\end{center}

The most familiar examples of such quotients are the Garsia-Procesi modules $R_\lambda$ (see \cite{GP}), which correspond to the case that $E$ is of shape $\lambda$ and $d = n(\lambda) = \sum_i (i-1)\lambda_i$; refer to this representation $E \subseteq R_{n(\lambda)}$ as the \emph{Garnir representation of shape} $\lambda$ or, more briefly, $G_\lambda$. Combining the work of Hotta-Springer (see \cite{H3}) and Lascoux \cite{La} (see also \cite{SW}) gives the Frobenius series
\be
\label{e frobenius}
\mathscr{F}_{R_\lambda}(t) = \sum_{\substack{T \in SYT\\\ctype(T) \gd \lambda}} t^{\ccharge(T)} s_{\sh(T)},\ee
where $\ctype(T)$ is the catabolizability of $T$ (see \textsection\ref{ss catabolizability 2}).

Though this interpretation of the character of $R_\lambda$ has been known for some time, the only proofs were difficult and indirect. One of the goals of this research, towards which we have been partially successful, was to give a more transparent explanation of the appearance of catabolism in the combinatorics of the coinvariants.

%Not much is known about this problem outside of the Garsia-Procesi case, and even there, the combinatorics of catabolism is somewhat mysterious.
More recent work suggests that there are other combinatorial mysteries hiding in the ring of coinvariants. We strongly suspect that modules with graded characters corresponding to the $k$-atoms of Lascoux, Lapointe, and Morse \cite{LLM} and a generalization of $k$-atoms due to Li-Chung Chen \cite{Ch} sit inside the coinvariants as subquotients. It is also natural to conjecture that the generalization of catabolism due to Shimozono and Weyman \cite{SW} gives a combinatorial description of certain subquotients of the coinvariants which are graded versions of induction products of  $\S_n$-irreducibles.

This paper describes an approach to these problems using canonical bases, which has so far been quite successful and will hopefully help solve some of the difficult conjectures in this area.
After reviewing the necessary background on Weyl groups and Hecke algebras (\textsection\ref{s Hecke algebra}) and canonical bases and cells (\textsection\ref{s canonical bases and cells}), we introduce the central algebraic object of our work, a subalgebra $\pH$ of the extended affine Hecke algebra which is a $\u$-analogue of the monoid algebra of $\S_n \ltimes \ZZ_{\geq 0}^n$. In \textsection\ref{s type A and the positive part of eH}, we establish some basic properties of this subalgebra and describe its left cells.
It turns out that these cells are naturally labeled by tableaux filled with positive integer entries having distinct residues mod $n$, which we term \emph{positive affine tableaux} (PAT). Our investigations have convinced us that these are excellent combinatorial objects for describing graded $\S_n$-modules.

%In \textsection\ref{s cocyclage catabolism atoms}, we interrupt our study of $\pH$ to introduce the combinatorics of cyclage and catabolism, though with adaptations to express these notions using positive affine tableaux. In \textsection\ref{ss atom categories}, we introduce some formalism to compare sets of tableaux, cyclage posets, and cellular subquotients of $\pH$.

After some preparatory combinatorics and formalism in \textsection\ref{s cocyclage catabolism atoms}, we go on to show in \textsection\ref{s pW-graph version of R_(1^n)} that $\pH$ has a cellular quotient $\cinv$ that is a $\u$-analogue of $R_{1^n}$.  The module $\cinv$ has a canonical basis labeled by affine words that are essentially standard words with cocharge labels, with left cells labeled by PAT that are essentially standard tableaux with cocharge labels.
Multiplying canonical basis elements by a certain element $\pi \in \pH$ corresponds to rotations of words, and on left cells corresponds to cocyclage.

In this cellular picture of the coinvariants,  $G_\lambda$ corresponds to a left cell of $\cinv$ labeled by a PAT of shape $\lambda$, termed the \emph{Garnir tableau of shape}  $\lambda$, again denoted $G_\lambda$. In \textsection\ref{s extending garsia-procesi}, we identify $\u$-analogues $\R_\lambda$ of  the $R_\lambda$ and give several equivalent descriptions of these objects. Most importantly, we show that $\R_\lambda$ is cellular and its left cells are labeled by  (a PAT version of) the $\lambda$-catabolizable tableaux. The proof uses several ingredients:
\begin{itemize}
\item The positivity of the structure coefficients of the canonical basis of $\pH$,
\item Identifying certain canonical basis elements of $\pH$ as elementary symmetric functions in subsets of the Bernstein generators $\y_1, \dots, \y_n$ (Theorem \ref{t elemsymkl}),
\item The $\u=1$ results of Garsia-Procesi and Bergeron-Garsia.
\end{itemize}
Given these ingredients, the proof is quite easy. One of the hopes of this approach was to give a proof of equation (\ref{e frobenius}) not relying on the $\u=1$ results. Though we have not yet achieved this goal, the cellular picture provided by $\pH$ seems to give an extremely good way of connecting representation theory with difficult combinatorics, both intuitively and conjecturally.

There is a well-known perfect pairing $\langle , \rangle: R_{1^n} \times R_{1^n} \to \CC$ given by $\langle f_1, f_2 \rangle$ equal to the projection of $f_1 f_2$ onto the sign representation of $R_{1^n}$. In section \ref{cj flip}, we conjecture a stronger duality for the canonical basis of $\cinv$ which is surprisingly subtle.  Under the perfect pairing, the Garsia-Procesi modules correspond to what we call dual Garsia-Procesi modules. If the conjectured duality holds, then these modules have $\u$-analogues that are cellular, called dual GP csq (csq stands for cellular subquotient).

The final goal of this paper, the subject of \textsection\ref{s atoms}, is to describe our progress towards connecting more elaborate combinatorics with other cellular subquotients of $\pH$. Though we are primarily interested in subquotients of the coinvariants, it appears that there are many other copies of these subquotients in $\pH$. Though we believe these copies to be isomorphic as cellular subquotients, they come with genuinely different combinatorics, just as the cocyclage poset on semistandard tableaux is not obviously isomorphic to a subposet of the cocyclage poset on standard tableaux.  We conjecture that there is a cellular filtration of $\pH$, the subquotients of which are isomorphic to dual GP csq.    This turns out to be closely related to the combinatorics of the cells of the extended affine Weyl group worked out by Shi, Lusztig, and Xi \cite{Shi, L two-sided An, X2}, and we state explicit conjectures along these lines. We also conjecture descriptions of some of these copies of dual GP csq in terms of a version of catabolizability  for PAT. In \textsection\ref{ss SSYT in PAT} we show that a certain subset of PAT are essentially the same as semistandard tableaux of partition content, and conjecture a similar statement for arbitrary content. This leads to a new interpretation of charge for semistandard tableaux (proven for partition content, conjectural in general).

We also conjecture that the $R$-catabolizable tableaux of Shimozono and Weyman, the $k$-atoms of Lascoux, Lapointe, and Morse, and Chen's atoms all have cellular counterparts in $\pH$. The conjectural isomorphic copies of such atoms in $\pH$ generalize both Lascoux's standardization map \cite{La} and the atom copies in \cite{LLM}.
We believe that a critical problem towards understanding $k$-atoms and catabolizability is to produce a combinatorial structure less rigid than tableaux that makes it obvious that these copies are isomorphic. See the introduction to \textsection\ref{s atoms} and Remark \ref{r sign insertion fail} for more about this.

%The grading of $\pH$ automatically carries the combinatorics of cocharge for standard tableaux.
\section{Hecke algebras}
\label{s Hecke algebra}
Following \cite{H} (see also \cite{L3}), we introduce Weyl groups and Hecke algebras in full generality. In \textsection\ref{s type A and the positive part of eH} and on, we work only in type A.

\subsection{}
\label{ss coxeter groups}
Let $(W,S)$ be a Coxeter group and $\Pi$ an abelian group acting on $(W,S)$ by automorphisms. The extended Coxeter group associated to this data is the pair $(\eW, S)$, where $\eW$ is the semidirect product $\Pi \ltimes W$. The length function $\ell$ and partial order $\leq$ on $W$ extend to $\eW$: $\ell(\pi v) = \ell(v)$, and $\pi v \leq \pi' v'$ if and only if $\pi = \pi'$ and $v \leq v'$, where $\pi, \pi' \in \Pi$, $v, v' \in W$.

%The \emph{length} $\ell(w)$ of $w \in \eW$ is the minimal $l$ such that $w=s_1\ldots s_l$ for some $s_i\in S$.
If $\ell(uv)=\ell(u)+\ell(v)$, then $uv = u\cdot v$ is a \emph{reduced factorization}. The notation $L(w) = \{s\in S : sw < w\}, R(w) = \{s\in S : ws < w\}$ will be used for the left and right descent sets of $w$.

Although it is possible to allow parabolic subgroups to be extended Coxeter groups, we define a parabolic subgroup of $\eW$ to be an ordinary parabolic subgroup of $W$ to simplify the discussion (this is the only case we will need).  %With this convention, each coset of a parabolic subgroup $\eW_J$ contains a unique element of minimal length.

For any $J\subseteq S$, the \emph{parabolic subgroup} $\eW_J = W_J$ is the subgroup of $\eW$ generated by $J$. Each left (resp. right) coset $w\eW_J$ (resp. $\eW_Jw$) of $\eW_J$ contains a unique element of minimal length called a minimal coset representative. The set of all such elements is denoted $\eW^J$ (resp. $\leftexp{J}{\eW}$). For any $w \in \eW$, define $\lJ{w}{J}$, $\rj{w}{J}$ by
\be w=\lJ{w}{J} \cdot \rj{w}{J},\ \lJ{w}{J} \in \eW^J,\ \rj{w}{J} \in \eW_J.\ee
Similarly, define $\lj{w}{J}$, $\rJ{w}{J}$ by
\be w= \lj{w}{J} \cdot \rJ{w}{J},\ \lj{w}{J} \in \eW_J,\ \rJ{w}{J} \in \leftexp{J}\eW.\ee

\subsection{} \label{ss affine weyl groups}
Let $(Y, \alpha'_i, \alpha'^\vee_i)$, $i \in [n-1]$ be the root system specifying a reductive algebraic group $G$ over $\CC$. Write $Y^\vee$ for the dual lattice $\hom(Y,\ZZ)$ and $\langle\ ,\ \rangle$ for the pairing between $Y$ and $Y^\vee$. Let $W_f$ be the Weyl group of this root system and $S = \{s_1, \dots, s_{n-1}\}$ the set of simple reflections.
%The simple reflection $s_i$ can be thought of as the reflection of $Y$ about the hyperplane $\langle\lambda, \alpha_i^\vee\rangle =0$ that sends $\alpha_i$ to $-\alpha_i$.
The group $W_f$ is the subgroup of automorphisms of the lattice $Y$ generated by the reflections $s_i$. Let $R'_f$ be the set of roots and $Q'_f$ the root lattice.

The \emph{extended affine Weyl group} is the semidirect product
\[W_e:=Y \rtimes W_f.\]
Elements of $Y \subseteq \eW$ will be denoted by the multiplicative notation $y^\lambda, \lambda\in Y$.

The group $W_e$ is also equal to $\Pi \ltimes W_a$, where $W_a$ is the Weyl group of an affine root system we will now construct and $\Pi$ is an abelian group. Let $X=Y^\vee\oplus\ZZ$ and $\delta$ be a generator of $\ZZ$. The pairing of $X$ and $X^\vee$ is obtained by extending the pairing of $Y$ and $Y^\vee$ together with $\langle\delta, Y\rangle=0$. Let $\phi'$ be the dominant short root of $(Y, \alpha'_i, \alpha'^\vee_i)$ and $\theta=\phi'^\vee$ the highest coroot. For $i\neq 0$ put $\alpha_i=\alpha'^\vee_i$ and $\alpha^\vee_i=\alpha'_i$; put $\alpha_0=\delta-\theta$ and $\alpha^\vee_0=-\phi'$. Then $(X, \alpha_i, \alpha^\vee_i)$, $i \in [0,n-1]$ is an affine root system with Weyl group $\aW$.

The abelian group $Q'_f$ is realized as a subgroup of $W_a$ acting on $X$ and $X^\vee$ by translations. This action extends to an action of $Y$, which realizes $W_e$ as a subgroup of the automorphisms of $X$ and $X^\vee$. The inclusion $W_a\hookrightarrow W_e$ is given on simple reflections by $s_i \mapsto s_i$ for $i\neq 0$ and $s_0 \mapsto y^{\phi'}s_{\phi'}$. The subgroup $\aW$ is normal in $\eW$ with quotient $W_e/W_a \cong Y/Q'_f$, denoted $\Pi$. And, as was our goal, we have $W_e=\Pi\ltimes W_a$.

%Let $H=\{x\in X^\vee_\RR:\langle\delta,x\rangle=1\}$ be the \emph{level 1 plane}. The \emph{basic alcove} $A_0=\{x\in H:\langle x, \alpha \rangle \geq 0$ for all $\alpha\in R_+\}$. The group $\eW$ acts on $X$, fixing $\delta$, and therefore acts on $H$. The stabilizer of $A_0$ is the subgroup $\Pi$.

The set of dominant weights $Y_+$ is the cone in $Y$ given by
\be Y_+ = \{ \lambda \in Y: \langle \lambda,\alpha'^\vee_i\rangle \geq 0 \text{ for all } i \}. \ee

Let $K = \{s_0, s_1, \dots, s_{n-1}\}$ be the set of simple reflections of $W_a$. The pairs $(W_f, S)$ and $(\aW, K)$ are Coxeter groups, and $(\eW, K)$ is an extended Coxeter group. The parabolic subgroup $\eW_S$ is equal to $W_f$.

%how does W_e act on Y??
\subsection{}

Let $A = \ZZ[\u,\ui]$ be the ring of Laurent polynomials in the indeterminate $\u$ and $A^-$ be the subring $\ZZ[\ui]$. The \emph{Hecke algebra} $\H(W)$ of an (extended) Coxeter group $(W, S)$ is the free $A$-module with basis $\{T_w :\ w\in W\}$ and relations generated by
\begin{equation}\label{hecke eq}
\begin{array}{ll}T_uT_v = T_{uv} & \text{if } uv = u\cdot v\ \text{is a reduced factorization}\\
 (T_s - \u)(T_s + \ui) = 0 & \text{if } s\in S.\end{array}
\end{equation}

For each $J\subseteq S$, $\H(W)_J$ denotes the subalgebra of $\H(W)$ with $A$-basis $\{T_w:\ w\in W_J\}$, which is also the Hecke algebra of $W_J$.
\subsection{}
\label{ss two presentations of eH}
The \emph{extended affine Hecke algebra} $\eH$ is the Hecke algebra $\H(\eW)$. Just as the extended affine Weyl group $W_e$ can be realized both as $\Pi\ltimes W_a$ and $Y \rtimes W_f$, the extended affine Hecke algebra can be realized in two analogous ways:

The algebra $\eH$ contains the Hecke algebra $\H(\aW)$ and is isomorphic to the twisted group algebra $\Pi\cdot\H(W_a)$ generated by $\Pi$ and $\H(W_a)$ with relations generated by
$$\pi T_w=T_{\pi w \pi^{-1}}\pi$$
for $\pi\in\Pi$, $w\in W_a$.

There is also a presentation of $\eH$ due to Bernstein. For any $\lambda\in Y$ there exist $\mu,\nu\in Y_+$ such that $\lambda = \mu-\nu$. Define
\[\y^\lambda := T_{y^\mu}(T_{y^\nu})^{-1},\]
which is independent of the choice of $\mu$ and $\nu$. The algebra $\eH$ is the free $A$-module with basis $\{\y^\lambda T_w :\ w\in W_f,\lambda\in Y\}$ and relations generated by
\[ \begin{array}{ll}T_i\y^\lambda=\y^\lambda T_i & \text{if }\langle\lambda,\alpha'^\vee_i\rangle=0,\\
T^{-1}_i\y^\lambda T^{-1}_i=\y^{s_i(\lambda)} & \text{if }\langle \lambda,\alpha'^\vee_i\rangle=1,\\
(T_i - \u)(T_i + \ui) = 0 & \end{array}\]
for all $i \in [n-1]$, where $T_i := T_{s_i}$. From this, one may deduce the more general commutation relation for $\lambda \in Y$:

\be\label{e TY} T_i \y^\lambda - \y^{s_i(\lambda)} T_i = \frac{(\u-\ui)(\y^{\alpha'_i})}{\y^{\alpha'_i}-1}(\y^\lambda-\y^{s_i(\lambda)}), \quad i \in [n-1]. \ee
Be aware that, in the language of \cite{H}, we are using the right affine Hecke algebra, so this equation differs slightly from its counterpart \cite[(19)]{H} for the left.

We will make use of the following three important bases of $\eH$; the last one, the canonical basis, will be defined in the next section.
\setcounter{ctr}{0}
\begin{list}{\emph{(\roman{ctr})}} {\usecounter{ctr} \setlength{\itemsep}{1pt} \setlength{\topsep}{2pt}}
\item The \emph{standard basis} $\{T_w : w\in W_e\}$,
\item The \emph{Bernstein basis} $\{\y^\lambda T_w:\lambda\in Y,w\in W_f\}$,
\item The \emph{canonical basis} $\{\C_w:w\in W_e\}$.
\end{list}
We remark that $\{T_w \y^\lambda :\lambda\in Y,w\in W_f\}$ is also a basis of $\eH$ and that the results we state using the basis (ii) have counterparts using this basis, but we will not state them explicitly.

\section{Canonical bases and cells}
\label{s canonical bases and cells}
\subsection{}
The bar-involution, $\br{\cdot}$, of $\H(W)$ is the additive map from $\H(W)$ to itself extending the involution $\br{\cdot}$: $A\to A$ given by $\br{u} = \ui$ and satisfying $\br{T_w} = T_{w^{-1}}^{-1}$. Observe that $\br{T_{s}} = T_s^{-1} = T_s + \ui - u$ for $s \in S$. Some simple $\br{\cdot}$-invariant elements of $\H(W)$ are $\C_\text{id} := T_\text{id}$ and $\C_s := T_s + \ui = T_s^{-1} + u$, $s\in S$. The $\br{\cdot}$-invariant $\u$-integers are $[k] := \frac{\u^k - \u^{-k}}{\u - \ui} \in A$.

\subsection{}
\label{ss kl basics}
In \cite{KL}, Kazhdan and Lusztig introduce $W$-graphs as a combinatorial structure for describing an $\H(W)$-module with a special basis. A $W$-graph consists of a vertex set $\Gamma$, an edge weight $\mu(\delta,\gamma)\in \ZZ$ for each ordered pair $(\delta,\gamma)\in\Gamma\times\Gamma$, and a descent set $L(\gamma) \subseteq S$ for each $\gamma\in\Gamma$. These are subject to the condition that $A\Gamma$ has a left $\H(W)$-module structure given by
\begin{equation}\label{Wgrapheq}\C_{s}\gamma = \left\{\begin{array}{ll} [2] \gamma & \text{if}\ s \in L(\gamma),\\
\sum_{\substack{\{\delta \in \Gamma : s \in L(\delta)\}}} \mu(\delta,\gamma)\delta & \text{if}\ s \notin L(\gamma). \end{array}\right.\end{equation}

We will use the same name for a $W$-graph and its vertex set.  If an $\H(W)$-module $E$ has an $A$-basis $\Gamma$ that satisfies (\ref{Wgrapheq}) for some choice of descent sets, then we say that $\Gamma$ gives $E$ a \emph{$W$-graph structure}, or $\Gamma$ is a $W$-graph on $E$.
%We will never put more than one $W$-graph structure on a given $\H(W)$-module, so we will sometimes refer to a $W$-graph $\Gamma$

It is convenient to define two $W$-graphs $\Gamma,\Gamma'$ to be isomorphic if they give rise to isomorphic $\H(W)$-modules with basis. That is, $\Gamma\cong\Gamma'$ if there is a bijection $\alpha:\Gamma\to\Gamma'$ of vertex sets such that $L(\alpha(\gamma)) = L(\gamma)$ and $\mu(\alpha(\delta),\alpha(\gamma)) = \mu(\delta,\gamma)$ whenever $L(\delta)\not\subseteq L(\gamma)$.

 Define the lattice
\[ \L = A^{-} \{ T_w : w \in W \}. \]

\begin{theorem}[Kazhdan-Lusztig \cite{KL}]
\label{t kl canonical basis}
For each $w \in W$, there is  a unique element $\C_w \in \H(W)$ such that $\br{\C_w} = \C_w$ and $\C_w$ is congruent to $T_w \mod \ui \L$.  There exist integers $\mu(x,w)$, $x, w \in W$ so that $\{\C_w: w\in W\}$ gives $\H(W)$ a $W$-graph structure.
 \end{theorem}
The $A$-basis $\{\C_w : w\in W\}$ of $\H(W)$ is the \emph{canonical basis} or Kazhdan-Lusztig basis.  The corresponding $W$-graph is denoted $\Gamma_W$.

The coefficients of the $\C$'s in terms of the $T$'s are the \emph{Kazhdan-Lusztig polynomials} $P'_{x,w}$:
\be \C_w = \sum_{x \in W} P'_{x,w} T_x. \ee
(Our $P'_{x,w}$ are equal to $q^{(\ell(x)-\ell(w))/2}P_{x,w}$, where $P_{x,w}$ are the polynomials defined in \cite{KL} and $q^{1/2} = \u$.) The $W$-graph $\Gamma_W$ may be described in terms of Kazhdan-Lusztig polynomials as follows: the edge-weight $\mu(x,w)$ is equal to the coefficient of $\ui$ in $P'_{x,w}$ (resp. $P'_{w,x}$) if $x \leq w$ (resp. $w \leq x$).

\begin{remark}
\label{r wgraph symmetric}
Not all of the integers $\mu(x,w)$ matter for the $\H(W)$-module structure on $A \Gamma_W$, i.e., different choices of certain edge-weights would lead to isomorphic $W$-graphs. However, the convention above in which $\mu(w,x) = \mu(x,w)$  is sometimes convenient and we maintain this throughout the paper.
\end{remark}

\subsection{}
\label{ss cells}
Let $\Gamma$ be a $W$-graph and put $E=A\Gamma$. The preorder $\klo{\Gamma}$ (also denoted $\klo{E}$) on the vertex set $\Gamma$ is generated by the relations/edges
\be
\label{e preorder}
\delta\klocov{\Gamma}\gamma \begin{array}{c}\text{if there is an $h\in\H(W)$ such that $\delta$ appears with non-zero}\\ \text{coefficient in the expansion of $h\gamma$
in the basis $\Gamma$}. \end{array}
\ee
Equivalence classes of $\klo{\Gamma}$ are the \emph{left cells} of $\Gamma$. Sometimes we will speak of the left cells of $E$ or the preorder on $E$ to mean that of $\Gamma$, when the $W$-graph $\Gamma$ is clear from context.
 A \emph{cellular submodule} of $E$ is a submodule of $E$ that is spanned by a subset of $\Gamma$ (and is necessarily a union of left cells). A \emph{cellular quotient} of $E$ is a quotient of $E$ by a cellular submodule, and a \emph{cellular subquotient} of $E$ is a cellular submodule of a cellular quotient. We will abuse notation and sometimes refer to a cellular subquotient by its corresponding union of cells.

\begin{remark} \label{r edge direction}
Throughout this paper we use the convention that when identifying a poset with a directed acyclic graph, edges are directed from bigger elements to smaller ones.
\end{remark}

\subsection{}
\label{ss easyedges}
The preorder $\klo{E}$ induces a partial order on the cells of $E$, which is also denoted $\klo{E}$. This seems to be quite difficult to compute completely; it is not even known for the $\S_n$-graph $\Gamma_{\S_n}$. We will see some results that help determine $\klo{E}$ throughout the paper. We can state one such result now, which originated in the work of Barbasch and Vogan on primitive ideals, and is proven in the generality stated here by Roichman \cite{R} (see also \cite[\textsection3.3]{B0}).

\begin{proposition}\label{p restrict Wgraph}
Let $J \subseteq S$ and $E = \Res_{\H(W_J)} \H(W)$. Then for any $x \in \leftexp{J}{W}$,
\be A \{\C_{vx} : v \in W_J \} \xrightarrow{\cong} \H(W_J), \C_{vx} \mapsto \C_v \ee
is an isomorphism of $\H(W_J)$-modules with basis (equivalently, the corresponding map of $W_J$-graphs is an isomorphism). In particular, any left cell of $E$ is isomorphic to one occurring in $\H(W_J)$.
\end{proposition}

Despite the difficulty of computing $\klo{E}$, there are two kinds of easy edges that will be of interest to us.

If $\Gamma$ is a cellular subquotient of the $W$-graph $\Gamma_W$, then
\be \C_{sw} \klo{\Gamma} \C_w, \quad \text{if } sw > w, s \in S. \ee
We will refer to such edges as \emph{ascent-edges} and the corresponding edges between cells as \emph{ascent-induced edges} (that is, for left cells $T_1, T_2$ of $\Gamma$, $T_1 \klo{\Gamma} T_2$ is an ascent-induced edge if there exist $\gamma_1 \in T_1$, $\gamma_2 \in T_2$ such that $\gamma_1 \klo{\Gamma} \gamma_2$ is an ascent-edge).

If $\Gamma$ is a cellular subquotient of the $\eW$-graph $\Gamma_\eW$, then
\be \C_{\pi w} \klo{\Gamma} \C_w, \quad \text{for any } \pi \in \Pi, w\in \eW. \ee
A specific instance of this type of edge will be called a corotation-edge (see \ref{ss cells of pH}).

\section{Type $A$ and the positive part of $\eH$}
\label{s type A and the positive part of eH}
Here we introduce a subalgebra $\pH$ of $\eH$ that plays a crucial role in our goal of relating subquotients of $R$ to tableau combinatorics. We also introduce the set of affine tableaux (AT) and positive affine tableaux (PAT), which label left cells of $\Res_\H \eH$ and $\pH$.

\subsection{}
From now on, specialize to the case $G=GL_n$. The groups $W_f, \aW, \eW$, roots $R'_f$, root lattice $Q'_f$, etc. are now understood to be those of type A. Let $\H, \aH, \eH$ denote the Hecke algebras of $W_f, \aW, \eW$, sometimes decorated with a subscript $n$ to emphasize that they correspond to type $A_{n-1}$ or $\tilde{A}_{n-1}$. As in \textsection\ref{ss affine weyl groups}, $S = \{s_1, \dots, s_{n-1}\}$ are the simple reflections of $W_f$ and $K = \{s_0, \dots, s_{n-1}\}$ are those of $\aW$ and $\eW$.

The lattices $Y$ and $Y^\vee$ are equal to $\ZZ^n$ and $\alpha'_i=\epsilon_i-\epsilon_{i+1}$, $\alpha'^\vee_i=\epsilon^\vee_i-\epsilon^\vee_{i+1}$, where $\epsilon_i$ and $\epsilon^\vee_i$ are the standard basis vectors of $Y$ and $Y^\vee$. The finite Weyl group $W_f$ is $\S_n$ and the subgroup $\Pi$ of $W_e$ is $\ZZ$. The element $\pi=y_1s_1s_2\ldots s_{n-1}\in\Pi$ is a generator of $\Pi$. This satisfies the relation $\pi s_i=s_{i+1}\pi$, where, here and from now on, the subscripts of the $s_i$ are taken mod $n$.

Here is a table that summarizes the algebras defined so far and some to be defined shortly.
\[\begin{array}{c|c|c}
\text{Group, monoid, etc.} & \text{Group algebra over $\CC$} & \u\text{-analogue}\\
\S_n=W_f & \CC\S_n & \H_n\\
\widetilde{\S_n}=W_a\cong Q\rtimes W_f & \CC\widetilde{\S_n} & \aH_n\\
\widehat{\S_n}=W_e\cong Y\rtimes W_f & \CC[{y_1}^{\pm 1},\ldots, {y_n}^{\pm 1}] \star\S_n := \CC(Y\rtimes W_f) & \eH_n\\
\widehat{\S_n}^+=W^+_e\cong \pY\rtimes W_f & \CC[y_1,\ldots, y_n]\star\S_n := \CC(\pY \rtimes W_f) & \pH_n\\
\pY & R=\CC[y_1,\ldots, y_n] & \R\\
\ds & R_{1^n}=\CC[y_1,\ldots, y_n]/(e_1,\ldots,e_n) & \cinv
\end{array}\]

\subsection{}
\label{ss words of eW}
Another description of $\eW$, due to Lusztig, identifies it with the group of permutations $w: \ZZ \to \ZZ$ satisfying $w(i+n) = w(i)+n$ and $\sum_{i = 1}^n (w(i) - i) \equiv 0$ mod $n$. The identification takes $s_i$ to the permutation transposing $i+kn$ and $i+1+kn$ for all $k \in \ZZ$, and takes $\pi$ to the permutation $k \mapsto k+1$ for all $k \in \ZZ$. We take the convention of specifying the permutation of an element $w \in \eW$ by the word
\[{\small n+1-w^{-1}(1)\ \ n+1-w^{-1}(2)\ \dots\ n+1-w^{-1}(n).}\]
We refer to this as the \emph{inverted window word}, \emph{affine word}, or simply \emph{word} of $w$, and, when there is no confusion, the word of $w$ will be written as $w_1 w_2 \cdots w_n$; this is understood to be part of an infinite word so that $w_i = \rsd{i}-i + w_{\rsd{i}}$, where $\rsd{\cdot} : \ZZ \to [n]$ is the map sending an integer $i$ to the integer in $[n]$ it is congruent to mod $n$. For example, if $n = 4$ and $w = \pi^2 s_2 s_0 s_1$, then the word of $w$ is $8\ 3\ 5\ 2$, thought of as part of the infinite word $\dots 12\ 7\ 9\ 6\ 8\ 3\ 5\ 2\ 4\ \ng 1\ 1\ \ng 2 \dots$ .

The following formulas relate multiplication of elements of $\eW$ with manipulations on words.  We adopt the convention of writing $a.b$ in place of $na+b$ ($a,b \in \ZZ$). In examples with actual numbers, $a$ and  $b$ will always be single-digit numbers and we will omit the dot.
{\footnotesize
\renewcommand{\minalignsep}{0pt}
\begin{flalign}
\nonumber \quad \text{\footnotesize{Element of }} \footnotesize{\eW} &&\text{\footnotesize{inverted window word}} &&\\
\label{e wordmult1} \idelm && n\ n-1 \cdots 2\ 1 &&\\
\label{e wordmult2} w &&x_1\ x_2 \cdots x_n&& \\
\label{e wordmult3} s_i w &&x_1\ x_2 \cdots x_{i+1}\ x_i \cdots x_n && i \in [n-1]\\
\label{e wordmult4} s_0 w &&1.x_n\ x_2 \cdots x_{n-1}\ (-1).x_1 &&\\
\label{e wordmult5} w s_{n-i} &&x_1 \cdots x_j + 1 \cdots x_k - 1 \cdots x_n && x_j \equiv i, x_k \equiv i+1, i \in [n] \\
\label{e wordmult6} y^\lambda w &&\lambda_1.x_1\ \lambda_2.x_2 \cdots \lambda_n.x_n && \\
\label{e wordmult7} \pi w &&1 . x_n\ x_1\ x_2 \cdots x_{n-1}  &&\\
\label{e wordmult8} w \pi &&x_1 + 1\ x_2 + 1\cdots x_n + 1
\end{flalign}}
Here are some basic facts we will need about words of $\eW$. See \cite{X2} for a thorough treatment.
\begin{proposition}\label{p affinewordlessthan}
For $w \in \eW$ and $s_i \in S$, $s_i w > w$ if and only if $w_i > w_{i+1}$. Similarly, $w s_{n-i} > w$ if and only if $j > k$, where $j$ and $k$ are such that $w_j = i, w_k = i+1$.
\end{proposition}

\begin{proposition}
\label{p affinewordlength}
For $w \in \eW$, the length of $w$ may be expressed in terms of its word by
\be
l(w) = \sum_{1 \leq i < j \leq n} \left|\left\lfloor\frac{w_i - w_j}{n}\right\rfloor\right|,
\ee
where $\lfloor x\rfloor$ is the greatest integer less than $x$.
\end{proposition}

\begin{proposition}
\label{p wordlj}
Given $w \in \eW$, let $x_1 x_2 \cdots x_n$ be the result of replacing the numbers of the word $w_1 w_2 \cdots w_n$ of $w$ by the numbers $1,\dots,n$ so that relative order is preserved. Then $x$ is the word of $\lj{w}{S}$ (the notation $\lj{w}{S}$ is defined in \textsection\ref{ss coxeter groups}).
\end{proposition}
\begin{proof}
Left-multiply $w$ by a sequence $s_{i_1} s_{i_2} \cdots s_{i_l}$, $i_j \in [n-1]$ until the resulting element $w'$ has word $w'_1 w'_2 \cdots w'_n$ such that $w'_1 > w'_2 > \cdots > w'_n$ and $\{w'_1, w'_2, \ldots, w'_n\} = \{w_1, w_2, \ldots, w_n\}$. This may be done so that each left-multiplication decreases length by 1. The same sequence of left-multiplications transforms $x_1 x_2 \cdots x_n$ into $\idelm = n\ n-1 \cdots 2\ 1$. By Proposition \ref{p affinewordlessthan}, $L(w') \subseteq \{s_0\}$. Therefore, $\rj{w}{S} = w'$ and $\lj{w}{S} = s_{i_1} s_{i_2} \cdots s_{i_l}$, and $s_{i_1} \cdots s_{i_l}$ has word $x_1 \cdots x_n$.
\end{proof}

Let $w^j$ be the subword of the word of $w$ in the alphabet $[jn+1,( j + 1) n]$ and $(w^j)^*$ denote the result of subtracting $jn$ from all the numbers in $w^j$.

\begin{proposition}
\label{p wordrj}
For $w \in \eW$, the word of $\rj{w}{S}$ is given by $w^0 (w^1)^* (w^2)^*\dots$. Equivalently, $\rj{w}{S}$ is given by $w_{j_1} w_{j_2} \cdots w_{j_n}$, where $j_1 < j_2 < \dots < j_n$ are such that $w_{j_i} \in [n]$.
\end{proposition}
\begin{proof}
The proof is essentially the same as that of Proposition \ref{p wordlj}, but right-multiplications on words are harder to deal with. By looking at the word of $w$ on the subword $w_{j_1} w_{j_2} \cdots w_{j_n}$ and using Proposition \ref{p affinewordlessthan}, we can see that the subword on the indices $j_1, j_2, \ldots, j_n$ can be transformed into $n\ n-1 \cdots 1$ by a sequence of right-multiplications by $s_i \in S$ that decrease length by 1. Then again by Proposition \ref{p affinewordlessthan}, the resulting word $w'$ satisfies $R(w') \subseteq \{s_0\}$, so $w' = \lJ{w}{S}$. Therefore, the sequence of right-multiplications gives a factorization of $\rj{w}{S}$ into a product of simple reflections, from which the result follows.
\end{proof}

\subsection{}
\label{ss invlr}
There is an automorphism $\Delta$ of $\eW$ given on generators by $s_i \mapsto s_{n-i}$, $\pi \mapsto \pi^{-1}$.
\begin{definition}
Let $\invlr: \eW \to \eW$ be the anti-automorphism defined by $\invlr(w) = \Delta(w^{-1}) = (\Delta(w))^{-1}$. This restricts to an anti-automorphism $\invlr: \pW \to \pW$. Finally, also denote by $\invlr$ the maps $\eH \to \eH$ and $\pH \to \pH$ given by $T_w \mapsto T_{\invlr(w)}$.
\end{definition}
The word of $\invlr(w)$ is given by $x_1 \cdots x_n$ where $x_i$ is determined by $w_{x_i} = i$. For example,
\[ \Psi(8\ 3\ 5\ 2) = \Psi(\pi^2 s_2 s_0 s_1) = s_3 s_0 s_2 \pi^2 = 7\ 4\ 2\ 5 \]

\subsection{}
The subset $\pY := \ZZ^n_{\geq 0}$ of the weight lattice $Y$ is left stable under the action of the Weyl group $W_f$. Thus $\pY \rtimes W_f$ is a submonoid of $\eW$. Note that this is only true in type $A$.

\begin{propdef}
\label{p pW}
The \emph{positive part} of $\eW$, denoted $\pW$, has the following three equivalent descriptions:
\begin{list}{\emph{(\arabic{ctr})}} {\usecounter{ctr} \setlength{\itemsep}{1pt} \setlength{\topsep}{2pt}}
\item $\pY \rtimes W_f$,
\item The submonoid of $\eW$ generated by $\pi$ and $W_f$,
\item $\{w \in \eW : w_i > 0 \text{ for all } i \in [n] \}$.
\end{list}
\end{propdef}
\begin{proof}
We will show $(1) \subseteq (2) \subseteq (3) \subseteq (1)$. As $y_i = s_{i-1} s_{i-2} \cdots s_1 y_1 s_1 s_2 \cdots s_{i-1}$ and $y_1 = \pi s_{n-1} s_{n-2} \cdots s_1$, $(1) \subseteq (2)$. The inclusion $(2) \subseteq (3)$ is clear from (\ref{e wordmult3}) and (\ref{e wordmult7}).

The word of any $w \in \eW$ can be written uniquely as
\[ {\small \lambda_1.x_1\ \ \lambda_2.x_2\ \cdots\ \lambda_n.x_n }\]
with $x_i \in [n]$ and $\lambda \in Y$. Then by (\ref{e wordmult6}), $w = y^\lambda v$ and $v$ has word $x_1 x_2 \cdots x_n$. Therefore $v \in W_f$. Then since $w_i > 0$ implies $\lambda_i \geq 0$, we have $(3) \subseteq (1)$.
\end{proof}

For $d \geq 0$, let $(\pY)_d$ (resp. $(\pY)_{\geq d}$) denote the set $\{\lambda \in \pY : |\lambda| = d\}$ (resp. $\{\lambda \in \pY : |\lambda| \geq d\}$). Define the \emph{degree $d$ part $(\pW)_d$ of $\pW$} to be any of the following
\be
\label{e degree d of pW}
\begin{array}{cl}
%{\emph{(\arabic{ctr})}} {\usecounter{ctr} \setlength{\itemsep}{1pt} \setlength{\topsep}{2pt}}
\text{(1$'$)}& (\pY)_d \rtimes W_f, \hfill\ \\
\text{(2$'$)}& \{ w \in \pW : w = \pi^d v, v \in \aW \}, \\
\text{(3$'$)}& \{ w \in \pW : \sum_{i=1}^n (w_i - i) = d n \}. \\
\end{array} \ee
The equality of these follows from the proof of Proposition-Definition \ref{p pW}, observing that if $y^\lambda v' = \pi^{d} v = w$, $v' \in W_f$, $v \in \aW$, then $|\lambda| = d = \frac{1}{n} \sum_{i=1}^n (w_i - i)$.

Define the \emph{degree} of a word $w \in \pW$, denoted $\deg(w)$, to be the $d$ for which $w \in (\pW)_d$, or equivalently, $\deg(w) = \frac{1}{n} \sum_{i=1}^n (w_i - i)$. The degree $d$ part of $\eW$ can be similarly defined and the definition of $\deg(w)$ also makes sense for $w \in \eW$.

\begin{lemma}
\label{l pivpi}
Any $w \in \pW$ has a reduced expression of the form $w = v_1\cdot\pi\cdot v_2\cdot\pi \ldots v_d\cdot\pi\cdot v_{d+1}$, where $v_i \in W_f$.
\end{lemma}
\begin{proof}
Use the description (3) of Proposition-Definition \ref{p pW}. By Proposition \ref{p affinewordlessthan}, one checks that any word of $w \in \eW$ with $w_i > 0$  of the form (3) can be brought to the identity by a sequence of left-multiplications by $\pi^{-1}$ and left-multiplications by $s_i \in S$ that decrease length by 1. This yields a desired reduced expression for $w$.
\end{proof}

\begin{propdef}
\label{p pH exists}
The subalgebra $\pH$ of $\eH$ has the following four equivalent descriptions:
\begin{list}{\emph{(\roman{ctr})}} {\usecounter{ctr} \setlength{\itemsep}{1pt} \setlength{\topsep}{2pt}}
\item $A\{\y^\lambda T_w : \lambda \in \pY, w\in W_f\}$,
\item $A\{T_w:w\in \pW\}$,
\item $A\{\C_w:w\in \pW\}$,
\item the subalgebra of $\eH$ generated by $\pi$ and $\H$.
\end{list}
\end{propdef}
\begin{proof}
As $\y_i = T_{i-1}^{-1} T_{i-2}^{-1} \cdots T_1^{-1} \y_1 T_1^{-1} T_2^{-1} \cdots T_{i-1}^{-1}$ and $\y_1 = \pi T_{n-1} T_{n-2} \cdots T_1$, (i) $\subseteq$ (iv). Then since $\pi \in$ (i) and $\H \subseteq$ (i),  (iv) $\subseteq$ (i) follows if we can show that (i) is a subalgebra. This can be seen from the relations (\ref{e TY}) since $\frac{\y^{\alpha'_i}(\y^\lambda - \y^{s_i(\lambda)})}{\y^{\alpha'_i}-1}$ is a polynomial in the $\y_i$ whenever $\lambda \in \ZZ^n_{\geq 0}$.

The inclusion (ii) $\subseteq$ (iv) follows from Lemma \ref{l pivpi}.  Again, showing that (ii) is a subalgebra will prove (iv) $\subseteq$ (ii). Given $w_1, w_2 \in \pW$,
\be \label{e w1w2} T_{w_1} T_{w_2} = \sum_{v_1 \leq w_1, v_2 \leq w_2} c_{v_1,v_2} T_{v_1v_2},\ \ c_{v_1,v_2} \in A. \ee
By Lemma \ref{l pivpi}, $w \in \pW$ implies $v \in \pW$ for any $v \leq w$. Also $v_1, v_2 \in \pW$ implies $v_1v_2 \in \pW$ as $\pW$ is a monoid. Thus the right-hand side of (\ref{e w1w2}) is in (ii). The equality (iv) = (iii) is similar to (iv) = (ii).
\end{proof}

The \emph{degree $d$ part} of $\pH$, $(\pH)_d$, has the corresponding descriptions:
\be
\label{e pH graded}
\begin{array}{cl}
\text{(i$'$)}& A\{\y^\lambda T_w : \lambda \in (\pY)_d, w\in W_f\}, \\
\text{(ii$'$)}& A\{T_w:w\in (\pW)_d\}, \\
\text{(iii$'$)}& A\{\C_w:w\in (\pW)_d\}.
\end{array} \ee

Also define $(\pH)_{\geq d} = \oplus_{i\geq d} (\pH)_d$ and $(\pH)_{\leq d} = \oplus_{i\leq d} (\pH)_d$.  The decomposition $\pH = (\pH)_0 \oplus (\pH)_1 \oplus \dots$ makes $\pH$ into a graded $A$-algebra. The descriptions (i), (ii), (iii) of Proposition-Definition \ref{p pH exists} give three $A$-bases for $\pH$ consisting of homogeneous elements.

Just as we write $\H(W)$ for the Hecke algebra of an extended Coxeter group $W$, generalizing the notion of a Hecke algebra of a Coxeter group, we further extend this to saying that $\pH$ is the Hecke algebra of the monoid $\pW$.

\subsection{} \label{ss cells of pH}
The left cells of $\Res_\H \eH, \pH$ can be determined by Proposition \ref{p restrict Wgraph}. These results are stated as the two corollaries below. Keep in mind our convention from \textsection \ref{ss words of eW} for the word of $w$.

The work of Kazhdan and Lusztig \cite{KL} shows that the left cells of $\H$ are in bijection with the set of SYT and the left cell containing $\C_w$ corresponds to the insertion tableau of $w$ under this bijection. The left cell containing those $\C_w$ such that $w$ has insertion tableau $P$ is the left cell \emph{labeled} by $P$, denoted $\Gamma_P$. A combinatorial discussion of left cells in type $A$ is given in \cite[\textsection 4]{B0}.

\begin{definition}
An \emph{affine tableau} (AT) of size $n$ is a semistandard Young tableau filled with integer entries that have distinct residues mod $n$. A \emph{positive affine tableau} (PAT) of size $n$ is a semistandard Young tableau filled with positive integer entries that have distinct residues mod $n$.
\end{definition}

For $w \in \eW$, the word $w_1 w_2 \cdots w_n$ may be inserted into a tableau, and the result is an affine tableau, denoted $P(w)$ (see \textsection\ref{ss tableau basics} for our tableau conventions). It is a positive affine tableau exactly when  $w \in \pW$. By Proposition \ref{p wordlj}, the SYT $P(\lj{w}{S})$ is obtained from $P(w)$ by replacing its entries with the numbers $1, \dots, n$ so that the relative order of entries in $P(w)$ and $P(\lj{w}{S})$ agree. Since  $P(\lj{w}{S})$ is determined by the tableau $Q := P(w)$, independent of the chosen $w$ inserting to  $Q$, we write  $\lj{Q}{S}$ for this tableau.  For example, for the given $w$ below, $\lj{w}{S}$, $P(w)$, and $P(\lj{w}{S}) = \lj{P(w)}{S}$ are as follows.

% hoogte = height
% breedte = width
% dikte = linewidth
\hoogte=10pt
\breedte=12pt
\dikte=0.2pt
\Yboxdim12pt
\newcommand{\duma}{4}
\newcommand{\dumb}{13}
\newcommand{\dumc}{15}
\newcommand{\dumd}{12}
\newcommand{\dume}{16}
\newcommand{\dumf}{21}

\[ \begin{array}{rcccccccc}
w = &21&12&13&16&4&15, & P(w) = {\footnotesize \young(\duma\dumb\dumc,\dumd\dume,\dumf)}\\
\lj{w}{S} = &6&2&3&5&1&4, & P(\lj{w}{S}) = {\footnotesize \young(134,25,6)}
\end{array} \]

Define the \emph{degree} of an affine tableau $Q$, denoted $\deg(Q)$, to be $\deg(w)$ for any (every) $w$ inserting to $Q$.

Let $Q$ be an affine tableau. The set of $w \in \eW$ inserting to $Q$ is  \[ \{v x : v \in W_f, P(v) = \lj{Q}{S}\}, \] where  the word of $x$ is obtained from  $Q$ by sorting its entries in decreasing order.  For any $x \in \leftexp{S}{\eW}$, define
\be \Gamma_Q := \{ \C_{vx} : v \in W_f,\ P(v) = \lj{Q}{S} \} = \{ \C_w : w \in \eW,\ P(w) = Q \}. \ee
By the following result, $\Gamma_Q$ is a left cell of $\Res_\H A \Gamma_{\eW}$, which we refer to as the left cell \emph{labeled by} $Q$.

\begin{corollary}
\label{c restrict cells}
For any $x \in \leftexp{S}{\eW}$, the set $\{ \C_{wx} : w\in W_f \}$ is a cellular subquotient of $\Res_\H \eH$, isomorphic as a $W_f$-graph to $\Gamma_{W_f}$. In particular,
\[ \Gamma_{\eW} = \bigsqcup_{Q \in AT}\Gamma_Q  \]
is the decomposition of $\Res_\H \eH$ into left cells.
\end{corollary}

Note that the definition (\ref{e preorder}) for the preorder $\klo{E}$ works just as well for any module with a distinguished basis. Write $\klo{\pH}$ for the preorder on the canonical basis of $\pH$ coming from considering $\pH$ as a left $\pH$-module. We also refer to $\pH$ as a $\pW$-graph and say that $\klo{\pH}$ is the preorder on the $\pW$-graph $\pH$. The ascent-edges of \textsection\ref{ss easyedges} have their obvious meaning as certain relations in $\klo{\pH}$. Similar remarks apply to the partial order on the left cells of $\pH$, also denoted $\klo{\pH}$. We refer to the relation $\C_{\pi w} \klo{\pH} \C_w$ as a \emph{corotation-edge} and the corresponding edge between cells as a \emph{cocyclage-edge}. We will soon see that cocyclage-edges are a generalization of cocyclage for standard Young tableaux. Also define a \emph{rotation-edge} to be left-multiplication by $\pi^{-1}$, which is a relation in $\klo{\eH}$ but not $\klo{\pH}$.

\begin{proposition}
\label{p preorder pH}
The preorder $\klo{\pH}$ is the transitive closure of the relation $\klo{\Res_\H \pH}$ and corotation-edges.
\end{proposition}
\begin{proof}
This is clear from the description Proposition-Definition \ref{p pH exists} (iv) of $\pH$.
\end{proof}

\begin{corollary}
\label{c restrict cells2}
For any $x \in \leftexp{S}{\pW}$, the set $\{ \C_{wx} : w\in W_f \}$ is a cellular subquotient of the $\pW$-graph $\pH$. This subquotient, restricted to be a $W_f$-graph, is isomorphic to the $W_f$-graph $\Gamma_{W_f}$. In particular,
\[ \Gamma_{\pW} = \bigsqcup_{Q \in PAT} \Gamma_Q \]
is the decomposition of $\pH$ into left cells.
\end{corollary}
\begin{proof}
As is evident from (\ref{e pH graded}) (iii$'$), the submodule $(\pH)_{\geq d}$ of $\pH$ is cellular. Since corotation-edges increase degree by 1, Proposition \ref{p preorder pH} implies that the preorder for the cellular subquotient $(\pH)_{\geq d} / (\pH)_{\geq d+1}$ is the same as that of $\Res_\H (\pH)_{\geq d} / (\pH)_{\geq d+1}$. Thus since
$\{ \C_{wx} : w\in W_f \} \subseteq (\pH)_d$ for some $d$, the result is a special case of Corollary \ref{c restrict cells}.
\end{proof}

\begin{remark}
For the purposes of this paper, it makes little difference whether we work with $\pH$ as an $\pH$-module or  $\eH$ as an $\pH$-module. Any finite dimensional cellular subquotient of the $\pH$-module  $\eH$ is isomorphic to a cellular subquotient of $\pH$, the isomorphism given by multiplication by a suitable power of $\pi^n$. In this paper, we will almost always work with $\eH$ as an $\pH$-module, referring to the corresponding cells as  $\pW$-cells. In section \ref{s atoms}, we will briefly look at the cells of $\eH$ for the action of $\eH$ (i.e. the cells of $\Gamma_\eW$ as a $\eW$-graph) as worked out by Lusztig, Shi, and Xi (see \textsection\ref{ss cells in We}), and will refer to these cells as $\eW$-cells. We define $\pH$-cellular subquotients (resp. $\eH$-cellular subquotients) of $\eH$ to be cellular subquotients of $\eH$ for the $\pH$ action (resp. $\eH$ action). Cellular subquotients of $\eH$ will by default mean $\pH$-cellular subquotients.
\end{remark}

\section{Cocyclage, catabolism, and atoms}
\label{s cocyclage catabolism atoms}
Before going deeper into the study of the canonical basis of $\pH$, we need some intricate tableau combinatorics which will be used to describe cellular subquotients of $\pH$.  In this section we discuss cocyclage, define a variation of catabolizability for affine tableaux, and introduce a formalism for comparing cellular subquotients of  $\pH$ to certain subsets of tableaux defined in \cite{LLM} and \cite{SW}.  Such subsets of tableaux were referred to as super atoms in \cite{LLM}; here we refer to these subsets and their variations as atoms.  This section is long and heavy in definitions, so the reader may wish to skim it and refer back to it as needed; the material here is used most extensively in \textsection \ref{s atoms}.

\subsection{}
\label{ss tableau basics}
Let $\Theta, \nu$ be partitions with  $\nu \subseteq \Theta$.
The \emph{diagram} of a (skew) shape $\theta = \Theta / \nu$ is the subset
\[ \{(r,c) \in \ZZ_{\geq 1} \times \ZZ_{\geq 1}: c \in [\nu_r + 1, \Theta_r]\} \]
of the array $\ZZ_{\geq 1} \times \ZZ_{\geq 1}$. Diagrams are drawn in English notation so that rows (resp. columns) are labeled starting with 1 and increasing from north to south (resp. west to east). We often refer to the diagram of $\theta$ simply by $\theta$.

The conjugate partition $\lambda'$ of a partition $\lambda$ is the partition whose diagram is the transpose of that of $\lambda$.

A tableau $T$ of shape $\lambda$ is a filling of $\lambda$ with entries in $\ZZ$ so that entries strictly increase from north to south along columns and weakly increase from west to east along rows. We write $\sh(T)$ for the shape of $T$.
\subsection{}
\label{ss cocyclage poset}
Let us review the definitions of cocyclage poset and related combinatorics originating in \cite{La, LS} (see also \cite{SW}).

The \emph{cocharge labeling} of a word $v$, denoted $\cl{v}$, is a (non-standard) word of the same length as $v$, and its numbers are thought of as labels of the numbers of $v$. It is obtained from $v$ by reading the numbers of $v$ in increasing order, labeling the 1 of $v$ with a 0, and if the $i$ of $v$ is labeled by $k$, then labeling the $i+1$ of $v$ with a $k$ (resp. $k+1$) if the $i+1$ in $v$ appears to the right (resp. left) of the $i$ in $v$. For example, the cocharge labeling of $614352$ is $302120$; also see Example \ref{ex cocharge}.

Write $\reading(T)$ and $\creading(T)$ for the row and column reading words of a tableau $T$.  Define the cocharge labeling $\cl{T}$ of a tableau $T$ to be $P(\cl{\reading(T)})$, where numbers are inserted as for semistandard tableaux -- if two numbers are the same, then the one on the right is considered slightly bigger. The tableau $\cl{T}$ is also $P(\cl{w})$ for any $w$ inserting to $T$. This follows from the fact that Knuth transformations do not change left descent sets.

The sum of the numbers in the cocharge labeling of a standard word $v$ (resp. standard tableau $T$) is the \emph{cocharge} of $v$ (resp.  $T$) or $\ccharge(v)$ (resp. $\ccharge(T)$). Cocharge of semistandard words and tableaux are more subtle notions, which we do not define in the usual way here. We will come across another way of understanding this statistic in \textsection\ref{ss SSYT in PAT}.

For a composition $\eta$ of $n$, let  $\word(\eta)$ and $\Tab(\eta)$ be the sets of semistandard words and semistandard tableaux of content $\eta$, respectively.

For a semistandard word $w$ and number $a \neq 1$, $aw$ (resp. $wa$) is a \emph{corotation} (resp. \emph{rotation}) of $wa$ (resp. of $aw$). There is a \emph{cocyclage} from the tableau $T$ to the tableau $T'$, written $T \cc T'$, if there exist words $u, v$ such that $v$ is the corotation of $u$ and $P(u) = T$ and $P(v) = T'$. Rephrasing this condition solely in terms of tableaux, $T \cc T'$ if there exists a corner square $(r,c)$ of $T$ and uninserting the square $(r,c)$ from $T$ yields a tableau $Q$ and number $a$ such that $T'$ is the result of column-inserting $a$ into $Q$.

If $\eta$ is a partition, then the \emph{cocyclage poset} $\ccp(\Tab(\eta))$ is the poset on the set $\Tab(\eta)$ generated by the relation \cc.  For $\eta$ not a partition, the cocyclage poset $\ccp(\Tab(\eta))$ is defined in terms of  $\ccp(\Tab(\eta_+))$ using reflection operators (see \cite{SW}), where $\eta_+$ denotes the partition obtained from $\eta$ by sorting its parts in decreasing order.
The \emph{cyclage poset} on  $\Tab(\eta)$ is the dual of the poset $\ccp(\Tab(\eta))$, i.e. the poset obtained by reversing all relations. With our convention from Remark \ref{r edge direction}, we have
%?? (is there a way to define this without reflection operators?)

\begin{theorem}[\cite{LS}]The cyclage poset on $\Tab(\eta)$ is graded, with rank function given by cocharge.
\end{theorem}

Similarly, define $\ccp(PAT)$ (resp. $\ccp(AT)$) to be the poset on the set of PAT (resp. AT) generated by cocyclage-edges (see \textsection\ref{ss cells of pH}). The poset $\ccp(PAT)$ inherits a grading from that of $\pW$ (see (\ref{e degree d of pW})). The poset $\ccp(AT)$ also inherits a grading from that of $\eW$.

The covering relations of $\ccp(\Tab(\eta))$ (resp. $\ccp(PAT)$ or $\ccp(AT)$) are exactly cocyclages (resp. cocyclage-edges). We consider the covering relation $T \cc T'$ to be colored by the following additional datum: the set of outer corners of $T$ that result in a cocyclage to $T'$. Note that this set can only have more than one element if $\sh(T) = \sh(T')$.

In preparation for the formalism of \textsection\ref{ss atom categories}, we define the category \emph{Cocyclage Posets (CCP)} as follows.
\begin{definition}
\label{d cocyclage posets}
An object of Cocyclage Posets, called a cocyclage poset (ccp), is allowed to be either of the following:
\begin{itemize}
\item A subset $X$ of $\Tab(\eta)$ with a poset structure generated by the cocyclages with both tableaux in $X$.
\item A subset $X$ of AT with the poset structure generated by the cocyclage-edges with both ends in $X$.
\end{itemize}
A morphism  $f$ from $X_1$ to $X_2$ is a color-preserving map (that is, if $T \cc T'$ with $T, T' \in X_1$ then $f(T) \cc f(T')$ and these relations have the same color) from $X_1 \cup \{0\}$ to $X_2 \cup \{0\}$ such that $\sh(f(T)) = \sh(T)$ for all  $T \in X_1$ and  $f(0) = 0$, where $0$ is the bottom element of $X_i \cup \{0\}$.  We take the convention that for each minimal element $T$ of  $X_i$ and outer corner $(r,c)$ of  $\sh(T)$, there is a cocyclage from $T$ to  $0$ with color  $(r,c)$, and  $0$ is considered to have any shape.  Thus for a minimal $T \in X_1$,  $f(T) =0$ or $f(T)$ is minimal in $X_2$.
\end{definition}
Note that with this definition, a morphism  $f: X_1 \to X_2$ is automatically order preserving, i.e.  $T \leq T'$ implies $f(T) \leq f(T')$.

\begin{definition}
\label{d strongly isomorphic}
Two cocyclage posets $X_1,X_2$ are \emph{strongly isomorphic} if there exists an isomorphism $f:X_1 \to X_2$ in Cocyclage Posets such that the uninsertion path and insertion path corresponding to the cocyclage $T \cc T'$ are the same as those for $f(T) \cc f(T')$, for all $T \cc T'$ in  $X_1$.
\end{definition}

See Example \ref{ex ssyt in PAT} for an example of three isomorphic cocyclage posets, two of which are strongly isomorphic to each other, but not to the third.
\subsection{}
\label{ss catabolizability definition}
Here we consider an adaptation of catabolizability to affine tableaux.

For a tableau $T$ and index $r$ (resp. index $c$), let $T_{r,\text{north}}$ and $T_{r, \text{south}}$ (resp. $T_{c,\text{east}}$ and $T_{c,\text{west}}$)  be the north and south (resp. east and west) subtableaux obtained by slicing $T$ horizontally (resp. vertically) between its $r$-th and $(r+1)$-st rows (resp. $c$-th and $(c+1)$-st columns). For a tableau $T$ and partition $\lambda \subseteq \sh(T)$, let $T_\lambda$ be the subtableau of $T$ obtained by restricting  $T$ to the diagram of $\lambda$. For a tableau $T$ and $a \in \ZZ$, let $a+T$ denote the tableau obtained by adding $a$ to all entries of $T$.

Let $Q$ be a tableau of shape $\lambda = (\lambda_1,\dots,\lambda_r)$ and $\eta = (\eta_1, \dots, \eta_k)$ a composition of $r$. Let $R_1$ be the partition $(\lambda_1, \dots, \lambda_{\eta_1})$. If $R_1 \subseteq \sh(T)$, then define the $R_1$\emph{-row catabolism} of $T$, notated $\cat_{R_1}(T)$, to be \[(n+T^*_{\eta_1,\text{north}})T^*_{\eta_1,\text{south}},\] where $T^*$ is the skew subtableau of $T$ obtained by removing $T_{R_1}$.

For $Q, \eta, R$ as above, $(Q,\eta)$\emph{-row catabolizability} is defined inductively as follows:  $\emptyset$ is the unique $(\emptyset,())$-row catabolizable tableau; otherwise set $\eta = (\eta_1, \widehat{\eta})$ and define $T$ to be $(Q,\eta)$-row catabolizable if $T_{R_1} = Q_{R_1}$ and $\cat_{R_1}(T)$ is $(Q_{\eta_1,south},\widehat{\eta})$-row catabolizable.

Column catabolizability is defined similarly: let $Q$ be a tableau of shape $\lambda$ and $\lambda' = (\lambda'_1, \dots, \lambda'_c)$ and $\eta = (\eta_1, \dots, \eta_k)$ a composition of $c$. Let $C_1$ be the partition $(\lambda'_1, \dots, \lambda'_{\eta_1})'$. If $C_1 \subseteq \sh(T)$, the $C_1$\emph{-column catabolism} of $T$, notated $\ccat_{C_1}(T)$, is the tableau \[T^*_{\eta_1, \text{east}} (\ng n + T^*_{\eta_1, \text{west}}),\] where $T^*$ is the skew subtableau of $T$ obtained by removing $T_{C_1}$.

For $Q, \eta, R$ as above, $(Q,\eta)$\emph{-column catabolizability} is defined inductively as follows: $\emptyset$ is the unique $(\emptyset,())$-column catabolizable tableau; otherwise set $\eta = (\eta_1, \widehat{\eta})$ and define $T$ to be $(Q,\eta)$-column catabolizable if $T_{C_1} = Q_{C_1}$ and $\ccat_{C_1}(T)$ is $(Q_{\eta_1,east}, \widehat{\eta})$-column catabolizable.

If $(n+T^*_{\eta_1,\text{north}})$ is replaced by $T^*_{\eta_1,\text{north}}$ in the definition of row-catabolizability above and $Q$ is a superstandard tableau, then we recover the definition of catabolizability in \cite{SW}.

%Define the $(Q,\eta)$-Shimozono-Weyman atom copy $SW_{Q,\eta}$ to be the set of $(Q,\eta)$-row catabolizable tableaux.
We will see in \textsection\ref{ss sw atoms} that for certain $Q$ of shape $\lambda$, the set of $(Q,\eta)$-column catabolizable tableaux is strongly isomorphic to a dual version of the original set of tableaux $CT(\lambda; R)$ defined in \cite{SW}.

\begin{example}
Let $A, B, C, D, E$ denote the integers 10, 11, 12, 13, 14, and maintain the convention of \textsection\ref{ss words of eW} of writing $ab$ for $na + b$, $a,b \in \ZZ$. Let
\setlength{\cellsize}{13pt}
\be
Q = {\tiny \tableau{1&2&3\\14&15&16\\27&28&29\\3A&3B&3C\\4D&4E}} \ \  \text{ and } \ \ T = {\tiny \tableau{1&2&3&17&2B\\14&15&16&2A\\28&29&3E\\3C&3D}}.  \ee
The tableau $T$ is $(Q,\eta)$-row catabolizable for the following $\eta$: $(2, 2, 1)$, $(2,1,2)$, $(1,2,2)$ and all refinements of these compositions. The following computation shows $T$ to be $(Q,(1,2,2))$-row catabolizable:
We have $T_{(3)} = Q_{(3)}$ and the tableaux $n+T^*_{1,\text{north}}$ and $T^*_{1,\text{south}}$ are the left-hand side of
\be {\tiny \tableau{27&3B}\quad\tableau{14&15&16&2A\\28&29&3E\\3C&3D}} \equiv {\tiny \tableau{14&15&16&2A&2E\\27&28&29&{3D}\\{3B}&{3C}}}. \ee
Letting  $P$ be the tableau on the right, then  $P$ is $(Q_{1,\text{south}}, (2,2))$-row catabolizable as $P_{(3,3)} = (Q_{1,\text{south}})_{(3,3)}$ and the computation
\be {\tiny \tableau{{3A}&{4E}\\{4D}}\quad \tableau{{3B}&{3C}}} \equiv {\tiny \tableau{{3A}&{3B}&{3C}\\{4D}&{4E}}}\ . \ee

The tableau on the right is $(Q_{3,\text{south}},(2))$-row catabolizable.

Figure \ref{f column catabolizable} depicts the set of $(Q,(1,2,1))$-column catabolizable tableaux for $Q$ the tableau in the top row of the figure. The $(Q,(3,1))$-column catabolizable tableaux are $Q$ and the first two tableaux in the second row. The $(Q, (1,3))$-column catabolizable tableaux are $Q$, the last two tableaux on the second row, and the last tableau on the third row.

\begin{figure}
\begin{pspicture}(0pt,0pt)(500pt,250pt){\tiny
\hoogte=9pt
\breedte=9pt
\dikte=0.2pt

\newdimen\horizcent
\newdimen\ycor
\newdimen\xcor
\newdimen\horizspace
\newdimen\temp

\horizcent=175pt
\ycor=240pt
\advance\ycor by 0pt
\horizspace=50pt
\xcor=\horizcent
\temp=50pt
\multiply \temp by 0 \divide \temp by 2
\advance\xcor by -\temp
\rput(\xcor,\ycor){\rnode{v36h1}{\begin{Young}
1&4&6&8\cr
2&5&7\cr
3\cr
\end{Young}}}

\advance\xcor by \horizspace
\advance\ycor by -55pt
\horizspace=50pt
\xcor=\horizcent
\temp=50pt
\multiply \temp by 2\divide \temp by 2
\advance\xcor by -\temp
\rput(\xcor,\ycor){\rnode{v46h2}{\begin{Young}
1&4&6\cr
2&5&7\cr
3&18\cr
\end{Young}}}
\advance\xcor by \horizspace
\rput(\xcor,\ycor){\rnode{v46h3}{\begin{Young}
1&4&6\cr
2&5&7\cr
3\cr
18\cr
\end{Young}}}
\advance\xcor by \horizspace
%\rput(\xcor,\ycor){\rnode{v46h4}{}}
%\advance\xcor by \horizspace
\rput(\xcor,\ycor){\rnode{v46h5}{\begin{Young}
1&4&7&8\cr
2&5\cr
3\cr
16\cr
\end{Young}}}
\advance\xcor by \horizspace
\advance\ycor by -55pt
\horizspace=50pt
\xcor=\horizcent
\temp=50pt
\multiply \temp by 2 \divide \temp by 2
\advance\xcor by -\temp
%\rput(\xcor,\ycor){\rnode{v56h6}{}}
%\advance\xcor by \horizspace
\rput(\xcor,\ycor){\rnode{v56h7}{\begin{Young}
1&4&7\cr
2&5&18\cr
3\cr
16\cr
\end{Young}}}
\advance\xcor by \horizspace
\rput(\xcor,\ycor){\rnode{v56h8}{\begin{Young}
1&4&7\cr
2&5\cr
3&18\cr
16\cr
\end{Young}}}
\advance\xcor by \horizspace
\rput(\xcor,\ycor){\rnode{v56h9}{\begin{Young}
1&4&7\cr
2&5\cr
3\cr
16\cr
18\cr
\end{Young}}}
\advance\xcor by \horizspace
\advance\ycor by -55pt
\horizspace=50pt
\xcor=\horizcent
\temp=50pt
\multiply \temp by 1 \divide \temp by 2
\advance\xcor by -\temp
\rput(\xcor,\ycor){\rnode{v66h14}{\begin{Young}
1&4&18\cr
2&5\cr
3\cr
16\cr
17\cr
\end{Young}}}
\advance\xcor by \horizspace
\rput(\xcor,\ycor){\rnode{v66h15}{\begin{Young}
1&4\cr
2&5\cr
3&18\cr
16\cr
17\cr
\end{Young}}}
\advance\xcor by \horizspace
\rput(\xcor,\ycor){\rnode{v66h16}{}}
\advance\xcor by \horizspace
\rput(\xcor,\ycor){\rnode{v66h17}{}}
\advance\xcor by \horizspace
\rput(\xcor,\ycor){\rnode{v66h18}{}}
\advance\xcor by \horizspace
\rput(\xcor,\ycor){\rnode{v66h19}{}}
\advance\xcor by \horizspace
\advance\ycor by -55pt
\horizspace=50pt
\xcor=\horizcent
\temp=50pt
\multiply \temp by 0 \divide \temp by 2
\advance\xcor by -\temp

\rput(\xcor,\ycor){\rnode{v76h21}{\begin{Young}
1&4\cr
2&5\cr
3\cr
16\cr
17\cr
28\cr
\end{Young}}}
}
\end{pspicture}
\caption{The  $(Q,(1,2,1))$-column catabolizable tableaux with  $Q$ the tableau on the top row.}
\label{f column catabolizable}
\end{figure}
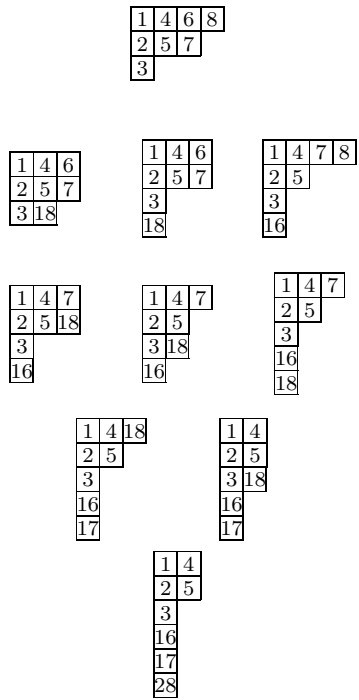
\end{example}

\subsection{}
\label{ss catabolizability 2}
%As is standard, define $\rho = \frac{1}{2}\sum_{\alpha \in R_{f+}} \alpha$.
For this subsection, we will use Propositions \ref{p bijectionwfdsw} and \ref{p ccp syt equals ccp cinv}, which show that $\ccp(SYT)$ is isomorphic to a sub-cocyclage poset of $\ccp(PAT)$. The embedding is given on the level of words by $w \mapsto n \cl{w} + w$, where the sum is taken entry-wise.

Let $Z^*_\lambda$ be the standard tableau with $l_r+1,l_r+2, \dots, l_{r+1}$ in the $r$-th row, where $l_r = \sum_{i=1}^{r-1} \lambda_i$ are the partial sums of $\lambda$ (the empty sum is understood to be 0).  The tableau $G_\lambda$ mentioned in the introduction is $n \cl{Z^*_\lambda} + Z^*_\lambda$, where the sum is taken entry-wise (see Proposition \ref{p ccp syt equals ccp cinv}). The \emph{dual Garnir tableau of shape} $\lambda$ is the highest degree occurrence of a PAT of shape $\lambda$ in $\cinv$, denoted $\dual{G}_\lambda$.

Let us briefly introduce a certain duality in $\cinv$, which will be discussed more thoroughly in \textsection\ref{s flip}. For a standard word $x = x_1\cdots x_n$, let $x^\dagger$ denote the word $x_n x_{n-1} \dots x_1$. Then for any $w$ with $\C_w \in \cinv$, $w = n \cl{x} + x$ for some standard word $x$ (see \textsection\ref{ss factorization theorem}). Define the \emph{dual element} $\dual{w}$ by $\dual{w} = n \cl{(x^\dagger)} + x^\dagger$. Extend this notation to tableaux by defining $\dual{T}$ to be $P(\dual{w})$ for any (every) $w$ inserting to $T$.

Note that $G_{(n)} = \dual{G}_{(n)}$ is the single row tableau
\renewcommand{\duma}{\ \cdots}
${\footnotesize \Yboxdim14pt \young(12\duma n)}$
and $G_{1^n} = \dual{G}_{1^n}$  is the single column tableau with the entry $(r-1).r$ in the $r$-th row.

\begin{proposition}
\label{p catabolizable basics}
The following are equivalent for a tableau $T$ in the image of the embedding $\ccp(SYT) \hookrightarrow \ccp(PAT)$ of Proposition \ref{p ccp syt equals ccp cinv}:
\begin{list}{\emph{(\alph{ctr})}} {\usecounter{ctr} \setlength{\itemsep}{1pt} \setlength{\topsep}{2pt}}
\item $T$ is $(G_{(n)},\lambda)$-column catabolizable,
\item $T$ is $(G_\lambda,1^{\ell(\lambda)})$-row catabolizable,
\item $\dual{T}$ is $(\dual{G}_{1^n}, \lambda)$-row catabolizable,
\item $\dual{T}$ is $(\dual{G}_{\lambda'}, 1^{\ell(\lambda)})$-column catabolizable.
\end{list}
\end{proposition}
\begin{proof}
In view of Proposition \ref{p ccp syt equals ccp cinv}, the equivalence of (a) and (b) is the equivalence of row and column catabolizability established in \cite{SW} (see \cite{B1} for a nice proof). Since the SYT corresponding to $\dual{T}$ is just the transpose of the SYT corresponding to $T$, it is easy to see that (a) and (c) are equivalent and (b) and (d) are equivalent.
\end{proof}

From a well-known result about catabolizability of standard tableaux, a tableau $T$ labeling a left cell of $\cinv$ is $(G_{(n)},\lambda)$-column catabolizable for a unique maximal in dominance order partition $\lambda$.  We write $\ctype(T)$ for this partition and also use this notation for the usual notion of catabolizability if  $T$ is a standard tableaux (see \cite{SW}). These definitions of $\ctype$ coincide under the embedding of Proposition \ref{p ccp syt equals ccp cinv}.

\subsection{}
\label{ss atom categories}
At the risk of being overly formal, we will define several categories which are generalizations or variations of the cocyclage posets of Lascoux and Sch\"utzenberger and the super atoms of Lascoux, Lapointe, and Morse \cite{LLM}.  We will primarily be concerned with the underlying sets of objects of these categories and isomorphism in these categories.

For a ring $k$ and $k$-algebra $H$, the category of $H$-modules with basis has objects that are pairs $(E,\Gamma)$, where $E$ is a free $k$-module and an $H$-module (the action of $H$ extends that of $k$) with $k$-basis $\Gamma$. A morphism $(E,\Gamma) \to (E',\Gamma')$ is an $H$-module morphism $\theta: E \to E'$ such that $\theta(\gamma) \in \Gamma' \cup \{0\}$ for all $\gamma \in \Gamma$.

Let $P$ and $P'$ be AT and $\Gamma_{P}, \Gamma_{P'}$ the corresponding left  $\pW$-cells of $\Gamma_{\eW}$. By Proposition \ref{p restrict Wgraph}  together with the facts that a left cell of $\Gamma_{W_f}$ is irreducible at $\u = 1$ and the left cells corresponding to the same shape are isomorphic as $W_f$-graphs \cite[Theorem 1.4]{KL}, we have

\refstepcounter{equation} \label{e cell isomorphisms}
\begin{enumerate}[label={(\theequation.\roman{*})}]
\item If $\sh(P) = \sh(P')$, then there are exactly two $\pH$-morphisms from $A \Gamma_P$ to $A \Gamma_{P'}$: the 0 map and the map taking $C'_w$ to $C'_{w'}$ for $w\xrightarrow{RSK} (P,Q), w'\xrightarrow{RSK} (P',Q)$ for all SYT $Q$ of shape $\sh(P)$.
\item If $\sh(P) \neq \sh(P')$, then the 0 map is the only $\pH$-morphism from $A \Gamma_P$ to $A \Gamma_{P'}$.
\end{enumerate}

The following categories will be denoted by the plural form of an object in the category, i.e., a cocyclage poset is an object in the category Cocyclage Posets.  We refer to these categories as Atom Categories and their objects as atoms.
\begin{itemize}
\setlength{\itemsep}{1pt}\setlength{\topsep}{2pt}
\item \emph{$\pH$-Cellular Subquotients of} $\eH$ (\cs): The full subcategory of $\pH$-modules with basis whose objects are $\pH$-cellular subquotients of $\eH$ with the canonical basis.
\item \emph{Convex Cocyclage Posets (XCCP)}: Let $\pPi$ be the submonoid $\Pi \cap \pW = \langle \pi \rangle$ of $\pW$, where $\Pi$ is as in \textsection\ref{ss two presentations of eH}. Write $A \pPi$ for the corresponding subalgebra of $\pH$. A convex cocyclage poset (xccp) is a union $E$ of left $\pW$-cells of $\eH$ such that $\Res_{A\pPi} E$ is an $\pH$-cellular subquotient of $\Res_{A\pPi} \eH$. A morphism $\alpha: E \to E'$ is a morphism in the category of $A\pPi$-modules with basis (the basis for an object being the canonical basis) such that the composition $A\Gamma \hookrightarrow E \xrightarrow{\alpha} E' \twoheadrightarrow A\Gamma'$ is of the form (\ref{e cell isomorphisms}.i) or (\ref{e cell isomorphisms}.ii) for $\Gamma, \Gamma'$ left cells of $E, E'$.
    Equivalently, a convex cocyclage poset is a convex induced subposet of $\ccp(AT)$.  A morphism is the same as a morphism in Cocyclage Posets (Definition \ref{d cocyclage posets}).
\item Cocyclage Posets (CCP) as in Definition \ref{d cocyclage posets}.
\item $R\star W_f$-\Mod: objects are $R\star W_f$ modules equipped with a grading compatible with that of $R\star W_f$.
\item $\CC[t] \tsr \Lambda$: objects are symmetric functions with coefficients in $\CC[t]$. There is a unique morphism between each ordered pair of objects.
\end{itemize}

The remainder of our list consists of certain full subcategories of these categories. Before defining these, we establish some basic properties of the above categories. We have the following diagram of functors:
\be \label{e diagram atom categories}
{\footnotesize \xymatrix@R=.7cm@C=1.6cm{
{\cs}\ar[rr]_{\fxccp}\ar[d]_{\fmod}\ar[rrd]_{\fsp} & & \text{Convex Cocyclage Posets} \ar[d]\\
R\star W_f\text{-}\Mod \ar[rd]_{\mathscr{F}} & & \text{Cocyclage Posets} \ar[ld]_{\mathscr{F}}\\
& \CC[t] \tsr \Lambda &
}}
\ee

%\be
%{\footnotesize \xymatrix@R=.25cm@C=1cm{ \text{GP atoms} \ar@{^{(}->}[r]\ar@{^{(}->}[rdd] & \text{SW atoms} \ar[rdd]\ar[rr] &  \ & \text{Cocyclage Posets}  \\
%\ & \ & \ & \  \\
%\text{dual GP atoms} \ar@{^{(}->}[ruu]\ar@{^{(}->}[r] & \text{Chen atoms} \ar[r]\ar[rruu] & {\cs} \ar[r]_{\fsp}\ar[ruu]_{\fxccp}\ar[rd]_{\fmod}& \text{Convex Cocyclage Posets} \ar[uu]_{\fsp}\\
%\text{LLM atoms} \ar@{^{(}->}[ru]& \ & \ & R\star W_f\text{-}\Mod}}
%\ee

The functor $\fxccp$ just restricts an $\pH$-module with basis to an $A\pPi$-module with basis.
The functors $\fsp$ and and the vertical arrow on the right just forget about  $\pH$ or  $A \pPi$-module structures and retain the underlying set of tableaux corresponding to left cell labels;
the poset structure on this set is defined to be that generated by the cocyclage-edges with both ends in the set (see \textsection\ref{ss cells of pH}).
The functor $\fmod$ takes $E$ to $\CC\tsr_A E$ and forgets about the canonical basis, where $A \to \CC$ is given by $\u \mapsto 1$. Thus  $\cs$ connects the algebraic $R \star W_f$-$\Mod$ and combinatorial Convex Cocyclage Posets in the sense that a cellular subquotient of $\eH$ gives rise to both an $R \star W_f$-module and a convex cocyclage poset. The functor $\mathscr{F}$ on the left takes a module to its Frobenius series, with the exponent of $t$ keeping track of the grading. The functor $\mathscr{F}$ on the right takes a ccp $X$ to $\sum_{T \in X} t^{\deg(T)} s_\sh(T)$.

We record the fact, immediate from Proposition \ref{p preorder pH}, that
\begin{proposition}
\label{p preorder cells vs cocyclage poset}
For any $E \in \cs$, the partial order $\klo{E}$ on cells is the transitive closure of $\klo{\Res_\H E}$ and cocyclage-edges.
\end{proposition}

\begin{definition}
A cocyclage poset or convex cocyclage poset is \emph{connected} if its poset is connected as an undirected graph.
\end{definition}

\begin{definition}
For $Q,P \in \text{AT}$, the \emph{cellular subquotient}  $\csqa{Q}{P}$ is the minimal $\pH$-cellular subquotient of $\eH$ containing $\Gamma_Q$ and $\Gamma_P$.
\end{definition}

A \emph{copy}  $X'$ of an atom  $X$ is an object isomorphic to  $X$.  For example, we say that an object in $\cs$ isomorphic to a GP csq (defined below) is a \emph{GP csq copy}.
%An object in CCP isomorphic to $\gpda{\dual{G}_\lambda}{\dual{G}_{1^n}}$ for some $\lambda$ is a \emph{DGP ccp copy}.
\begin{itemize}
\setlength{\itemsep}{1pt}\setlength{\topsep}{2pt}
\item \emph{Garsia-Procesi Cellular Subquotients} (GP CSQ). We say that the element $\csqa{G_{(n)}}{G_\lambda}$ of $\cs$ is the GP csq of shape $\lambda$.  This category is the full subcategory of $\cs$ with objects the GP csq of shape $\lambda$ for all $\lambda \vdash n$ and their copies. In \textsection\ref{s extending garsia-procesi}, we will show that $\fmod(\csqa{G_{(n)}}{G_\lambda})$ equals $R / I_\lambda$, the Garsia-Procesi module of shape $\lambda$.
\item \emph{Garsia-Procesi Cocyclage Posets} (GP CCP). Define $\gpa{G_{(n)}}{G_\lambda}$ to be the ccp on the set of tableaux given by the catabolizability conditions (a) and (b) of Proposition \ref{p catabolizable basics}. This is the full subcategory of CCP with objects $\{\fsp(X) : X \in \text{GP CSQ}\}$. In \textsection\ref{s extending garsia-procesi}, we will see that $\gpa{G_{(n)}}{G_\lambda} = \fsp(\csqa{G_{(n)}}{G_\lambda})$.
\item \emph{Dual Garsia-Procesi Cellular Subquotients} (dual GP CSQ): We say that the element $\csqa{\dual{G}_\lambda}{\dual{G}_{1^n}}$ of $\cs$ is the dual GP csq of shape $\lambda$. This category is the full subcategory of $\cs$ with objects the dual GP csq of shape $\lambda$ for all $\lambda \vdash n$ and their copies.
\item \emph{Dual Garsia-Procesi Cocyclage Posets} (DGP CCP). Define $\gpda{\dual{G}_\lambda}{\dual{G}_{1^n}}$ to be the sub-cocyclage poset of $\ccp(PAT)$ consisting of the tableaux given by conditions (c) and (d) of Proposition \ref{p catabolizable basics} (i.e.  $T$ such that  $T$ is $(\dual{G}_{1^n}, \lambda')$-row catabolizable). This is the full subcategory of CCP with objects $\gpda{\dual{G}_\lambda}{\dual{G}_{1^n}}$ and their copies. We conjecture that $\fsp(\csqa{\dual{G}_\lambda}{\dual{G}_{1^n}})$ equals $\gpda{\dual{G}_\lambda}{\dual{G}_{1^n}}$.
\item \emph{Shimozono-Weyman Cocyclage Posets} (SW CCP): The SW ccp $\swra{G_\lambda}{\eta}$ (resp. $\swca{\dual{G}_\lambda}{\eta}$) is the cocyclage poset consisting of the $(G_\lambda,\eta)$-row (resp.  $(\dual{G}_\lambda,\eta)$-column) catabolizable tableaux. This category is the full subcategory of CCP consisting of these cocyclage posets and their copies. Its objects are conjecturally in the image of $\fxccp$ and, stronger, in the image of $\fsp$.
\item \emph{Lascoux-Lapointe-Morse Cocyclage Posets} (LLM CCP): An LLM ccp will be defined in \textsection\ref{ss LLM atom} as the intersection of certain SW ccp. Again, these are conjecturally in the image of $\fxccp$ and $\fsp$.
\item \emph{Li-Chung Chen Cocyclage Posets} (Chen CCP): Chen ccp are a generalization of LLM ccp, also defined as the intersection of certain SW ccp; see \textsection\ref{ss chen atom}. Again, these are conjecturally in the image of $\fxccp$ and $\fsp$.
\end{itemize}
We have the following diagram of functors, which are all inclusions of full subcategories.  The ccp  $\ccp(\dual{\Tab(\eta)})$ will be defined in \textsection\ref{ss SSYT in PAT}.  They are dual GP ccp copies and this gives rise to the inclusion into DGP CCP.
\be
{\footnotesize \xymatrix@R=.5cm@C=1.5cm{
\text{GP CCP}  \ar[r]  \ar[rd]  & \text{SW CCP} \ar[rd]  & \                                  \\
\text{DGP CCP} \ar[ru] \ar[r]   & \text{Chen CCP} \ar[r] &  \text{CCP} \\
\{\ccp(\dual{\Tab(\eta)})\}_{\eta} \ar[u] &\text{LLM CCP} \ar[u]          &  \ }}
\ee

Note that if there is a ccp $X$ and an object $K$ of $\cs$ such that $\fsp(K) = X$, then $K$ is unique, and we may write $(\fsp)^{-1}(X)$ in place of $K$.  We conjecture the existence of categories  SW CSQ,  LLM CSQ,  Chen CSQ that map to SW CCP,  LLM CCP,  Chen CCP under  $\fsp$.  We have similar weaker conjectures for subcategories of XCCP.
There are similar conjectural diagrams for inclusions of full subcategories of convex cocyclage posets and full subcategories of $\cs$ that would map to these full subcategories under $\fxccp$ and $\fsp$. There is also a similar conjectural diagram for full subcategories of $R\star W_f$-$\Mod$ that would be the image of the diagram of full subcategories of $\cs$ under $\fmod$.

%We can also define the categories GP XCCP $:= \fxccp(GP CSQ)$,
\section{A $\pW$-graph version of the coinvariants}
\label{s pW-graph version of R_(1^n)}
We exhibit a cellular subquotient $\cinv$ of $\pH$ which is a $\pW$-graph version of the ring of coinvariants $R_{1^n}$. We show that under a natural identification of the left cells of $\cinv$ with SYT, the subposet of $\klo{\cinv}$ consisting of the cocyclage-edges is exactly the cocyclage poset on SYT.
\subsection{}
\label{ss factorization theorem}
There are two important theorems that give the canonical basis of $\eH$ a more explicit description. These theorems hold in arbitrary type, but we state them in type $A$ to simplify notation.

Recall that $Y_+ \subseteq Y$ is the set of dominant weights, which in type $A_{n-1}$ are weakly decreasing $n$-tuples of integers; put  $\pY_+ = \pY \cap Y_+$.  As is customary, let $w_0$ denote the longest element of $W_f$.
If $\lambda \in Y_+$, then $w_0 y^\lambda$ is maximal in its double coset $W_f y^\lambda W_f$. For $\lambda \in \pY_+$, let $s_\lambda(\y) \in \eH$ denote the Schur function of shape $\lambda$ in the Bernstein generators $\y_i$.
%(if $\lambda \notin \pY_+$, then there exists a $d \in \ZZ$ such that $\lambda + (d,\dots,d) \in \pY_+$ and $s_\lambda(\y) := (\y_1\dots \y_n)^{-d} s_{\lambda+(d,\dots,d)}(\y)$).

\begin{theorem}[Lusztig {\cite[Proposition 8.6]{L}}]\label{t Lusztig}
For any $\lambda \in \pY_+$, the canonical basis element $\C_{w_0 y^\lambda}$ can be expressed in terms of the Bernstein generators as
\[\C_{w_0 y^\lambda}=s_\lambda(\y)\C_{w_0}=\C_{w_0}s_\lambda(\y).\]
\end{theorem}

Recall from \textsection\ref{ss cocyclage poset} that for a standard word $v \in W_f$, $\cl{v}$ denotes the cocharge labeling of $v$, which is a sequence of $n$ nonnegative integers. Thinking of $\cl{v}$ as an element of $Y^{+}$, let $D \subset Y^{+}$ denote the set of cocharge labelings, which is in bijection with $W_f$. The set $\{y^{\beta} : \beta \in D \}$ are the \emph{descent monomials}. Next, put
\be \begin{array}{rl}
\ds := &\{\lJ{(y^\beta)}{S} : \beta \in D\}, \\
\dsw := & \{\lJ{(y^\beta)}{S}w_0 : \beta \in D\},
\end{array} \ee
which are the minimal and maximal coset representatives corresponding to descent monomials. The set $\dsw$ will index a canonical basis of the coinvariants.

\begin{proposition} \label{p bijectionwfdsw} There is a bijection $W_f \to \dsw$, $v \mapsto w$, defined by setting the word of $w$ to be $w_i = n \cl{v}_i + v_i$. Its inverse has the two descriptions
{
\renewcommand{\minalignsep}{0pt}
\begin{flalign}
\label{e inverselj} \lj{w}{S} \mapsfrom w \\
\label{e inversersd} \rsd{w_1} \rsd{w_2} \dots \rsd{w_n} \mapsfrom w
\end{flalign}}
where \rsd{w_i} is the residue of $w_i$ as defined in \textsection\ref{ss words of eW}.
\end{proposition}

\begin{proof}
We know that $W_f \to D$, $v \mapsto \cl{v}$ is a bijection, so we need to show that $\lJ{({y}^{\cl{v}})}{S} w_0$ has word $n \cl{v} + v$. This holds because $n \cl{v} + v$ belongs to the coset ${{y}^{\cl{v}}} W_f$ by (\ref{e wordmult6}) and is maximal in this coset by Proposition \ref{p wordrj}.
Equation (\ref{e inverselj}) follows from Proposition \ref{p wordlj} and the fact that a permutation can be recovered from its cocharge labeling by breaking ties with the rule that cocharge labels increase from left to right.
\end{proof}

\begin{example}
\label{ex cocharge}
For the $v \in \S_9$ given by its word below, the corresponding $\cl{v}$ and $w$ follow.
\[   \begin{array}{rcccccccccccc}
v &= &1&6&8&4&2&9&5&7&3, \\
\cl{v} &= &0&2&3&1&0&3&1&2&0, \\
w = n \cl{v} + v &= &1&26&38&14&2&39&15&27&3. \\
\end{array}\]
\end{example}

The \emph{lowest two-sided $\eW$-cell} of $W_e$ is the set $\{w \in W_e: w = x \cdot w_0 \cdot z, \text{ for some } x,z \in W_e\}$, denoted $\lrcelllong{(n)}$ (see \textsection\ref{ss cells in We} for more on these two-sided cells). As preparation for the next theorem, we have a proposition giving the factorization of any $w \in \lrcelllong{(n)} \cap \pW$ in terms of descent monomials. This is not too hard to see from the combinatorial description, however it is more easily proved with the help of a geometric description of $\ds$ in terms of alcoves, which we omit here.

\begin{proposition}[{\cite[Proposition 3.7]{B2}}]
\label{p two-sided primitive}
For any $w \in \lrcelllong{(n)} \cap \pW$, there is a unique expression for $w$ of the form
\be w = u_1 \cdot w_0 y^\lambda \cdot u_2 \ee
where $u_1, \invlr(u_2) \in \ds$ and $\lambda \in \pY_+$ ($\invlr$ is defined in \textsection\ref{ss invlr}).
\end{proposition}
\begin{proof}
This follows easily from the corresponding \cite[Proposition 3.7]{B2} for $G = SL_n$.
\end{proof}

%$v(\lambda)$ can also be obtained from $v$ by first taking $(v^{-1})^\dagger$ and then, beginning with $1$ in the word, label by 0 and then go to $2$, etc., increasing the label if and only if you go to the left.

The next powerful theorem simplifying the canonical basis of $\pH$ is due to Xi (\cite[Corollary 2.11]{X}), also found independently by the author. In the language of \cite{X}, the condition $w \in \ds$ is written $wA_{v'} \subseteq \Pi_{v'}$, where $v'$ is a special point, $A_{v'}$ is an alcove, and $\Pi_{v'}$ is a box. We state here a combination of Lusztig's theorem (Theorem \ref{t Lusztig}) and Xi's theorem.

For $v$ such that $v$ minimal in $v W_f$ (resp. $W_f v$), define $\lC_v$ (resp. $\rC_v$) by $\C_{v w_0} = \lC_v \C_{w_0}$ (resp. $\C_{w_0 v} = \C_{w_0} \rC_v $).

\begin{theorem}
\label{t factoriz}
For $w \in \lrcelllong{(n)} \cap \pW$ and with  $w = u_1 \cdot w_0 y^\lambda \cdot u_2$ as in Proposition \ref{p two-sided primitive}, we have the factorization
\[ \C_{w} = s_\lambda({\y}) \lC_{u_1} \C_{w_0} \rC_{u_2}. \]
\end{theorem}

\begin{remark}
The generalization of this theorem to arbitrary types is most natural for root systems associated to simply connected Lie groups $G$ (in particular, more natural for $G = SL_n$ than $G = GL_n$) because, in the simply connected case, the set playing the role of $\ds$ is naturally in bijection with $W_f$. The result Xi proves is for the simply connected case. However, if we work with the positive part $\pW$ of $\eW$, the $G = GL_n$ case is just as nice or nicer than the $SL_n$ case.
\end{remark}

\subsection{}

%\begin{proposition}
%In type $A_{n-1}$, $x\in W_a$ is non-crossing if and only if  $\ 1 \leq x_{i+1}-x_i\leq n$ for $i=1,\ldots,n-1$, where $\ldots x_0\ x_1\ x_2\dots x_{n-1} \ x_n \ x_{n+1} \ldots$ is $x$ written as an $n$-periodic permutation.
%\end{proposition}
%
%\begin{proof}
%The correspondence between $n$-periodic permutation notation and $x$ as an automorphism of $X$ is as follows.
%$$x(\epsilon_i) = \epsilon_{\br{x}_i}+\left(\frac{\br{x}_i-x_i}{n}\right)\delta\ ,$$
%where the $\br{x}_i$ is the unique integer in $[n]$ congruent to $x_i\mod n$. Define $\rho_n = (0,-1,\ldots,-n+1)\in Y$. The function $R\to \ZZ$ defined by
%$$\tau:\ \alpha\to\langle\alpha,\rho_n+n\Lambda^{\vee}\rangle$$ induces a partial order on $R$ from the order of $\ZZ$. This map takes $\alpha_i$ to $1$ for $i\in [n]$. The inverse image of $[n]$ under this map is $(R_{f-}+\delta) \cup R_{f+}$. Then, $\tau(x(\epsilon_i)) = -\br{x}_i + \br{x}_i - x_i = -x_i$, so $\tau(x(\alpha_i)) = x_{i+1} - x_i$ is in $[n]$ if and only if $x(\alpha_i)\in (R_{f-}+\delta) \cup R_{f+}$.
%\end{proof}

Let $e^+=\C_{w_0}$. Then $Ae^+$ is the one-dimensional trivial left-module of $\H$ in which the $T_i$ act by $\u$ for $i \in [n-1]$. The $\pH$-module $\pH e^+ = \pH \tsr_\H e^+$ is a $\u$-analogue of the polynomial ring $R$; more precisely, $\pH e^+$ is a $\u$-analogue of the left $R\star W_f$-module $Re^+$. Without saying so explicitly, we will identify the $\pH$-module $\pH e^+$ with the cellular submodule of $\pH$ spanned by $\{\C_w:w \text{ maximal in } wW_f\}$ as modules with basis. (It is easy to see directly that this is possible; it is also a special case of general results about inducing $W$-graphs \cite[Proposition 2.6]{B0}.)

Let $\R$ denote the subalgebra of $\pH$ generated by the Bernstein generators $\y_i$. Thus $\R \cong R$ as algebras. Write $(\pY)_{\geq d}^{W_f} \subseteq \R$ for the set of $W_f$-invariant polynomials of degree at least $d$. Now Theorem \ref{t factoriz} applied to the canonical basis of $\pH e^+$ yields the following corollary, which gives a $\u$-analogue of the ring of coinvariants.  Later, in \textsection\ref{ss atoms 1}, we will prove a more general result (Theorem \ref{t coinvariant copies}) that uses the full power of Theorem \ref{t factoriz}.

\begin{corollary}
\label{c cinv}
The $\pH_n$-module $\pH_n e^+$ has a cellular quotient equal to
\[ \cinv := \pH_n e^+/\pH_n (\pY)_{\geq 1}^{\S_n}e^+ \]
with canonical basis $\{\C_w : w \in \dsw\}$.
\end{corollary}
%\begin{proof}
%Put \[\I{cell}{1^n} = A\{\C_w e^+ : w \in \pW^S, w \notin \ds \}.\]
%>From Theorems \ref{t Lusztig} and \ref{t factoriz} with $z \in \pY_+$ and Proposition \ref{p factorize help}, it is clear that the submodule $\pH (\pY)_{\geq 1}^{\S_n} e^+$ of $\pH e^+$ contains $\I{cell}{1^n}$. Moreover, any $h \in \pH (\pY)_{\geq 1}^{\S_n} e^+$ can be written as
%\be h = \sum_{\lambda \in \pY_+,\ \lambda \neq \textbf{0}} s_\lambda(\y) c_\lambda e^+,\ c_\lambda \in \pH. \ee
%Expanding the $c_\lambda$ in the basis $\{\C_w e^+: w \in \pW^S\}$ and applying Theorems \ref{t Lusztig} and \ref{t factoriz} and Proposition \ref{p factorize help}, the sum can be rewritten as
%\be \sum_{\substack{\lambda \in \pY_+,\ \lambda \neq \textbf{0}\\\mu \in \pY_+,\ u \in \ds}} s_\lambda(\y) s_\mu(\y) \C_u e^+. \ee
%Further rewriting the products $s_\lambda(\y) s_\mu(\y)$ in terms of $s_\nu(\y)$, $\nu \neq \textbf{0}  $ and again applying Theorems \ref{t Lusztig} and \ref{t factoriz} shows that $h \in \I{cell}{1^n}$. Thus the submodule $\pH (\pY)_{\geq 1}^{\S_n} e^+ = \I{cell}{1^n}$ is cellular.
%\end{proof}
A careful proof of this corollary is postponed to the proof of Theorem \ref{ss atoms 1}.

\begin{example}
The $\pW$-graph $\cinvn{3}$ is drawn in Figure \ref{f wgraph n=3} with the following conventions: basis elements of the same degree are drawn on the same horizontal level; the edges with a downward component are exactly the corotation-edges (these correspond to left-multiplication by $\pi$ and increase degree by 1); arrows indicate relations in the preorder $\klo{\cinv}$.

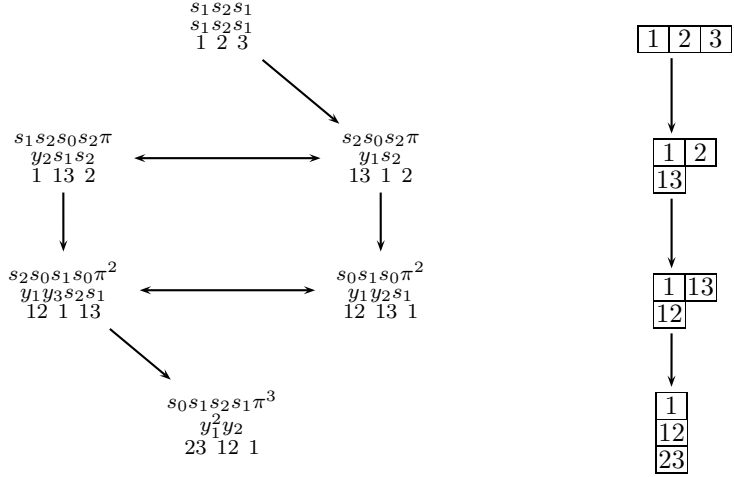
\begin{figure}
\begin{pspicture}(0,0)(7,6){\small
%\newdimen\ycor
%\newdimen\xcor
\newdimen\hcent
%\newdimen\temp
\newdimen\hspacedim

\hcent=30pt
\ycor=165pt
\advance\ycor by 0pt

\hspacedim=140pt
\xcor=\hcent

\temp=140pt
\multiply \temp by 0 \divide \temp by 2
\advance\xcor by -\temp
\rput(\xcor,\ycor){\rnode{a}{$\substack{s_1 s_2 s_1\\s_1 s_2 s_1\\1\ 2\ 3}$}}
\advance\xcor by \hspacedim
\advance\ycor by -50pt
\hspacedim=120pt

\xcor=\hcent

\temp=120pt
\multiply \temp by 1 \divide \temp by 2
\advance\xcor by -\temp
\rput(\xcor,\ycor){\rnode{b}{$\substack{s_1 s_2 s_0 s_2 \pi\\y_2 s_1 s_2\\1\ 13\ 2}$}}
\advance\xcor by \hspacedim
\rput(\xcor,\ycor){\rnode{c}{$\substack{s_2 s_0 s_2 \pi\\y_1 s_2\\13\ 1\ 2}$}}
\advance\xcor by \hspacedim
\advance\ycor by -50pt
\xcor=\hcent
\temp=120pt
\multiply \temp by 1 \divide \temp by 2
\advance\xcor by -\temp
\rput(\xcor,\ycor){\rnode{d}{$\substack{s_2 s_0 s_1 s_0 \pi^2\\y_1 y_3 s_2 s_1\\12\ 1\ 13}$}}
\advance\xcor by \hspacedim
\rput(\xcor,\ycor){\rnode{e}{$\substack{s_0 s_1 s_0 \pi^2\\y_1 y_2 s_1\\12\ 13\ 1}$}}
\advance\xcor by \hspacedim
\advance\ycor by -50pt
%\hspacedim=160pt
\xcor=\hcent
\temp=120pt
\multiply \temp by 0 \divide \temp by 2
\advance\xcor by -\temp
\rput(\xcor,\ycor){\rnode{f}{$\substack{s_0 s_1 s_2 s_1 \pi^3\\y_1^2 y_2\\23\ 12\ 1}$}}}
\ncline[nodesep=4pt]{->}{a}{c}
\ncline[nodesep=8pt]{<->}{b}{c}
\ncline[nodesep=8pt]{<->}{d}{e}
\ncline[nodesep=4pt]{->}{d}{f}
\ncline[nodesep=4pt]{->}{c}{e}
\ncline[nodesep=4pt]{->}{b}{d}
\footnotesize
\horizcent=200pt
\ycor=160pt
\advance\ycor by 0pt
\horizspace=100pt
\xcor=\horizcent
\temp=100pt
\multiply \temp by 0 \divide \temp by 2
\advance\xcor by -\temp
\rput(\xcor,\ycor){\rnode{v0h4}{\begin{Young}
1&2&3\cr
\end{Young}}}
\advance\xcor by \horizspace
\advance\ycor by -48pt
\horizspace=100pt
\xcor=\horizcent
\temp=100pt
\multiply \temp by 0 \divide \temp by 2
\advance\xcor by -\temp
\rput(\xcor,\ycor){\rnode{v1h3}{\begin{Young}
1&2\cr
13\cr
\end{Young}}}
\advance\xcor by \horizspace
\advance\ycor by -51pt
\horizspace=100pt
\xcor=\horizcent
\temp=100pt
\multiply \temp by 0 \divide \temp by 2
\advance\xcor by -\temp
\rput(\xcor,\ycor){\rnode{v2h2}{\begin{Young}
1&13\cr
12\cr
\end{Young}}}
\advance\xcor by \horizspace
\advance\ycor by -50pt
\horizspace=100pt
\xcor=\horizcent
\temp=100pt
\multiply \temp by 0 \divide \temp by 2
\advance\xcor by -\temp
\rput(\xcor,\ycor){\rnode{v3h1}{\begin{Young}
1\cr
12\cr
23\cr
\end{Young}\quad}}
\advance\xcor by \horizspace
\ncline[nodesep=2pt]{->}{v2h2}{v3h1}
\ncline[nodesep=2pt]{->}{v1h3}{v2h2}
\ncline[nodesep=2pt]{->}{v0h4}{v1h3}
\end{pspicture}
\caption{On the left is the $\pW$-graph of $\cinvn{3}$ with three labels for each canonical basis element. The bottom labels are inverted window words. On the right are the corresponding left cells and the covering relations of the partial order on left cells.}
\label{f wgraph n=3}
\end{figure}
\end{example}

\subsection{}
We now relate combinatorics of the cellular subquotient $\cinv$ to the cocyclage poset on standard tableaux. Let $\klo{\cinv}$ be the preorder of the $\pW$-graph $\cinv$, which is the restriction of the preorder $\klo{\pH}$ on $\pH$ to the subquotient $\cinv$.  We know from Proposition \ref{p preorder cells vs cocyclage poset} that the partial order $\klo{\cinv}$ on cells is the transitive closure of $\klo{\Res_\H \cinv}$ and cocyclage-edges (see \textsection\ref{ss cells of pH}). As the next proposition shows, cocyclage-edges in $\cinv$ are essentially cocyclages.

Let $T+T'$ denote the entry-wise sum of two tableau  $T,T'$ of the same shape.

\begin{proposition}
\label{p ccp syt equals ccp cinv}
The map
\[ \ccp(SYT) \to \fsp(\cinv), \\
T \mapsto n \cl{T}+T\]
is an isomorphism in Cocyclage Posets ($\cl{T}$ is defined in \textsection\ref{ss cocyclage poset}).
\end{proposition}
\begin{proof}
Since under the bijection of Proposition \ref{p bijectionwfdsw} $\lj{w}{S} = v$ holds, $w = n\cl{v} + v$ implies $P(w) = P(n \cl{v} + v) = n \cl{P(v)} + P(v)$. The statement is then a consequence of the following proposition.
\end{proof}

\begin{proposition} Under the bijection of Proposition \ref{p bijectionwfdsw}, corotation of standard words corresponds exactly to corotation of affine words.
\end{proposition}

\begin{proof}
To see that a corotation of a standard word maps to a corotation of an affine word, observe that corotating a standard word adds 1 to the cocharge label of the corotated number. To go the other way, use (\ref{e inversersd}): the inverse $\dsw \to W_f$ can be computed from $v_i = \rsd{w_i}$. Finally, observe that the last number of $v$ is 1 exactly when $\pi(n\cl{v}+v) \notin \dsw$.
\end{proof}

Figure \ref{f affine5} depicts the cells of the $\pW$-graph on $\cinvn{5}$ and the partial order $\klo{\cinvn{5}}$ on cells.

\begin{figure}
\begin{pspicture}(0pt,-25pt)(300pt,545pt){\footnotesize
% hoogte = height
% breedte = width
% dikte = linewidth
\hoogte=10pt
\breedte=13pt
\dikte=0.2pt

%\newdimen\hcent
\hcent=150pt
%\newdimen\ycor
\ycor=530pt
\advance\ycor by 0pt
\newdimen\hspacedim
\hspacedim=100pt
%\newdimen\xcor
\xcor=\hcent
%\newdimen\temp
\temp=100pt
\newdimen\ytmp
\ytmp=10pt
\multiply \temp by 0 \divide \temp by 2
\advance\xcor by -\temp
\rput(\xcor,\ycor){\rnode{v0h26}{\begin{Young}
1&2&3&4&5\cr
\end{Young}}}
\advance\xcor by \hspacedim
\advance\ycor by -40pt
\advance\ycor by \ytmp
%\newdimen\hspacedim
\hspacedim=100pt
%\newdimen\xcor
\xcor=\hcent
%\newdimen\temp
\temp=100pt
\multiply \temp by 0 \divide \temp by 2
\advance\xcor by -\temp
\rput(\xcor,\ycor){\rnode{v1h25}{\begin{Young}
1&2&3&4\cr
15\cr
\end{Young}}}
\advance\xcor by \hspacedim
\advance\ycor by -50pt
\advance\ycor by \ytmp
%\newdimen\hspacedim
\hspacedim=100pt
%\newdimen\xcor
\xcor=\hcent
%\newdimen\temp
\temp=100pt
\multiply \temp by 1 \divide \temp by 2
\advance\xcor by -\temp
\rput(\xcor,\ycor){\rnode{v2h23}{\begin{Young}
1&2&3&15\cr
14\cr
\end{Young}}}
\advance\xcor by \hspacedim
\rput(\xcor,\ycor){\rnode{v2h24}{\begin{Young}
1&2&3\cr
14&15\cr
\end{Young}}}
\advance\xcor by \hspacedim
\advance\ycor by -55pt
\advance\ycor by \ytmp
%\newdimen\hspacedim
\hspacedim=100pt
%\newdimen\xcor
\xcor=\hcent
%\newdimen\temp
\temp=100pt
\multiply \temp by 2 \divide \temp by 2
\advance\xcor by -\temp
\rput(\xcor,\ycor){\rnode{v3h19}{\begin{Young}
1&2&14&15\cr
13\cr
\end{Young}}}
\advance\xcor by \hspacedim
\rput(\xcor,\ycor){\rnode{v3h21}{\begin{Young}
1&2&15\cr
13&14\cr
\end{Young}}}
\advance\xcor by \hspacedim
\rput(\xcor,\ycor){\rnode{v3h22}{\begin{Young}
1&2&3\cr
14\cr
25\cr
\end{Young}}}
\advance\xcor by \hspacedim
\advance\ycor by -60pt
\advance\ycor by \ytmp
\newdimen\hspacedim
\hspacedim=105pt
\newdimen\xcor
\xcor=\hcent
\newdimen\temp
\temp=105pt
\multiply \temp by 3 \divide \temp by 2
\advance\xcor by -\temp
\rput(\xcor,\ycor){\rnode{v4h13}{\begin{Young}
1&13&14&15\cr
12\cr
\end{Young}}}
\advance\xcor by \hspacedim
\rput(\xcor,\ycor){\rnode{v4h17}{\begin{Young}
1&2&14\cr
13&25\cr
\end{Young}}}
\advance\xcor by \hspacedim
\rput(\xcor,\ycor){\rnode{v4h18}{\begin{Young}
1&2&14\cr
13\cr
25\cr
\end{Young}}}
\advance\xcor by \hspacedim
\rput(\xcor,\ycor){\rnode{v4h20}{\begin{Young}
1&2\cr
13&14\cr
25\cr
\end{Young}}}
\advance\xcor by \hspacedim
\advance\ycor by -70pt
\advance\ycor by \ytmp
\newdimen\hspacedim
\hspacedim=115pt
\newdimen\xcor
\xcor=\hcent
\newdimen\temp
\temp=115pt
\multiply \temp by 3 \divide \temp by 2
\advance\xcor by -\temp
\rput(\xcor,\ycor){\rnode{v5h11}{\begin{Young}
1&13&14\cr
12&25\cr
\end{Young}}}
\advance\xcor by \hspacedim
\rput(\xcor,\ycor){\rnode{v5h12}{\begin{Young}
1&13&14\cr
12\cr
25\cr
\end{Young}}}
\advance\xcor by \hspacedim
\rput(\xcor,\ycor){\rnode{v5h15}{\begin{Young}
1&2&25\cr
13\cr
24\cr
\end{Young}}}
\advance\xcor by \hspacedim
\rput(\xcor,\ycor){\rnode{v5h16}{\begin{Young}
1&2\cr
13&25\cr
24\cr
\end{Young}}}
\advance\xcor by \hspacedim
\advance\ycor by -78pt
%\advance\ycor by -70pt
\advance\ycor by \ytmp
\newdimen\hspacedim
\hspacedim=105pt
\newdimen\xcor
\xcor=\hcent
\newdimen\temp
\temp=105pt
\multiply \temp by 3 \divide \temp by 2
\advance\xcor by -\temp
\rput(\xcor,\ycor){\rnode{v6h7}{\begin{Young}
1&13&25\cr
12&24\cr
\end{Young}}}
\advance\xcor by \hspacedim
\rput(\xcor,\ycor){\rnode{v6h9}{\begin{Young}
1&13&25\cr
12\cr
24\cr
\end{Young}}}
\advance\xcor by \hspacedim
\rput(\xcor,\ycor){\rnode{v6h10}{\begin{Young}
1&13\cr
12&25\cr
24\cr
\end{Young}}}
\advance\xcor by \hspacedim
\rput(\xcor,\ycor){\rnode{v6h14}{\begin{Young}
1&2\cr
13\cr
24\cr
35\cr
\end{Young}}}
\advance\xcor by \hspacedim
\advance\ycor by -70pt
\advance\ycor by \ytmp
\newdimen\hspacedim
\hspacedim=100pt
\newdimen\xcor
\xcor=\hcent
\newdimen\temp
\temp=100pt
\multiply \temp by 2 \divide \temp by 2
\advance\xcor by -\temp
\rput(\xcor,\ycor){\rnode{v7h5}{\begin{Young}
1&24&25\cr
12\cr
23\cr
\end{Young}}}
\advance\xcor by \hspacedim
\rput(\xcor,\ycor){\rnode{v7h6}{\begin{Young}
1&13\cr
12&24\cr
35\cr
\end{Young}}}
\advance\xcor by \hspacedim
\rput(\xcor,\ycor){\rnode{v7h8}{\begin{Young}
1&13\cr
12\cr
24\cr
35\cr
\end{Young}}}
\advance\xcor by \hspacedim
\advance\ycor by -70pt
\advance\ycor by \ytmp
\newdimen\hspacedim
\hspacedim=100pt
\newdimen\xcor
\xcor=\hcent
\newdimen\temp
\temp=100pt
\multiply \temp by 1 \divide \temp by 2
\advance\xcor by -\temp
\rput(\xcor,\ycor){\rnode{v8h3}{\begin{Young}
1&24\cr
12&35\cr
23\cr
\end{Young}}}
\advance\xcor by \hspacedim
\rput(\xcor,\ycor){\rnode{v8h4}{\begin{Young}
1&24\cr
12\cr
23\cr
35\cr
\end{Young}}}
\advance\xcor by \hspacedim
\advance\ycor by -70pt
\advance\ycor by 20pt
\newdimen\hspacedim
\hspacedim=100pt
\newdimen\xcor
\xcor=\hcent
\newdimen\temp
\temp=100pt
\multiply \temp by 0 \divide \temp by 2
\advance\xcor by -\temp
\rput(\xcor,\ycor){\rnode{v9h2}{\begin{Young}
1&35\cr
12\cr
23\cr
34\cr
\end{Young}}}
\advance\xcor by \hspacedim
\advance\ycor by -82pt
\advance\ycor by 20pt
\newdimen\hspacedim
\hspacedim=100pt
\newdimen\xcor
\xcor=\hcent
\newdimen\temp
\temp=100pt
\multiply \temp by 0 \divide \temp by 2
\advance\xcor by -\temp
\rput(\xcor,\ycor){\rnode{v10h1}{\begin{Young}
1\cr
12\cr
23\cr
34\cr
45\cr
\end{Young}}}
\advance\xcor by \hspacedim
\ncarc[nodesep=1pt]{->}{v8h4}{v8h3}
\ncarc[nodesep=1pt]{->}{v7h6}{v7h5}
%\ncarc[nodesep=1pt]{->}{v7h8}{v7h5}
\ncarc[nodesep=1pt]{->}{v7h8}{v7h6}
\ncarc[nodesep=1pt]{->}{v6h9}{v6h7}
%\ncarc[nodesep=1pt]{->}{v6h10}{v6h7}
\ncarc[nodesep=1pt]{->}{v6h10}{v6h9}
\ncarc[nodesep=1pt]{->}{v5h12}{v5h11}
%\ncarc[nodesep=1pt]{->}{v6h14}{v6h7}
%\ncarc[nodesep=1pt]{->}{v6h14}{v6h9}
\ncarc[nodesep=1pt]{->}{v6h14}{v6h10}
\ncarc[nodesep=1pt,arcangle=15]{->}{v5h15}{v5h11}
%\ncarc[nodesep=1pt]{->}{v5h16}{v5h11}
\ncarc[nodesep=1pt,arcangle=15]{->}{v5h16}{v5h12}
\ncarc[nodesep=1pt]{->}{v5h16}{v5h15}
\ncarc[nodesep=1pt]{->}{v4h17}{v4h13}
%\ncarc[nodesep=1pt]{->}{v4h18}{v4h13}
\ncarc[nodesep=1pt]{->}{v4h18}{v4h17}
%\ncarc[nodesep=1pt]{->}{v4h20}{v4h13}
%\ncarc[nodesep=1pt]{->}{v4h20}{v4h17}
\ncarc[nodesep=1pt]{->}{v4h20}{v4h18}
\ncarc[nodesep=1pt]{->}{v3h21}{v3h19}
%\ncarc[nodesep=1pt]{->}{v3h22}{v3h19}
\ncarc[nodesep=1pt]{->}{v3h22}{v3h21}
\ncarc[nodesep=1pt]{->}{v2h24}{v2h23}
\ncline[nodesep=1pt]{->}{v9h2}{v10h1}
\ncline[nodesep=1pt]{->}{v8h3}{v9h2}
%\ncline[nodesep=1pt]{->}{v8h4}{v9h2}
\ncline[nodesep=1pt]{->}{v7h5}{v8h4}
%\ncline[nodesep=1pt]{->}{v7h6}{v8h3}
%\ncline[nodesep=1pt]{->}{v6h7}{v7h5}
\ncline[nodesep=1pt]{->}{v6h7}{v7h6}
%\ncline[nodesep=1pt]{->}{v7h8}{v8h4}
\ncline[nodesep=1pt]{->}{v6h9}{v7h8}
%\ncline[nodesep=1pt]{->}{v6h10}{v7h5}
%\ncline[nodesep=1pt]{->}{v5h11}{v6h9}
\ncline[nodesep=1pt]{->}{v5h11}{v6h10}
%\ncline[nodesep=1pt]{->}{v5h12}{v6h10}
\ncline[nodesep=1pt]{->}{v4h13}{v5h12}
%\ncline[nodesep=1pt]{->}{v6h14}{v7h6}
%\ncline[nodesep=1pt]{->}{v6h14}{v7h8}
%\ncline[nodesep=1pt]{->}{v5h15}{v6h7}
\ncline[nodesep=1pt]{->}{v5h15}{v6h14}
%\ncline[nodesep=1pt]{->}{v5h16}{v6h7}
%\ncline[nodesep=1pt]{->}{v5h16}{v6h9}
%\ncline[nodesep=1pt]{->}{v4h17}{v5h15}
\ncline[nodesep=1pt]{->}{v4h17}{v5h16}
%\ncline[nodesep=1pt]{->}{v4h18}{v5h11}
%\ncline[nodesep=1pt]{->}{v4h18}{v5h16}
%\ncline[nodesep=1pt]{->}{v3h19}{v4h13}
\ncline[nodesep=1pt]{->}{v3h19}{v4h18}
%\ncline[nodesep=1pt]{->}{v4h20}{v5h11}
%\ncline[nodesep=1pt]{->}{v4h20}{v5h12}
%\ncline[nodesep=1pt]{->}{v3h21}{v4h13}
\ncline[nodesep=1pt]{->}{v3h21}{v4h20}
%\ncline[nodesep=1pt]{->}{v3h22}{v4h17}
%\ncline[nodesep=1pt]{->}{v3h22}{v4h20}
%\ncline[nodesep=1pt]{->}{v2h23}{v3h19}
\ncline[nodesep=1pt]{->}{v2h23}{v3h22}
%\ncline[nodesep=1pt]{->}{v2h24}{v3h21}
%\ncline[nodesep=1pt]{->}{v1h25}{v2h23}
\ncline[nodesep=1pt]{->}{v1h25}{v2h24}
\ncline[nodesep=1pt]{->}{v0h26}{v1h25}
}
\end{pspicture}
\caption{The cells of the $\pW$-graph on $\mathscr{R}_{1^5}$. Edges are the covering relations of the partial order on cells.}
\label{f affine5}
\end{figure}

\section{A $\pW$-graph version of the Garsia-Procesi modules}
\label{s extending garsia-procesi}

The Garsia-Procesi approach to understanding the $R_\lambda = R / I_\lambda$ realizes $I_\lambda$ as the ideal of leading forms of functions vanishing on an orbit $\S_n \textbf{a}$, for certain $\textbf{a} \in \CC^n = \spec R$. We adapt this approach to the Hecke algebra setting using certain representations of $\eH$ studied by Bernstein and Zelevinsky in order to prove our main result, Theorem \ref{t main theorem gp}, which shows that the  $\u$-analogues  $\R_\lambda$ of the $R_\lambda$ are actually cellular.

Let $\mathcal{C}^\ZZ_n$ (resp. $\mathcal{C}^{+\ZZ}_n$) be the category of finite-dimensional $\eH_n$-modules (resp.  $\pH_n$-modules) in which the $\y_i$'s have their eigenvalues in $\u^{2\ZZ}$. In the next subsection, we review the needed results about the category $\mathcal{C}^\ZZ_n$, referring the reader to \cite{LNT, V} for a more thorough treatment.

\subsection{}
\label{ss BZ category}
For $\eta = (\eta_1, \eta_2, \dots, \eta_r)$ an $r$-composition of $n$, write $l_j = \sum_{i=1}^{j-1} \eta_i$, $j \in [r+1]$ for the partial sums of $\eta$ (where the empty sum is defined to be 0). Let $B_j$ be the interval $[l_j+1,l_{j+1}]$, $j \in [r]$, and define
\be J_\eta = \{s_i:\{i, i +1\} \subseteq B_j \text{ for some } j\} \ee
so that ${\S_n}_{J_\eta} \cong \S_{\eta_1} \times \dots \times \S_{\eta_r}$.

Let $\eH_\eta$ be the subalgebra of $\eH$ generated by $\H_{J_\eta}$ and $\y_i^{\pm 1}$,  $i \in [n]$. The algebra $\eH_\eta$ is isomorphic to $\eH_{\eta_1} \times \dots \times \eH_{\eta_r}$. Similarly, let $\pH_\eta$ be the subalgebra of $\pH$ generated by $\H_{J_\eta}$ and $\y_i$,  $i \in [n]$. For $\textbf{a} = (a_1, \dots, a_r) \in \ZZ^r$, let $\textbf{C}_{\eta,\textbf{a}}$ be the 1-dimensional representation of $\pH_\eta$ on which $\H_{J_\eta} \subseteq \pH_\eta$ acts trivially ($T_i$ acts by  $\u$ for  $s_i \in J_\eta$) and $\y_{l_i+1}$ acts by $\u^{2a_i}$, $i \in [r]$. The relations in $\pH_\eta$ demand that $\y_{l_i+k}$ acts by $\u^{2(a_i-k+1)}$ for $l_i + k \in B_i$. Note that our conventions differ from those in \cite{LNT} since we use the right affine Hecke algebra while they use the left.

Next define $M_{\eta,\textbf{a}}$ to be the induced module
\be M_{\eta,\textbf{a}} = \pH_n \tsr_{\pH_\eta} \textbf{C}_{\eta,\textbf{a}}. \ee
%
%Of primary interest to us is the case when $\textbf{a} = (l_2, l_3, \dots, l_r, l_{r+1})$; let us denote this tuple by $\textbf{a}_\eta$. This is the case in which $M_{\eta,\textbf{a}}$ has as many submodules as possible.

For $M$ in $\mathcal{C}^\ZZ_n$ or $\mathcal{C}^{+\ZZ}_n$, the \emph{points} of $M$ are the joint generalized eigenspaces for the action of the $\y_i$. The \emph{coordinates} of a point $v$ of $M$ is the tuple $(c_1, \dots, c_n)$ of generalized eigenvalues, i.e. $(\y_i -c_i)^{k_i} v = 0$ for some $k_i$ and all  $i \in [n]$.  The tuple $(c_1, \dots, c_n)$ is also identified with the word  $c_1 \ c_2 \cdots c_n$.

We are interested in the case where the points of $M_{\eta,\textbf{a}}$ are 1-dimensional.
\begin{proposition}
\label{p eigenspaces}
If the intervals $[a_i - \eta_i,a_i]$ are disjoint, then the points of $M_{\eta,\textbf{a}}$ are 1-dimensional with coordinates
\[{\S_n}^{J_\eta} (u^{2a_1}, u^{2(a_1-1)}, \dots, u^{2(a_1-\eta_1)}, u^{2a_2}, \dots, u^{2(a_2-\eta_2)}, \dots, u^{2a_r}, u^{2(a_r-1)}, \dots, u^{2(a_r-\eta_r)} ), \]
where $s_i$ acts on an $n$-tuple by swapping its $i$-th and $(i+1)$-st entries. Equivalently, the coordinates of the points of $M_{\eta,\textbf{a}}$ are shuffles of the words
\[ u^{2a_1}\ u^{2(a_1-1)} \cdots u^{2(a_1-\eta_1)},\ \  u^{2a_2}\cdots u^{2(a_2-\eta_2)},\ \ \ldots, u^{2a_r}\ u^{2(a_r-1)} \cdots u^{2(a_r-\eta_r)}. \]
\end{proposition}
\begin{proof}
 This is a special case of well-known results about inducing modules in  $\mathcal{C}^\ZZ_n$ (see \cite[\textsection 5]{V}).
\end{proof}

\begin{remark}
Since $\u$ is invertible in $A$, the $\y_i$ act invertibly on any module in $\mathcal{C}^{+\ZZ}_n$. Therefore $\mathcal{C}^\ZZ_n \to \mathcal{C}^{+\ZZ}_n$, given by $M \mapsto \Res_{\pH} M$ is an isomorphism of categories, so the known results about $\mathcal{C}^\ZZ_n$ carry over to $\mathcal{C}^{+\ZZ}_n$.
\end{remark}

\begin{remark}
There is not a significant difference between  $\mathcal{C}^\ZZ_n$ and the category of finite-dimensional $\eH$-modules, so it is common to only focus on  $\mathcal{C}^\ZZ_n$.  For our purposes, we only need objects in $\mathcal{C}^{+\ZZ}_n$ of the form $\pH_n \tsr_{\pH_\eta} N$, where the eigenvalues of the $\y_{l_i+1}$ on $N$ are generic. However, this can be equivalently achieved in the category $\mathcal{C}^{+\ZZ}_n$ by taking $N = \textbf{C}_{\eta,\textbf{a}}$ with $\textbf{a}$ generic. Misleadingly, the corresponding $\u=1$ modules $\CC \tsr_A N$ do not have generic eigenvalues.
\end{remark}

We complete this section with a couple more algebraic generalities, further preparing us for our main result Theorem \ref{t main theorem gp}.
Given any left $\pH$-module $M$, the annihilator $\ann M = \{h \in \pH : hM = 0\}$ is a 2-sided ideal of $\pH$.

For any two-sided ideal  $N$ of  $\pH$,  $N$ has a filtration
\be 0 \subseteq N_{\leq 0} \subseteq \dots \subseteq N_{\leq d} \subseteq \dots, \ee
where  $N_{\leq d} = (\pH)_{\leq d} \cap N$.
We can form the associated graded \[\gr(N) := \bigoplus_{d \geq 0} N_d / N_{d-1} \subseteq \bigoplus_{d \geq 0} (\pH)_{\leq d} / (\pH)_{<d} = \pH.\]
It is an ideal of $\pH$, isomorphic to $N$ as an $\H$-module. We also have that $\pH / N$ is isomorphic to $\pH / \gr(N)$ as an $\H$-module. For $h \in N$, define $\grin(h)$ to be the leading homogeneous component of $h$, i.e., the image of $h$ in $N_d / N_{d-1} = \gr(N)_d$, where $d$ is the smallest integer so that $h \in (\pH)_{\leq d}$.

\begin{proposition}
\label{p irreducibility needed}
Let $M_{\eta,\textbf{a}}$ be as above. If $M_{\eta,\textbf{a}}$ is irreducible, then it contains an element $v^+$ such that, setting $N = \ann v^+ $,  $\pH e^+/N e^+ \cong M_{\eta,\textbf{a}}$ as  $\pH$-modules. Thus by the discussion above, $\pH e^+/\gr(N) e^+ \cong M_{\eta,\textbf{a}}$ as  $\H$-modules.
\end{proposition}
\begin{proof}
As an $\H$-module, $M_{\eta,\textbf{a}}$ is the induced module $\H \tsr_{\H_{J_\lambda}} \lj{e^+}{J_\lambda}$, and this contains a copy of the trivial $\H$-module as a submodule. Let $v^+$ span this submodule. Thus there is an $\pH$-morphism $\pH e^+ \to M_{\eta,\textbf{a}}$, defined by $e^+ \mapsto v^+$. It is surjective by the irreducibility assumption and has kernel $N e^+$, hence the proposition.
%Let $v_1, \dots, v_k$ be an $A$-basis for $M$ with $Av_1$ a trivial $\H$-module. Since $M$ is irreducible, we can write $v_i = h_i v_1$ for some $h_i \in \pH$. Then $\ann(v_i) =  \ann(h_i v_1)$
%
%%We have $M \cong \pH / \ann v$ and  $\ann M = \cap_{w \in W_f} \ann (T_w v) = \cap_{w \in W_f} (\ann v)T_w^{-1}$. Thus $(\ann M) e^+ = (\ann v) e^+$.
\end{proof}

\begin{remark}
The assumption that $M_{\eta,\textbf{a}}$ is irreducible cannot be dropped.
\end{remark}

\subsection{}
The ideals $I_\lambda$ are generated by certain elementary symmetric functions in subsets variables, also known as Tanisaki generators (see \cite{GP, H3}). We show that certain $\C_w \in \cinv$ are essentially these generators. This will relate the ideals $\gr(\ann M_{\eta,\textbf{a}}) e^+$ to the canonical basis of $\pH e^+$. Indeed, this was our original motivation for applying the Garsia-Procesi approach to understand cellular submodules of $\cinv$.

Let us make the inclusion $\eH_{n-1} \hookrightarrow \eH_{(n-1,1)} \hookrightarrow \eH_n$ of \textsection\ref{ss BZ category} completely explicit. Recall that $S = \{s_1, s_2, \dots, s_{n-1}\}$ and let $S'$ be the subset $\{s_1, s_2, \dots, s_{n-2}\}$ of simple reflections of $W_f$. On the level of groups,
\be \iota_n:\ \eS_{n-1}\hookrightarrow \eS_n \ee
is given on generators by
\be
\label{e iota on s_n-1}
\begin{array}{ll} \iota_n(y_i) = y_i, &  i\in [n-1], \\
\iota_n(s_i) = s_i, & s_i \in S', \end{array}
\ee
from which it follows
\be
\iota_n(s_0) = s_{n-1} s_0 s_{n-1}, \\
\iota_n(\pi) = \pi s_{n-1}.
\ee
Since $\iota_n(s_0) \notin K$, this is not a morphism of Coxeter groups. This inclusion of groups restricts to an inclusion of monoids $\iota_n:\ \pS_{n-1}\hookrightarrow \pS_n$.

It is immediate from (\ref{e wordmult3}) and (\ref{e wordmult6}) that $\iota_n$ is given in terms of words by
\be \label{e wordiota}
\begin{array}{cccccl}
\lambda_1.x_1&\lambda_2.x_2& \cdots& \lambda_{n-1}.x_{n-1} &&\mapsto \\ \lambda_1.(x_1 + 1)& \lambda_2.(x_2 + 1)& \cdots& \lambda_{n-1}.(x_{n-1} + 1)&1, \end{array}
\ee
where $x_i \in [n-1]$ and $\lambda_i \in \ZZ$ (where, with the convention of \textsection\ref{ss words of eW}, $a . b = a(n-1) + b$ in the top line and $a . b = an + b$ in the bottom line).

The corresponding morphism of algebras $\iota_n:\ \eH_{n-1}\to \eH_n$ is given by
\be
\begin{array}{ll} \iota_n(Y_i) = Y_i, &  i\in [n-1], \\
\iota_n(T_i) = T_i, & s_i\in S'. \end{array}
\ee
from which it follows
\be \label{e iotapi} \begin{array}{rlll}\iota_n(\pi) &= \iota_n(\y_1T_1^{-1}T_2^{-1}\ldots T_{n-2}^{-1}) &= \y_1T_1^{-1}T_2^{-1}\ldots T_{n-2}^{-1} &= \pi T_{n-1}, \\
\iota_n(T_0) &= \iota_n(\pi^{-1} T_1\pi) &= T_{n-1}^{-1}\pi^{-1}T_1\pi T_{n-1} &= T_{n-1}^{-1}T_0T_{n-1}. \end{array}\ee

This map restricts to a map $\iota_n:\ \pH_{n-1}\to \pH_n$.

%\begin{proposition}
%\label{p iota-n}
%For any $w\in \pS_n$, $\iota_n(T_w) = T_{\iota_n(w)}$.
%\end{proposition}
%\begin{proof}Write $w = v_1\cdot\pi\cdot v_2\cdot\pi\ldots v_d\cdot\pi\cdot v_{d+1}$ with $v_i\in \S_{n-1}$ as in Lemma \ref{l pivpi}. Then
%\be \iota_n(T_w) = T_{v_1}\pi T_{n-1}T_{v_2}\pi T_{n-1}\ldots T_{v_k}\pi T_{n-1}T_{v_{k+1}}. \ee
%This is equal to $T_{\iota_n(w)}$ provided we can show $\ell(\iota_n(w)) = \ell(w) + k$. This is immediate from (\ref{e wordiota}) and the length formula Proposition \ref{p affinewordlength}.
%\end{proof}

\begin{lemma} \label{l elemsymkl}
For $k, d \in [n]$ such that $d \leq k$, let $\lambda = \epsilon_{k-d+1} + \ldots + \epsilon_k$. Then $\lJ{(y^\lambda)}{S} = v \pi^d$ for some $v \in \S_n$.
\end{lemma}
\begin{proof}
The word of $\lJ{(y_1y_2 \ldots y_d)}{S}$ is
{\footnotesize \be \pi^d = 1.d\ \ 1.(d-1) \cdots 1.1\ \ n\ \ n-1 \cdots d+1, \ee}
and the word of $\lJ{(y^\lambda)}{S}$ is
{\footnotesize \be n\ \ n-1\cdots n-(k-d)+1\ \ 1.d\ \ 1.(d-1) \cdots 1.1\ \ n-(k-d)\ \ n-(k-d)-1\cdots d+1. \ee}
This word is obtained from the word of $\pi^d$ by a sequence of left-multiplications by $s_i \in S$ that increase length by 1. This sequence yields the desired $v \in \S_n$.
\end{proof}

Recall that for any  $w$ maximal in its coset  $w W_f$, we can write  $\C_w = \lC_z \C_{w_0}$, where  $w = z \cdot w_0$.  We have $\lC_z = \sum_x \lP_{x,z} T_x$, where the sum is over $x \leq z$ such that $x \cdot w_0$ is reduced and $\lP_{x,z} := P'_{x w_0, z w_0}$ (see \cite{B0} for the more general construction of which this is a special case).

\begin{theorem} \label{t elemsymkl}
For $k, d \in [n]$ such that $d \leq k$, put $\lambda = \epsilon_{k-d+1} + \ldots + \epsilon_k$ and $w = \lJ{(y^{\lambda})}{S}w_0 \in \dsw$. Then
\be \label{e elemsymkl} \C_{w} = \u^{d(k-n)}s_{1^d}(\y_1,\ldots,\y_k) \C_{w_0}.\ee
\end{theorem}
\begin{proof} We proceed by induction on $n$. If $k= n$, then this is a special case of Theorem \ref{t Lusztig}. Otherwise, by induction, the following holds in $\pH_{n-1}$:
\be \label{e inductionstep} \C_{w'} = \u^{d(k-n+1)}s_{1^d}(\y_1,\ldots,\y_k) \C_{\lj{w_0}{S'}}, \ee
where $w' = \lJ{(y^\lambda)}{S'}\lj{w_0}{S'} \in D^{S'} w_0$.
Putting $z = \lJ{(y^\lambda)}{S'}$, we have $\C_{w'} = \sum_{x\leq z}\lP_{x,z}T_x\C_{\lj{w_0}{S'}}$. Applying $\iota_n$ to both sides, using Lemma \ref{l elemsymkl} ($z = u\pi^d, u \in \S_{n-1}$ implies any $x \leq z$ has a similar form) and then (\ref{e iotapi}), we obtain
\be \iota_n(\C_{w'}) = \sum_{x\leq z}\lP_{x,z}\iota_n(T_x)\C_{\lj{w_0}{S'}} = \sum_{x\leq z}\lP_{x,z}T_xT_{n-d}\ldots T_{n-1}\C_{\lj{w_0}{S'}}. \ee
Multiplying on the right by $\rC_{s_{n-1} s_{n-2} \ldots s_1}$, we obtain
\be \u^d \sum_{x\leq z}\lP_{x,z}T_x \C_{w_0}. \ee

Next, we show that
\be \label{e xzlp} \sum_{x\leq z}\lP_{x,z}T_x \C_{w_0} = \C_{w} \ee
using the characterization of the canonical basis from Theorem \ref{t kl canonical basis}.
It is not hard to see that $w' s_{n-1} s_{n-2} \dots s_1 = w$ using (\ref{e wordiota}) and the affine word computation in the proof of Lemma \ref{l elemsymkl}. Then the left-hand side of (\ref{e xzlp}) is certainly in $T_{w} + \ui \L$ ($\L$ as in Theorem \ref{t kl canonical basis}). To see that it is $\br{\cdot}$-invariant, use that
\be \sum_{x\leq z}\lP_{x,z}T_{x\pi^{-d}}\pi^d\C_{\lj{w_0}{S'}} = \sum_{x\leq z}\br{\lP_{x,z}}\br{T_{x\pi^{-d}}}\pi^d\C_{\lj{w_0}{S'}} \ee
as an equation in $\pH_{n-1}$.  Since $x \pi^{-d} \in \S_{n-1}$,  $\br{\iota_n(x \pi^{-d})} = \iota_n(\br{x \pi^{-d}})$.  Hence applying  $\iota_n$ to this equation and then multiplying on the right by $\u^{-d} \rC_{s_{n-1} s_{n-2} \ldots s_1}$ yields $\br{\cdot}$-invariance for the left-hand side of (\ref{e xzlp}).

Finally, the theorem follows by applying $\iota_n$ to both sides of (\ref{e inductionstep}) and multiplying on the right by $\u^{-d}\rC_{s_{n-1} s_{n-2} \ldots s_1}$ to obtain (\ref{e elemsymkl}).
\end{proof}

\subsection{}
In the next proposition, we relate the descriptions of $\gr(\ann M_{\lambda,\textbf{a}})$ in terms of elementary symmetric polynomials in subsets of the variables to catabolizability. Let $T^{d,k} = P(y_{k-d+1} y_{k-d+2} \dots y_k)$. Under the isomorphism $\ccp(SYT) \cong \fsp(\cinv)$ of Proposition \ref{p ccp syt equals ccp cinv}, $T^{d,k} = n\cl{T'}+T'$ for some SYT  $T'$.  The tableau $\cl{T'}$ has at most two rows and is filled with 0's and 1's; it has $n-d$ 0's in the first row and $\min(d,n-k)$ 1's in the second row. Set $\mu = \ctype(T^{d,k})$ (see \textsection\ref{ss catabolizability definition} for the definition of  $\ctype$).

\begin{proposition}
\label{p catab elem sym}
With $\mu$ as defined above, $d,k \in [n]$, $d \leq k$, and $\lambda \vdash n$, the following are equivalent:
\begin{list}{\emph{(\alph{ctr})}} {\usecounter{ctr} \setlength{\itemsep}{1pt} \setlength{\topsep}{2pt}}
\item $d > k-n+\lambda'_1 + \dots + \lambda'_{n-k}$,
\item $d > k-\sum_i (\lambda_i - (n-k))_{\geq 0}$,
\item $\mu \not\gd \lambda$,
\item $T^{d,k}$ is not $(G_\lambda,1^{\ell(\lambda)})$-row catabolizable,
\end{list}
where for $c \in \ZZ$, $(c)_{\geq 0}$ denotes $c$ if $c \geq 0$ and 0 otherwise.
\end{proposition}
\begin{proof}
The equivalence of (a) and (b) comes from counting the number of boxes in the first $n-k$ columns of the diagram of $\lambda$ in two different ways. The equivalence of (c) and (d) is well-known (see \cite{SW}). It is easy to see that the catabolizability of $T^{d,k}$ is $\mu = (n-d, \mu_2, \mu_2, \dots, \mu_2, r)$, where $\mu_2 = \min(d,n-k)$ and $r$  is the unique integer such that $r \leq \mu_2$ and $\mu \vdash n$.

Next, let $l$ be the number of parts of $\lambda$ that are greater than $n-k$. Rewriting condition (b), and using the computation of $\mu$, we have
\be \sum_{i=1}^l \lambda_i > k-d + l (n-k) = n-d+ (n-k) (l-1)  \geq \mu_1 + \sum_{i=2}^l \mu_i. \ee
This implies (c). To see that (c) implies (b), suppose  $\sum_{i=1}^l \lambda_i > \sum_{i=1}^l \mu_i$ for some $l$. Then as $\mu_2 = \min(d,n-k)$, this inequality also holds for $l$ equal to the number of parts of $\lambda$ that are greater than $n-k$.  If $\mu_2 = n-k$, then $\sum_{i=1}^l \mu_i = n-d+ (n-k) (l-1)$; this also holds if $\mu_2 = d$ because this implies $\mu = (n-d, d)$ which implies  $l=1$.  Hence (b) follows.
\end{proof}

A result of Garsia-Procesi (\cite[Proposition 3.1]{GP}) carries over to this setting virtually unchanged. For a composition $\eta$, let $\eta_+$ denote the partition obtained from $\eta$ by sorting its parts in decreasing order.
\begin{proposition}
\label{p garcia-procesi prop 3.1}
Suppose $\eta$ is an $r$-composition of $n$ with $\lambda := \eta_+$, and $k, d \in [n]$, $d \leq k$, such that any (all) of the conditions in Proposition \ref{p catab elem sym} are satisfied. If $M_{\eta, \textbf{a}}$ satisfies the hypotheses of Proposition \ref{p eigenspaces}, then
\[ s_{1^d}(\y_1,\dots,\y_k) \in \gr(\ann M_{\eta,\textbf{a}}). \]
\end{proposition}
\begin{proof}
Let $p_k^\y(t) = \prod_{i=1}^k (t+\y_i)$, thought of as a univariate polynomial in the indeterminate $t$. By Proposition \ref{p eigenspaces}, the points of $M_{\eta,\textbf{a}}$ are shuffles of words of length $\eta_1, \eta_2, \dots, \eta_r$. Thus the word of length $\eta_i$ must intersect the first $k$ letters of the shuffle in size at least $(\eta_i - (n-k))_{\geq 0}$. Therefore the value of $p_k^\y(t)$ on any point of $M_{\eta,\textbf{a}}$ is divisible by
\be g(t) := [t+u^{2a_1}]_{(\eta_1-(n-k))_{\geq 0}} [t+u^{2a_2}]_{(\eta_2-(n-k))_{\geq 0}} \ldots [t+u^{2a_r}]_{(\eta_r-(n-k))_{\geq 0}}, \ee
where
\be [t+u^a]_c := (t + u^a)(t + u^{a-2}) \cdots (t + u^{a-2(c-1)}). \ee

Put $m = \sum_i (\lambda_i - (n-k))_{\geq 0}$ and define $p_m^z(t) = \prod_{i=1}^m (t+z_i)$, thought of as a polynomial in $t$ with $\textbf{z}:=(z_1,\ldots,z_m) \in A^m$. Divide $p_k^\y(t)$ by $p_m^z(t)$ to obtain
\be \label{e division}
p_k^\y(t) = q(t)p_m^z(t) + r(t),
\ee
where $r(t) = \sum_{i = 0}^{m-1} c_i t^i$ is a polynomial in $t$ of degree less than $m$ with coefficients $c_i \in A[\y_1, \dots, \y_k]$. We will make use of the fact that equation (\ref{e division}) is homogeneous of degree $k$ if $t$, the $\y_i$'s, and the $z_i$'s have degree 1.

The coefficient $c_{k-d}$ exists as Proposition \ref{p catab elem sym} (b) is equivalent to $k-d < m$ and, for a certain $\textbf{z}$, it is the element of $\ann M_{\eta,\textbf{a}}$ we are looking for: on the one hand, if $\textbf{z}$ is chosen so that $g(t) = p_m^z(t)$, then $c_{k-d}$ evaluates to 0 on every point of $M_{\eta,\textbf{a}}$ as $p_k^\y(t)$ evaluated at any point of $M_{\eta,\textbf{a}}$ is divisible by $p_m^z(t)$. On the other hand, the leading component $\grin(c_{k-d})$ of  $c_{k-d}$ is obtained from $c_{k-d}$ by setting $\textbf{z}  = \textbf{0}$. Then since setting $\textbf{z} =\textbf{0}$ results in $p_m^z(t) = t^m$, (\ref{e division}) shows that $\grin(c_{k-d}) = s_{1^d}(\y_1, \dots, \y_k)$.
\end{proof}

\subsection{}
For $h \in \pH$, write $[\C_w]h$ for the coefficient of $\C_w$ of $h$ written as an $A$-linear combination of $\{\C_w: w \in \pW\}$. Define $\langle , \rangle_\lambda: \pH \times \pH e^+ \to A$ by
\be \langle h_1, h_2 \rangle_\lambda = [\C_{g_\lambda}]h_1 h_2, \ee
where $g_\lambda = \reading(G_\lambda)$.

We now come to our main result.
\begin{theorem}
\label{t main theorem gp}
Suppose $M_{\eta,\textbf{a}}$ satisfies the hypotheses of Propositions \ref{p eigenspaces} and \ref{p irreducibility needed} and maintain the notation of Proposition \ref{p irreducibility needed}. Then the following submodules of $\pH e^+$ are equal.
\begin{list}{\emph{(\roman{ctr})}} {\usecounter{ctr} \setlength{\itemsep}{1pt} \setlength{\topsep}{2pt}}
\item $\Il{o} := \gr(\ann v^+)e^+$,
\item $\Il{T} := \pH \{s_{1^d}(\y_1, \dots, \y_k) : d, k \text{ satisfy (a)-(d) of Proposition \ref{p catab elem sym}}\} e^+,$
\item $\Il{pair} := \{v \in \pH e^+ : \langle \pH, v \rangle_\lambda = 0\}$,
\item $\Il{cell}$ := The maximal cellular submodule of $\pH e^+$ not containing $\Gamma_{G_\lambda}$ ($\Gamma_{G_\lambda}$ is the cell labeled by $G_\lambda$),
\item $\Il{cat} := A \{ \C_w : P(w) \text{ is not $(G_\lambda, 1^{\ell(\lambda)})$-row catabolizable}\}$.
\end{list}
\end{theorem}

Note that $\Il{cat}$ is not obviously a submodule but will be shown to be one. The abbreviations o, T, pair, are shorthand for orbit, Tanisaki, and pairing.  Also note that modules $M_{\eta,\textbf{a}}$ satisfying the hypotheses of Propositions \ref{p eigenspaces} and \ref{p irreducibility needed} exist by the general theory. For instance, if $|a_i - a_j| >> 0$ for all $i \neq j$, then these hypotheses are satisfied.

Given the theorem, define $\R_\lambda$ to be  $\pH e^+/\Il{}$ for  $\Il{}$ equal to any (all) of the submodules above.

\begin{corollary}
For Garsia-Procesi atoms, we have the following diagram corresponding to the diagram in (\ref{e diagram atom categories})
\[
{\footnotesize \xymatrix@R=.5cm@C=1.4cm{
\R_\lambda = \csqa{G_{(n)}}{G_\lambda}  \ar@{{|}->}[rr]_{\fxccp} \ar@{{|}->}[d]_{\fmod} \ar@{{|}->}[rrd]_{\fsp} & & \fxccp(\R_\lambda) \ar@{{|}->}[d]\\
R_\lambda \ar@{{|}->}[rd]_{\mathscr{F}} & & \gpa{G_{(n)}}{G_\lambda} \ar@{{|}->}[ld]_{\mathscr{F}}\\
& \tilde{H}_\lambda(t) &
}}
\]
where $\tilde{H}_\lambda(t)$ are the cocharge variant transformed Hall-Littlewood polynomials (see \cite{H3}).

\end{corollary}

Write $\nnA$ for the semiring $\ZZ_{\geq 0}[\u,\ui] \subseteq A$ and $\posA$ for the subset $\ZZ_{>0}[\u,\ui] \subseteq \nnA$. Through the work of Kazhdan-Lusztig and Beilinson-Bernstein-Deligne-Gabber we have (see, for instance, \cite{L2})
\begin{theorem}
\label{t positive coefficients}
If $(W,S)$ is crystallographic, then the structure coefficients $\beta_{x,y,z} = [\C_z]\C_x\C_y$ belong to $\nnA$.
\end{theorem}

The next two corollaries could be phrased as general facts about any algebra with basis in which the structure coefficients are positive, however, we state them for the special cases that we need. Recall the notation of \textsection\ref{ss cells} and the definition (\ref{e preorder}) of $\delta \klocov{\Gamma} \gamma$. Note that
\be \label{e klocov fact}
\delta \klocov{\Gamma} \gamma \quad \Leftrightarrow \quad \beta_{x, \gamma, \delta} \neq 0 \text{ for some $x \in W$} \ee
as $[\delta] h\gamma \neq 0$ for some $h = \sum_{x \in W} a_x \C_x \in \H$, $a_x \in A$ implies $[\delta] \C_x\gamma = \beta_{x, \gamma, \delta} \neq 0$ for some $x$.

\begin{corollary}
\label{c positivity trick 1}
For any $\pW$-graph $\Gamma \subseteq \Gamma_{\pW}$ (i.e., the $\pW$-graph of some cellular subquotient of $\pH$), $\delta \klo{\Gamma} \gamma$ if and only if $\delta  \klocov{\Gamma} \gamma$.
\end{corollary}
\begin{proof}
The ``if'' direction is part of the definition of $\klo{\Gamma}$. For the ``only if'' direction, suppose $\gamma_3 \klocov{\Gamma} \gamma_2 \klocov{\Gamma} \gamma_1$. Then by (\ref{e klocov fact}) there exist $x_1, x_2 \in \pW$ such that $\beta_{x_1, \gamma_1, \gamma_2} \neq 0$ and $\beta_{x_2, \gamma_2, \gamma_3} \neq 0$. Applying Theorem \ref{t positive coefficients} yields
\begin{align}
x_1 \gamma_1 \in \posA \gamma_2 + \nnA \Gamma \text{ and}\\
x_2 \gamma_2 \in \posA \gamma_3 + \nnA \Gamma,
\end{align}
which imply
\be
x_2 x_1 \gamma_1 \in \posA \gamma_3 + \nnA \Gamma.
\ee
Thus $\gamma_3 \klocov{\Gamma} \gamma_1$.

The general case then follows by induction as $\delta \klo{\Gamma} \gamma$ means there exists $\delta = \gamma_n, \gamma_{n-1}, \dots, \gamma_1 = \gamma$ such that $\gamma_{i + 1} \klocov{\Gamma} \gamma_i$.
\end{proof}

\begin{corollary}
\label{c positivity trick 2}
If $\gamma \in \Il{pair}$, $\gamma \in \Gamma_{\pW}$, then $\delta \klo{\pH} \gamma$ $(\delta \in \Gamma_{\pW})$ implies $\delta \in \Il{pair}$, i.e., the cellular submodule generated by $\gamma$ is contained in $\Il{pair}$.
\end{corollary}
\begin{proof}
Suppose for a contradiction that $\delta \notin \Il{pair}$. Then by definition of $\Il{pair}$, $g_\lambda \klocov{\Gamma} \delta$. Applying Corollary \ref{c positivity trick 1} to this and the assumption $\delta \klo{\pH} \gamma$ implies $g_\lambda \klocov{\pH} \delta$, contradicting $\gamma \in \Il{pair}$.
\end{proof}

\begin{proof}[Proof of Theorem \ref{t main theorem gp}]
First we have $\Il{T} \subseteq \Il{o}$ by Proposition \ref{p garcia-procesi prop 3.1} and the inclusion $\ann M_{\eta,\textbf{a}} \subseteq \ann v^+$. We know by Proposition \ref{p irreducibility needed} that $\Res_\H \pH e^+ / \Il{o}$ affords the representation $\H_n \tsr_{\H_{J_\lambda}} \lj{e^+}{J_\lambda}$. Next, an argument of the same flavor as Proposition \ref{p garcia-procesi prop 3.1} yields $\Il{o} \subseteq \Il{pair}$: for $\mu \in \pY_+$, define
\[ f_\mu (\y_1, \dots, \y_n) = \prod_{i=1}^n \prod_{j = 1}^{\mu_i} (\y_i - \u^{2a_j}). \]
Assume that $\eta = \eta_+$; if not, the following argument works with the indices of the $a_j$ in the above expression permuted.  If $f_\mu \notin \ann M_{\eta,\textbf{a}}$, then $w^1 \subseteq [\mu'_1 + 1, n]$, $w^2 \subseteq [\mu'_2 + 1, n], \dots, w^r \subseteq [\mu'_r +1,n]$, where $w^1 \sqcup ,\ldots,\sqcup w^r = [n]$, $|w^i| = \eta_i$ determine the coordinates of a point in
$M_{\eta,\textbf{a}}$ by specifying the positions of the shuffled words of Proposition \ref{p eigenspaces}. Thus for $l \in [r]$, $w^1 \cup w^2 \cup \dots \cup w^l \subseteq [\mu'_l + 1, n]$ implying $\lambda_1 + \cdots + \lambda_l \leq n-\mu'_l$, or equivalently, $\mu'_l \leq \lambda_{l+1} + \dots + \lambda_r$. In particular, adding up these inequalities yields $|\mu| \leq n(\lambda)$. Thus $\{\y^\mu : \mu \in \pY_+, |\mu| > n(\lambda) \} \subseteq \gr(\ann M_{\eta,\textbf{a}}) \subseteq \gr(\ann v^+)$.

By specializing to $\u = 1$, it is easy to see that $\{ \H_n \y^\mu \H_n : \mu \in \pY_+, |\mu| = d\} = (\pH)_d$. Therefore, $(\pH)_{>n(\lambda)} \subseteq \Il{o}$. Since $\pH e^+ / \Il{o}$ contains a single copy of the representation of shape $\lambda$, we must have $\Gamma_{G_\lambda} \not\subseteq \Il{o}$. Note that $\Il{pair}$ is also the maximal submodule of $\pH e^+$ not containing $\Gamma_{G_\lambda}$.  Hence we have $\Il{o} \subseteq \Il{pair}$.

Now that the inclusion $\Il{T} \subseteq \Il{pair}$ is established, Theorem \ref{t elemsymkl} and Corollary \ref{c positivity trick 2} show that $\Il{T}$ is cellular. So $\Il{T}$ is a cellular submodule not containing $\Gamma_{G_\lambda}$, and hence $\Il{T} \subseteq \Il{cell}$.

Next, it follows from the algorithm for catabolizability in \cite{B1}, or alternatively, as a special case of Proposition \ref{p cat equivalences} (or rather, its dual version, which is just as good by Proposition \ref{p catabolizable basics}), that there is a sequence of ascent-edges and corotation-edges from $w$ to  $g_\lambda$ for any  $w$ with $P(w)$  $(G_\lambda, 1^{\ell(\lambda)})$-row catabolizable. Thus a cellular submodule of $\pH e^+$ containing $w$ contains $\Gamma_{G_\lambda}$, implying $\Il{cell} \subseteq \Il{cat}$.

We have shown that $\Il{T} \subseteq \Il{cell} \subseteq \Il{cat}$ and $\Il{T} \subseteq \Il{o} \subseteq \Il{pair}$.  The  $\u =1$ results of Garsia-Procesi and Bergeron-Garsia (see \cite{H3}) establish that $\rank_A (\pH e^+ / \Il{T}) = \rank_A (\pH e^+/\Il{pair}) = \binom{n}{\lambda_1, \dots, \lambda_r}$. The standardization map of Lascoux (see \cite{SW} and \textsection\ref{ss atoms 1}) shows that $\rank_A (\pH e^+ / \Il{cat}) = \binom{n}{\lambda_1, \dots, \lambda_r}$. Thus we have equalities \[\Il{T} = \Il{cell} = \Il{cat} = \Il{o} = \Il{pair}.\]
\end{proof}

\section{A duality in $\cinv$}
\label{s flip}
It is well-known that there is a perfect pairing $\langle , \rangle: R_{1^n} \times R_{1^n} \to \CC$ given by $\langle f_1, f_2 \rangle$ equal to the projection of $f_1 f_2$ onto the sign representation of $R_{1^n}$. With this, it is easy to show that an irreducible $V_\lambda \subseteq R_{1^n}$ in degree $d$ is dual to an irreducible $V_{\lambda'}\subseteq R_{1^n}$ in degree $\binom{n}{2} - d$. This duality on the character of $R_{1^n}$ is also easy to see from the cellular picture, as we will now show. However, there appears to be a stronger duality in the $\pW$-graph $\cinv$ which is surprisingly subtle.

\subsection{}
For a standard word $x = x_1\cdots x_n$, let $x^\dagger$ denote the word $x_n x_{n-1} \dots x_1$. For any $w \in \dsw$, let $x$ be the corresponding element of $W_f$ under the bijection $\dsw \cong W_f$ of Proposition \ref{p bijectionwfdsw}, i.e., $w_i = n \cl{x}_i + x_i$ for all  $i \in [n]$. Define the \emph{dual element} $\dual{w}$ to be the element of $\dsw$ that corresponds to $x^\dagger$ under the bijection $\dsw \cong W_f$, i.e., $\dual{w}_i = n \cl{(x^\dagger)}_i + x^\dagger_i$,  $i\in [n]$. Extend this notation to tableaux by defining $\dual{T}$ to be $P(\dual{w})$ for any (every) $w$ inserting to $T$.

From well-known properties of the insertion algorithm, the tableaux  $P(x)$ and $P(x^\dagger)$ are transposes of each other for any standard word $x$. Therefore, if $T$ is a PAT labeling a cell of $\cinv$, then  $T$ and $\dual{T}$ have shapes that are transposes of each other.  Let $\rho \in \pY_+$ be half the sum of the positive roots, i.e., $\rho = (n-1,n-2,\dots,0)$; note that $P(y^\rho) = G_{1^n}$. Then $\invlr(\dual{w}) w = w_0 y^\rho$ for any $w \in \dsw$. In particular, the sum of the degrees of $T$ and  $\dual{T}$ is  $\binom{n}{2}$.  Thus we have shown that  $\dual{}$ corresponds to a duality on the character of $\cinv$.

We have the following conjectural duality for the $\pW$-graph $\cinv$. The first part of this conjecture is proved below.
\begin{conjecture}
\label{cj flip}
For any $x,w \in \dsw$,
\begin{list}{\emph{(\alph{ctr})}} {\usecounter{ctr} \setlength{\itemsep}{1pt} \setlength{\topsep}{2pt}}
\item if $x = \pi w$, then $\dual{x} = \pi^{-1}\dual{w}$.
\item $\mu(x,w) = \mu(\dual{w},\dual{x})$ whenever $L(x)\cap S \not\subseteq L(w) \cap S$.
\end{list}
\end{conjecture}

Recall that $L(x)\cap S \not\subseteq L(w) \cap S$ if and only if the edge weight  $\mu(x,w)$ matters for the structure of $\cinv$ as an $\pH$-module, and therefore the main case we are interested in. This conjecture has been checked up to $n = 6$.

\begin{corollary} [of Conjecture \ref{cj flip}]
The csq $\csqa{\dual{G}_\lambda}{\dual{G}_{1^n}}$ is equal to
\begin{itemize}
\setlength{\itemsep}{1pt}\setlength{\topsep}{2pt}
\item $\{\Gamma_{\dual{T}} : \Gamma_T \in \R_\lambda \}$,
\item the minimal submodule of $\cinv$ containing $\Gamma_{\dual{G}_\lambda}$,
\item $(\fsp)^{-1}(\gpda{\dual{G}_\lambda}{\dual{G}_{1^n}})$, where $\gpda{\dual{G}_\lambda}{\dual{G}_{1^n}}$ is defined in terms of catabolizability in \textsection\ref{ss atom categories}.
\end{itemize}
\end{corollary}

One route to proving this conjecture is to exhibit a perfect pairing on $\cinv$ that respects canonical bases. This does not seem to work in a straightforward way, however the following approach seems promising.

For $h \in \pH$ write $[\C_w]h$ for the coefficient of $\C_w$ of $h$ written as an $A$-linear combination of $\{\C_w: w \in \pW\}$.  Define $\langle , \rangle: \pH \times \pH \to A$ by
\[ \langle h_1, h_2 \rangle = [\C_{w_0 y^\rho}]h_1 h_2. \]
Let $j: A \to A$ be the ring automorphism determined by $j(\u) = -\ui$, and also denote by $j$ the involution of $\eH$ given by $j(\sum_x a_x T_x) = \sum_x j(a_x) T_x$. The unprimed canonical basis element $C_w$, $w \in \eW$, is related to the primed $\C_w$ by $j(\C_w) = C_w$.

\begin{conjecture}
\label{cj pairing}
For $x \in \pW^S w_0$ and $w \in \invlr(\pW^S w_0)$,
\[ \langle C_x, \C_w \rangle = \begin{cases} 1 & \text{if } w \in \dsw,\ x \in \invlr(\dsw), \text{ and } \invlr(x) = \dual{w}, \\
0 & \text{otherwise.} \end{cases} \]
\end{conjecture}

As introduced in \textsection\ref{ss invlr}, there is an automorphism $\Delta$ of $\eW$ given on generators by $s_i \mapsto s_{n-i}$, $\pi \mapsto \pi^{-1}$.
\begin{proposition}
\label{p pairing}
Conjecture \ref{cj pairing} implies Conjecture \ref{cj flip}.
\end{proposition}
\begin{proof}[Proof of Proposition \ref{p pairing}]
It is easy to see that corotating and then applying $\dagger$ to a standard word is the same as applying $\dagger$ and then rotating. Part (a) then follows from the bijection $W_f \cong \dsw$ and the fact that this takes corotations to corotation-edges (see Proposition \ref{p ccp syt equals ccp cinv} and its proof).

For any $x \in \dsw$, the following are straightforward from the definitions of $\invlr$ and $\dual{}$:
\be \label{e leftright descents}
\begin{array}{rcl}
R(\invlr(x)) &=& \Delta(L(x)), \\
L(\dual{x}) \cap S &=& \{\Delta(s) : s \in S \backslash L(x)\}, \\
R(\invlr(\dual{x})) \cap S &=& \{s : s \in S \backslash L(x) \}.
\end{array} \ee
Suppose $x, w \in \dsw$, $L(x)\cap S \not\subseteq L(w) \cap S$, and let $s$ be any element of $S \cap L(x) \backslash L(w)$. We have
\begin{multline} \mu(\dual{w},\dual{x}) = \mu(\invlr(\dual{w}),\invlr(\dual{x})) = [C_{\invlr(\dual{w})}] C_{\invlr(\dual{x})} T_s \\=\langle C_{\invlr(\dual{x})} T_s, \C_w \rangle = \langle C_{\invlr(\dual{x})}, T_s \C_w \rangle = [\C_x] T_s \C_w = \mu(x,w).\end{multline}
The first equality holds because $\invlr$ is an anti-automorphism of the extended Coxeter group $\eW$. The third and fifth equalities use Conjecture \ref{cj pairing}. The last equality follows from $s \in S \cap L(x) \backslash L(w)$ and the definition of a $W$-graph (\ref{Wgrapheq}). Noting that (\ref{e leftright descents}) implies $s \in S \cap R(\invlr(\dual{w})) \backslash R(\invlr(\dual{x}))$, the second equality follows for the same reason.
\end{proof}

\section{Atoms}
\label{s atoms}

We are primarily interested in subquotients of the coinvariants $\cinv$, however it appears that there are many other copies of these subquotients in $\eH$, outside of $\cinv$.
The realization of an  $R \star W_f$-module $E$ as a cellular subquotient has genuinely different combinatorics depending on which element of ${(\fmod)}^{-1}(E)$ is chosen. It is reasonable to guess that two objects in $\cs$ are isomorphic if their images in Cocyclage Posets are connected and isomorphic, and this has been our empirical way of identifying copies of atoms.

One fundamental problem which we hope to make some steps towards in this section is to find an algorithm that takes a word  $w \in \eW$ as input and outputs two datum: one describing which atom copy  $w$ belongs to, and the other describing where $w$ sits inside this copy.  Our model for such an algorithm is the RSK algorithm which takes a word $w \in W_f$ and outputs  $P(w)$, which encodes which left  $W_f$-cell  $w$ belongs to, and  $Q(w)$, which encodes where $w$ sits inside this cell. The sign insertion algorithm we present in \textsection\ref{ss sign insertion} is our best attempt towards this goal, but has the serious flaw described in Remark \ref{r sign insertion fail}.  Also see Remark \ref{r sign insertion fail} for why such an algorithm would be so useful.
%that describes an atom, but without reference to a specific set of tableaux, just as a Knuth equivalence graph encodes a representation of  $\H$ without realizing it as the set of words with a fixed insertion tableau.

Despite the flaw of the sign insertion algorithm, it is good enough to allow us to state a conjecture about how $\eH$ decomposes into cellular subquotients isomorphic to the dual GP csq  $\gpda{\dual{G}_\lambda}{\dual{G}_{1^n}}$.  We discuss in \textsection\ref{ss cells in We} how this is closely related to the combinatorics of the $\eW$-cells of $\eW$ worked out by Shi, Lusztig, and Xi \cite{Shi, L two-sided An, X2}. Also of interest in this section is a new interpretation of charge for semistandard tableaux (proven for partition content, conjectural in general); see \textsection\ref{ss SSYT in PAT}.
\subsection{}
\label{ss atoms 1}
There are several examples in the literature of identifying cocyclage posets of different sets of tableaux. One example is Lascoux's \emph{standardization map} from tableaux of content $\lambda \vdash n$ to standard tableaux, which has image $\{T: \ctype(T) \gd \lambda\}$ \cite{La} (see also \cite[\textsection{4}]{SW}).  Since $\ccp(\Tab(\lambda))$ and $\ccp(SYT)$ are connected and contain a unique tableau of shape $(n)$, an embedding of cocyclage posets  $\ccp(\Tab(\lambda)) \to \ccp(SYT)$ is unique if it exists.  That is, the standardization map can be computed by taking any subposet of $\ccp(\Tab(\lambda))$ whose underlying undirected graph is a tree. The map then takes  the single-row tableau to the single-row tableau and
%or $Z_\lambda$, the superstandard tableau of shape and content $\lambda$, is mapped to the standard tableau $Z^*_\lambda$ (defined in \textsection\ref{ss catabolizability 2}).
is computed on the other tableaux by forcing it to be a color preserving embedding of cocyclage posets.  Miraculously, the map is shape-preserving and preserves all cocyclages, not only those in the subposet.

Another example of this are the copies of super atoms of Lascoux, Lapointe, and Morse \cite{LLM}. A super atom $\AA_\lambda^{(k)}$ and some connected subposet of its cocyclage poset is given.  A selected tableau of the super atom (usually the largest cocharge) is mapped to some other tableau of the same shape, and one tries to extend this map to the entire super atom as above. Then, again, miraculously, it appears that if the map extends, then it is an isomorphism in Cocyclage Posets.

There appear to be many more instances of this, and this has been our empirical way of finding atoms in $\eH$ that might be isomorphic to those that occur in the coinvariants. We will see that both of the examples above are special cases of the atom copies studied in this section.

\subsection{}
Before stating our conjectures about atom copies, we recall some basic facts about star operations and apply them to the study of $\cs$.  An easy way to show that two  $\pH$-cellular subquotients of $\eH$ are isomorphic is to show that they correspond under some sequence of right star operations.

For the following definitions, let $(W,S)$ be an (extended) Coxeter group. Let $s$ and $t$ be in $S$ such that $st$ has order 3. Define
\[ D_L(s,t) = \{ w \in W : |L(w) \cap \{s, t\}| = 1 \}, \]
\[ D_R(s,t) = \{ w \in W : |R(w) \cap \{s, t\}| = 1 \}. \]
The \emph{left star operation with respect to }$\{s,t\}$ is the involution $D_L(s,t) \to D_L(s,t)$, $w \mapsto \leftexp{*}w$, where $\leftexp{*}w$ is the single element of $D_L(s,t) \cap \{sw,tw\}$. Similarly, the \emph{right star operation with respect to} $\{s,t\}$ is the involution $D_R(s,t) \to D_R(s,t)$, $w \mapsto w^*$, where $w^*$ is the single element of $D_R(s,t) \cap \{ws,wt\}$. We use the convention of writing $* = \{s,t\}$ to signify that the star operation is with respect to $\{s,t\}$. We will need the following results from the original Kazhdan-Lusztig paper, the second of which is quite crucial to the theory that originated there.

\begin{proposition}[{\cite[Proposition 2.4]{KL}}]
\label{p right descents constant}
Suppose $x, w \in \Gamma_W$.
\begin{list}{\emph{(\roman{ctr})}} {\usecounter{ctr} \setlength{\itemsep}{1pt} \setlength{\topsep}{2pt}}
\item If $x$ and $w$ belong to the same left cell, then $R(x) = R(w)$.
\item If $x$ and $w$ belong to the same right cell, then $L(x) = L(w)$.
\end{list}
\end{proposition}

\begin{theorem}[{\cite[Theorem 4.2]{KL}}]
\label{t star operations kl}
With the convention of Remark \ref{r wgraph symmetric}, $* = \{s,t\} \subseteq S$, and $st$ of order 3,
\begin{list}{\emph{(\roman{ctr})}} {\usecounter{ctr} \setlength{\itemsep}{1pt} \setlength{\topsep}{2pt}}
\item if $x, w \in D_L(s,t)$, then $\mu(x,w) = \mu(\leftexp{*}{x},\leftexp{*}{w})$,
\item if $x, w \in D_R(s,t)$, then $\mu(x,w) = \mu(x^*,w^*)$.
\end{list}
\end{theorem}

%\begin{proposition}[{\cite[Lemma 1.4.3]{X2}}]
%\label{p xi star ops}
%Suppose $* = \{s,t\} \subseteq S$ and $\star = \{s', t'\} \subseteq S$ with $st$ and $s't'$ of order 3.  If $w \in D_L(s,t) \cap D_R(s',t')$, then
%\begin{list}{\emph{(\roman{ctr})}} {\usecounter{ctr} \setlength{\itemsep}{1pt} \setlength{\topsep}{2pt}}
%\item $\leftexp{*}{w} \in D_R(s',t')$ and $w^\star \in D_L(s,t)$,
%\item $\leftexp{*}{(w^\star)} = (\leftexp{*}{w})^\star$, so that we may define $\leftexp{*}{w}^\star$ to be $\leftexp{*}{(w^\star)} = (\leftexp{*}{w})^\star$.
%\end{list}
%\end{proposition}

For $(W,S) = (\eW, K)$ and elements identified with affine words, left and right star operations are Knuth transformations and dual Knuth transformations ($\tto_{KT}$ and $\tto_{DKT}$). See \cite[A1]{St} for an introduction to this combinatorics in the case $W = \S_n$. Knuth transformations look like those for standard words. A Knuth transformation of an affine word $w \in \eW$ is a transformation of one of the following forms:
\be \cdots 1.b\ 1.a\ 1.c \cdots b\ a\ c \cdots \tleftright_{KT} \cdots 1.b\ 1.c\ 1.a \cdots b\ c\ a \cdots, \ee
\be \cdots 1.a\ 1.c\ 1.b\cdots a\ c\ b\cdots\tleftright_{KT} \cdots 1.c\ 1.a \ 1.b \cdots c\ a\ b\cdots, \ee
for $a, b, c\in \ZZ$, $a < b < c$. These pictures are to be interpreted to mean that for every $k \in \ZZ$, the adjacent numbers $k.a$ and $k.c$ are transposed. These Knuth transformations correspond to the left star operation with respect to $\{s_i,s_{i+1}\}$ (subscripts taken mod $n$), where, for the first line, $i$ is the position of $b$, and, for the second line, $i$ is the position of the $a$ on the left-hand side.

To see a dual Knuth transformation of an affine word $w$, it is not enough to only examine the window $w_1 \dots w_n$. A dual Knuth transformation of an affine word $w \in \eW$ is a transformation of one of the following forms:
\be \cdots i \cdots i+2 \cdots i+1 \cdots \tleftright_{DKT} \cdots i+1 \cdots i+2 \cdots i \cdots, \ee
\be \cdots i+1 \cdots i \cdots i+2 \cdots \tleftright_{DKT} \cdots i+2 \cdots i \cdots i+1 \cdots. \ee
These pictures are to be interpreted to mean that a similar transformation is performed on the numbers $k.i$, $k.i+1$, $k.i+2$ for every $k \in \ZZ$. These dual Knuth transformations correspond to the right star operation with respect to $\{s_{n-i}, s_{n-(i+1)}\}$ (recall the unusual convention of (\ref{e wordmult5})).

For the remainder of this paper, we will understand Knuth transformations (resp. dual Knuth transformations) to be left (resp. right) star operations for $* \subseteq S$ rather than $* \subseteq K$. This is more natural given our focus on $\pW$-cells rather than $\eW$-cells, and it is not a significant change because any star operation is equivalent to one with $* \subseteq S$ via conjugation by some power of $\pi$.
\begin{example}
For $n = 5$, the following are examples of a Knuth transformation corresponding to the left star operation with respect to $\{s_1, s_2\}$ and a dual Knuth transformation corresponding to the right star operation with respect to $\{s_1,s_2\}$ (recall the unusual convention of (\ref{e wordmult5})).
\[\begin{array}{ccccccccccc}
13& 1& 42&14&5 &\tto_{KT}& 13&42&1&14&5 \\
13&1&42&14&5 & \tto_{DKT}& 13&1&42&15&4 \end{array} \]
\end{example}

The next proposition relates connectivity of ccp to the left $\eW$-cells of $\eW$. See \textsection\ref{ss cells in We} for more on the $\eW$-cells of $\eW$. We remark that the left $\eW$-cells of $\eW$ are typically infinite in contrast to the $\pW$-cells of $\eW$ and the $W_f$-cells of $W_f$. For a subset $\Gamma$ of  $\eW$, let $\widehat{\Gamma}$ be the minimal element of  $\cs$ containing $\Gamma$.
\begin{proposition}
\label{p connectivity left cell eW}
If $\Gamma$ is a union of left $\pW$-cells of $\eW$ such that the undirected graph on  $\Gamma$ consisting of cocyclage-edges is connected, then  $\Gamma$ and, stronger, $\widehat{\Gamma}$ are contained in a left $\eW$-cell of  $\eW$.
\end{proposition}
\begin{proof}
 The connectivity assumption and our knowledge of the left $\pW$-cells of $\eH$ (Corollary \ref{c restrict cells2} and \cite[Theorem 1.4]{KL}) show that the undirected graph on $\Gamma$ consisting of Knuth transformations and corotation-edges is connected. Thus $\Gamma$ is contained in a left $\eW$-cell $\Lambda$ of $\eW$. Here we are using that $\pi w$ and $w$ are contained in the same left $\eW$-cell of  $\eW$ for any  $w \in \eW$. Since left $\eW$-cells are $\eH$-cellular subquotients, they are also $\pH$-cellular subquotients. Thus $\widehat{\Gamma}$ is contained in $\Lambda$.
\end{proof}

\begin{proposition}
\label{p star ops give isomorphism}
Suppose $A \Gamma \in \cs$ and  $\Gamma$ is contained in a left  $\eW$-cell  $\Lambda$ of  $\eW$ (which holds, for instance, if $\fxccp(A \Gamma)$ is connected). Let $* = \{s,t\} \subseteq S$ with $st$ of order 3. Then
\begin{list}{\emph{(\roman{ctr})}} {\usecounter{ctr} \setlength{\itemsep}{1pt} \setlength{\topsep}{2pt}}
\item if $\gamma \in D_R(s,t)$ for some $\gamma \in \Gamma$, then $\Gamma \subset D_R(s,t)$,
\item If $\Gamma \subseteq D_R(s,t)$, then  $A \Gamma^* \in \cs$ and $\Gamma \cong \Gamma^* := \{\gamma^*: \gamma \in \Gamma\}$,
\item $A \Gamma \pi \in \cs$ and $\Gamma \cong \Gamma \pi := \{\gamma \pi : \gamma \in \Gamma\}$.
\end{list}
\end{proposition}

\begin{proof}
The assumption $\Gamma \subseteq \Lambda$ and \ref{p right descents constant}(i) imply $R(\gamma) = R(\gamma')$ for all $\gamma, \gamma' \in \Gamma$, hence (i).

For (ii), note that by Theorem \ref{t star operations kl} (ii) and Proposition \ref{p right descents constant} (ii) the edges  $\mu$ with both ends in  $\Gamma$ are the same as those of $\Gamma^*$, so it remains to show $A \Gamma^* \in \cs$. The previous sentence also shows that the edges  $\mu$ in $\widehat{\Gamma^*}$ are the same as those of $\widehat{\Gamma^*}^*$ since $\widehat{\Gamma^*} \subseteq D_R(s,t)$ by the assumption $\Gamma \subseteq \Lambda$ and \ref{p right descents constant}(i).  Then if  $\Upsilon_3 \klo{\pH} \Upsilon_2 \klo{\pH} \Upsilon_1$ for $\Upsilon_2 \subseteq \widehat{\Gamma^*} \backslash \Gamma^*$,  $\Upsilon_1, \Upsilon_3 \in \Gamma^*$ for some left cells $\Upsilon_i$ of  $\pH$, we would have $\Upsilon^*_3 \klo{\pH} \Upsilon^*_2 \klo{\pH} \Upsilon^*_1$ for $\Upsilon^*_2 \subseteq \widehat{\Gamma^*}^* \backslash \Gamma$,  $\Upsilon^*_1, \Upsilon^*_3 \in \Gamma$, contradiction.  Thus  $\widehat{\Gamma^*} = \Gamma^*$.

Statement (iii) is immediate from the identity $\mu(x,w) = \mu(x\pi,w\pi)$ for all $x, \pi \in \eW$.
\end{proof}

\begin{example}
\label{ex dke}
Let $\Gamma$ be the cellular subquotient on the left-hand side of the bottom row of Figure \ref{f atomcopies} (see \textsection\ref{ss chen atom}).
The cellular subquotient on the right of the bottom row is equal to $(\Gamma \pi)^{*_1*_2}$, where $*_1 = \{s_{n-1},s_{n-2}\}$, $*_2 =\{s_{n-2},s_{n-3}\}$ ($n = 6$).
\end{example}

\subsection{}
\label{ss sign insertion}
From any word $w \in \dsw$ there is a  path of Knuth transformations and rotation-edges to the word $w_0 = \reading(G_{(n)})$, which is a left cell corresponding to the trivial $\H$-representation. Also, there is a path consisting of Knuth transformations and corotation-edges from any $w \in  \dsw$ to the word $y^\rho = \reading(G_{1^n})$, which is a left cell corresponding to the sign $\H$-representation.  For an arbitrary $w \in \eW$, there does not exist such a path to a trivial representation, however there is always a path to a sign representation.  This leads to the following algorithm,
which enables us to state conjectures about how $\eH$ decomposes into cellular subquotients that are isomorphic to $\csqa{\dual{G}_\lambda}{\dual{G}_{1^n}}$ and how this relates to the $\eW$-cells of $\eW$.  This algorithm is an adaptation of a result of Shi on iterated star operations \cite[Lemma 9.2.1]{Shi}.

\begin{algorithm}\label{a sign insertion}
The \emph{sign insertion algorithm} takes as input an affine word $w \in \eS_n$ and outputs a pair $(\sgnp(w), \sgnq(w))$ of tableaux of shape $1^{n}$. The tableau $\sgnp(w)$ is an affine tableau, and $\sgnq(w)$ is a semistandard tableau with distinct integer entries; these entries are thought of as ``ordinary integers'' and will not be written in our base-n notation $a.b$.

We package the algorithm as a function $f$ which takes a 4-tuple  $(c, x, P, Q)$ to another such 4-tuple, where  $c$ is an integer counter,  $x$ is part of an affine word, and  $P$ and $Q$ are single-column tableaux of the same shape.   Write $x = za$ for $z$ a word and $a$ a number. Let $P' := a \to P$ be the result of column-inserting $a$ into $P$. If $a \to P$ has more than one column, let $a'$ be the entry in the second column of $P'$. Recall from \textsection \ref{ss catabolizability definition} that for a tableau $T$, $T_\lambda$ is the subtableau of $T$ obtained by restricting  $T$ to the diagram of $\lambda$. Let $Q'$ be the tableau obtained by appending the entry $c$ to the bottom of $Q$. Then
\be \label{e f}
f(c-1, x, P, Q) = \begin{cases}
(c, z, P', Q') &\text{if $P'$ has one column},\\
(c, 1.a' z, P'_{1^{|P|}}, Q) &\text{otherwise}.
\end{cases}
\ee

The sign insertion algorithm repeatedly applies $f$ to the tuple $(0, w, \emptyset, \emptyset)$. It terminates when the word of the tuple is empty and outputs the pair of tableaux.

The transition from $(c, x, P, Q) := f^{(c-1)}(0, w, \emptyset, \emptyset)$ to $f(c, x, P, Q)$ is the $c$-th \emph{step} of the algorithm.
%In this step we say that $a$ is \emph{presented} to $P$, and in the top case of (\ref{e f}), $a$ is \emph{inserted} into $P$, while in the bottom case, $a'$ is \emph{corotated}.
A step of the algorithm in the top case of (\ref{e f}) is an \emph{insertion step}, and in the bottom a  \emph{corotation step}.
\end{algorithm}

\begin{example}
\label{ex sign insertion}
The sign insertion algorithm applied to $w = 36\ 13\ 32\ 4\ 21\ 25$ produces the tuples shown in Figure \ref{f sign insert}. For any $v \in \dsw$, $\sgnp(v) = P(y^\rho) = \dual{G}_{1^n}$. The element $v := 36\ 12\ 24\ 1\ 13\ 25$ is the unique element of $\dsw$ such that $\sgnq(v) = \sgnq(w)$ for $w$ as above.
\begin{figure}
\setlength{\cellsize}{12pt}
\[  \Yboxdim9pt \small\begin{array}{rrcccrrcc}
\multicolumn{4}{c}{f^{(c)}(0,w,\emptyset,\emptyset)} & \qquad\qquad& \multicolumn{4}{c}{f^{(c)}(0,w,\emptyset,\emptyset)} \\ \\
0\quad & 36\ 13\ 32\ 4\ 21\ 25 & {\tiny\emptyset} & {\tiny\emptyset} &\qquad\qquad&7\quad& 46\ 42\ 31 & \tiny\tableau{4 \\ 13 \\ 35}  & \tiny\tableau{1 \\ 4 \\ 6} \\ \\
1\quad  & 36\ 13\ 32\ 4\ 21 & \tiny\tableau{25}  & \tiny\tableau{1} &\qquad\qquad&8\quad& 45\ 46\ 42 & \tiny\tableau{4 \\ 13 \\ 31} & \tiny\tableau{1 \\ 4 \\ 6} \\ \\
2\quad & 35\ 36\ 13\ 32\ 4\ & \tiny\tableau{21}  & \tiny\tableau{1} &\qquad\qquad&9\quad& 45\ 46 & \tiny\tableau{4 \\ 13 \\ 31 \\ 42} & \tiny\tableau{1 \\ 4 \\ 6 \\ 9}\\ \\
3\quad & 31\ 35\ 36\ 13\ 32\ & \tiny\tableau{4}  & \tiny\tableau{1} &\qquad\qquad&10\quad& 45 & \tiny\tableau{4 \\ 13 \\ 31 \\ 42 \\ 46} & \tiny\tableau{1 \\ 4 \\ 6 \\ 9 \\ 10} \\ \\
4\quad & 31\ 35\ 36\ 13\ & \tiny\tableau{4 \\ 32}  & \tiny\tableau{1 \\ 4} &\qquad\qquad&11\quad& 56 & \tiny\tableau{4 \\ 13 \\ 31 \\ 42 \\ 45} & \tiny\tableau{1 \\ 4 \\ 6 \\ 9 \\ 10}\\ \\
5\quad & 42\ 31\ 35\ 36\ & \tiny\tableau{4 \\ 13}  & \tiny\tableau{1 \\ 4} &\qquad\qquad&12\quad& \emptyset & \tiny\tableau{4 \\ 13 \\ 31 \\ 42 \\ 45 \\ 56} & \tiny\tableau{1 \\ 4 \\ 6 \\ 9 \\ 10 \\ 12}\\ \\
6\quad & 42\ 31\ 35\ & \tiny\tableau{4 \\ 13 \\ 36}  & \tiny\tableau{1 \\ 4 \\ 6} \\
\end{array} \]
\caption{The sign insertion algorithm applied to $w = 36\ 13\ 32\ 4\ 21\ 25$.}
\label{f sign insert}
\end{figure}
\end{example}

To ease notation, we will write $P$ in place of $\reading(P)$ for the tableaux in this algorithm. This compromises little as these are all single-column tableaux. Thus $P_i$ will denote the entry in row $|P|+1-i$ of $P$.   In the next proposition and its proof, we will write $(c-1, x, P, Q)$ and $(c, x', P',Q') := f(c-1, x, P, Q)$ for the  $4$-tuple before and after the $c$-th step.

\begin{proposition}\label{p sign insertion algorithm facts}
The sign insertion algorithm with input $w \in \eS_n$ satisfies:
\begin{list}{\emph{(\alph{ctr})}} {\usecounter{ctr} \setlength{\itemsep}{1pt} \setlength{\topsep}{2pt}}
\item after every step, the pair $(x,P)$ is such that $x P$ is an affine word,
\item terminates,
\item  $\sgnp(w) = \rJ{(y^\lambda w)}{S}$ for some $\lambda \in \pY$,
\item  If $k \in [n]$ is maximal such that $w_k < w_{k+1}$ and $w'$ is the word $x P$ after the first corotation step, then  $\sgnp(w)=\sgnp(w')$ and
    \[\sgnq(w') = \sgnq(w)_1 -1 \ \sgnq(w)_2 -1 \ldots \sgnq(w)_{k} -1\ n-k\ n-k-1 \ldots\ 1,\]
\item  $\sgnp(w) = \sgnp(w')$ if  $w' = \pi w$ and $w_{n-1} < w_n$,
\item  $\sgnp(w) = \sgnp(\leftexp{*}{w})$ if $w \tto \leftexp{*}{w}$ is a Knuth transformation,
\item If $w \tto \leftexp{*}{w}$ is a Knuth transformation,
then the tuple $(|\sgnq(w)_{i} - \sgnq(\leftexp{*}{w})_{i}|)_{i \in [n]}$ consists of $n-1$ $0$'s and one $1$,
\item If $w \tto w^*$ is a dual Knuth transformation, then $\sgnp(w^*) = (\sgnp(w))^*$ and $\sgnq(w) = \sgnq(w^*)$,
\item If  $w' = w \pi$, then  $\sgnp(w') = \sgnp(w) \pi$ and  $\sgnq(w) = \sgnq(w')$,
\item Suppose $P_{1^k}$ consists of the $k$ smallest numbers of $\sgnp(w)$. Then  $(\sgnp(w),\sgnq(w))$ can be computed by removing $P_{1^k}$ and $Q_{1^k}$ from $P$ and $Q$, running the algorithm with its usual rules, and then adding $P_{1^k}$ and $Q_{1^k}$ back to the bottom of the final insertion and recording tableaux.
\end{list}
\end{proposition}

\begin{proof}
For a tableau $T$ and number $a$,  $a\ \reading(T)$ is Knuth equivalent to  $\reading(a \to T)$.  Thus in a step of the algorithm, $x' P'$ is obtained from $x P$ by a sequence of Knuth transformations followed by, in the case of a corotation step, left multiplication by $\pi$. This proves (a).

In each insertion step of the algorithm, $|P|$ is increased by one. The algorithm terminates after $n$ insertion steps, so to prove (b) it suffices to show that there cannot be infinitely many corotation steps in succession.  This is so because if there are $|x|$ corotation steps in succession after the $c-1$-th step, the largest value of the word increases by at least $n+1$, and the largest entry of $P$ decreases or stays the same.

%After an insertion step of the algorithm, $x\ \reading{P} = x^*\ \reading{P^*}$. Let $x\ \reading{P} = y^\lambda v, v \in W_f$ while after a corotation step, \rJ{(y^\lambda w)}{S}$
Each entry $e = a.\rsd{e}$ of $\sgnp(w)$  $(a \in \ZZ)$ is congruent to some entry $a'.\rsd{e}$ of $w$. Since $\sgnp(w)$ is obtained from $w$ by a sequence of left-multiplications by $s \in S$ and left-multiplications by $\pi$, $a \geq a'$.  This is a rephrasing of statement (c).

Let $k$ be as in (d). The first $n-k$ steps of the algorithm are insertion steps. The next step is a corotation step.  The algorithm run on  $w'$ is nearly identical to the algorithm run on input $w$: the first $n-k$ steps are insertion steps, and for  $c > n-k$, the  $c$-th step is the same as the $c+1$-th step of the algorithm run on  $w$. This proves (d).

Statement (e) is a special case of (d).

To show (f), we prove the slightly stronger statement that, after any step of the algorithm, the word $x$ can be replaced by $\leftexp{*}{x}$ with $x \tto \leftexp{*}{x}$ a Knuth transformation without changing the final insertion tableau. We may assume that $x \tto \leftexp{*}{x}$ is a Knuth transformation in the last three numbers of $x$ and is written as
\be \cdots edf \tto \cdots efd \ee
with $d < e < f$ (we must also check the case $\cdots dfe \tto \cdots fde$, but this is similar).  Let $f'$ (resp. $d'$) be the smallest entry of $P$ greater than  $f$ (resp. $d$) if such an entry exists.  One checks that the insertion tableaux of $f^{(3)}(c-1,x,P,Q)$ and $f^{(3)}(c-1,\leftexp{*}{x},P,Q)$ are always the same. Now if $f'$ is defined (and therefore so is $d'$) and $d' < f$, then the result follows by induction from the Knuth transformation
\be e'+n\ d'+n\ f'+n\cdots\tto e'+n\ f'+n\ d'+n \cdots, \ee
where $d'$, $e'$, and $f'$ are the entries kicked out in the column-insertions of $d$, $e$, and $f$. On the other hand, if $f'$ is defined, $d' = f'$, and there is an entry $f''$ of  $P$ in the row below $f'$, then the result follows by induction from the Knuth transformation \be f''+n\ f+n\ f'+n\cdots\tto f+n\ f''+n\ f'+n \cdots. \ee  Otherwise, one checks that  the words of $f^{(3)}(c-1,x,P,Q)$ and $f^{(3)}(c-1,\leftexp{*}{x},P,Q)$ are the same.

Statement (g) is proved in a similar way to (f).  If $f'$ is defined and either  ($d' < f$) or ($d'=f'$ and  $f''$ is defined), then the $c,c+1$, and $c+2$-th steps are corotation steps and do not change the recording tableau, so the result follows by induction.  Otherwise let $U$,  $U'$ be the recording tableaux of $f^{(3)}(c-1,x,P,Q)$ and $f^{(3)}(c-1,\leftexp{*}{x},P,Q)$ respectively.
Then we have
\[ \begin{array}{ll}
|U| = |U'| = |Q| +1 \text{, }& \\
\quad U_1=c \text{, and } U'_1=c+1 &\text{if $d'$ is defined and $f'$ is not,}\\
|U| = |U'| = |Q| +2 \text{, }& \\
\quad U_1=c+2, U_2=c \text{, and } U'_1=c+1, U'_2=c &\text{if $f'$ and $d'$ are not defined,}\\
|U| = |U'| = |Q| +1 \text{, }& \\
\quad U_1=c+2 \text{, and } U'_1=c+1 &\text{if $d'=f'$ are defined and $f''$ is not}.
\end{array} \]
By the end of the previous paragraph, the entries added to the recording tableau $U$ in the remainder of the algorithm are the same as those added to $U'$.

For statements (h) and (i), we want to think of the sign insertion algorithm as producing the sequence of words  (as in (a)) obtained by concatenating the word and insertion tableau of $f^{(c-1)}(0,w,\emptyset,\emptyset)$ for each $c$.  The word after step $c$ is obtained from the word before step $c$ by some Knuth transformations and possibly a corotation, and whether or not the step is a corotation depends only on  the left descent set of the word. Then let $\Gamma$ be the subset of $\eW$ consisting of these words and $\widehat{\Gamma}$ be the minimal element of $\cs$ containing $\Gamma$. By Proposition \ref{p star ops give isomorphism}, $\widehat{\Gamma} \cong \widehat{\Gamma}^*$ and $\widehat{\Gamma} \cong \widehat{\Gamma} \pi$. Thus the sequence of words for the algorithm run on $w^*$ is obtained from the sequence of words for the algorithm run on $w$ by applying the right star operation with respect to $*$.  Statement (h) then follows and (i) follows in a similar way.

For statement (j), note that by (c),  $P_{1^k} = P'_{1^k} = \sgnp(w)_{1^k}$ for  $P'$ the insertion tableau after the $c'$-th step for any $c' \geq c$. The result is then clear because the entries of $P_{1^k}$ are never bumped out in a column-insertion, so have no effect on the algorithm.
\end{proof}

Part (f) of this proposition allows us to define, for an AT $T$, $\sgnp(T)$ to be $\sgnp(w)$ for any (every) $w$ inserting to $T$.

\begin{definition}
For an AT  $P$ of shape $1^n$, define $\sgnpa{P}$ to be the ccp on the set of tableaux $\{ T \in \text{ AT } : \sgnp(T) = P\}$.
\end{definition}
For example, $\sgnpa{G_{1^n}} = \fsp(\cinv)$.

It is not hard to see that the word $w$ can be recovered from the pair $(\sgnp(w), \sgnq(w))$ by running the sign insertion algorithm in reverse. In more detail, this is done by keeping track of a 4-tuple $(c, x', P', Q')$ as in the sign insertion algorithm. Let $x' = a' z'$, with $a'$ a number and $z'$ a word.  The function $f^{-1}$ is given by
\be \label{e f inverse}
f^{-1}(c, x', P', Q') = \begin{cases}
(c-1, x' P'_1, P'_{1^{|P'|-1}}, Q'_{1^{|Q'|-1}}) &\text{if $c = Q'_1$},\\
(c-1, z' a, P, Q') &\text{otherwise},
\end{cases}
\ee
where  $P$ is obtained by adding  $\ng 1.a'$ to $P'$ in the first row and second column and then column-uninserting it and $a$ is the number bumped out in this column-uninsertion.

Given an AT $P$ and a tableau $Q$ of shape $1^n$ with distinct entries in the positive integers, it is difficult to determine whether the reverse sign insertion algorithm can be run on this pair (adding $\ng 1.a'$ to  $P'$ might not result in a tableau). In the next subsection we give a conjectural description, though not a completely explicit one, of such pairs.

\begin{definition}
Define $\sgnrec{\lambda}$ to be the set $\{\sgnq(w) : w \in \gpda{\dual{G}_\lambda}{\dual{G}_{1^n}}\}$.
\end{definition}

\begin{example}
 The set $\sgnrec{( 2, 2 )}$ is equal to
\[ \left\{ {\Yboxdim10pt \tiny \young(1,3,5,6)}\ , {\Yboxdim10pt \tiny \young(1,2,5,6)}\ , {\Yboxdim10pt \tiny \young(1,3,4,5)}\ , {\Yboxdim10pt \tiny \young(1,2,3,5)}\ ,{\Yboxdim10pt \tiny \young(1,2,4,5)}\ ,{\Yboxdim10pt \tiny \young(1,2,3,4)} \right\} \ .\]
\end{example}

\begin{remark}
\label{r sign insertion fail}
The sign insertion algorithm is our best attempt at an algorithm meeting the requirements described in the introduction to this section. However, its main shortcoming is that it seems remarkably difficult to give a reasonable description of the sets $\sgnrec{\lambda}$ even for $\lambda = (n)$. We spent some time just working on the case where $\lambda$ has two columns, without much success.
\end{remark}

\subsection{} \label{ss cells in We}
Using the sign insertion algorithm, we are able to give conjectural descriptions of how $\eH$ decomposes into  $\pH$-cellular subquotients that are isomorphic to $\csqa{\dual{G}_\lambda}{\dual{G}_{1^n}}$. This combinatorics is closely related to that needed to understand the $\eW$-cells of $\eW$. A description of the $\eW$-cells of $\eW$ was conjectured by Lusztig and proved by Shi \cite{L two-sided An, Shi}, and further properties were worked out by Xi \cite{X2}. We now recall the description of the two-sided cells, using \cite{X2} as our main reference.

Given an element $w$ of $\eW$, define $\wposet(w)$ to be the poset on $[n]$ with relations $i \prec j$ if ($i < j$ and $w_i < w_{j}$) or ($i > j$ and $w_i < w_{j+n}$). Greene's theorem associates to any poset $\wposet$ a partition $\nu \vdash n$, denoted $\partition(\wposet)$.
\begin{theorem}[Shi \cite{Shi}, Lusztig \cite{L two-sided An}]
The sets \[\lrcell := \{w \in \eW : \partition(\wposet(w)) = \nu\}\] are the two-sided $\eW$-cells of $\eW$.
\end{theorem}

Since $\partition(\wposet(w))$ is constant on Knuth equivalence classes, we may define $\partition(T)$ to be $\partition(\wposet(w))$ for any (every) $w$ inserting to $T$.

\begin{definition}
\label{d dual gp ccp copy}
Suppose that $U$ is an AT such that $\sgnq(\reading(U)) = \sgnq(\reading(\dual{G}_\lambda))$. The dual GP ccp copy $\gpda{U}{\sgnp(U)}$ is the ccp on the set of tableaux
    \[\{P(w) : w \in \eW \text{ such that } \sgnp(w) = \sgnp(U) \text{ and } \sgnq(w) \in \sgnrec{\lambda}\}.\]
\end{definition}
For special $U$, this definition will be superseded by another more explicit description of the ccp (Definition \ref{d ccp copy stackable}). These definitions are conjecturally equivalent.  See Example \ref{ex dual gp ccp} for an example.

\begin{conjecture}
\label{cj gp dual copies}
Suppose that $U$ is an AT such that $\sgnq(\reading(U)) = \sgnq(\reading(\dual{G}_\lambda))$. Then
\begin{list}{\emph{(\alph{ctr})}} {\usecounter{ctr} \setlength{\itemsep}{1pt} \setlength{\topsep}{2pt}}
\item  $\gpda{U}{\sgnp{U}} \cong \gpda{\dual{G}_\lambda}{\dual{G}_{1^n}}$ in CCP.
\item $\fsp(\csqa{U}{\sgnp(U)}) = \gpda{U}{\sgnp(U)}$.
\item $\csqa{U}{\sgnp(U)} \cong \csqa{\dual{G}_\lambda}{\dual{G}_{1^n}}$ in  $\cs$.
\end{list}
\end{conjecture}

Note that the combinatorial conjecture (a) would follow from (b) and (c).  Our computer experimentation provides substantial evidence for (a), and our main reason for believing (c) is primarily this evidence as well.  Assuming (b), we checked using Magma that (c) holds in the special case that $U$ is a standard tableau and $n \leq 6$.  That is, we found an isomorphism $\alpha : \gpda{U}{\sgnp{U}} \to \gpda{\dual{G}_\lambda}{\dual{G}_{1^n}}$ in CCP and checked that the edge weights $\mu(\alpha(x),\alpha(w))$ and $\mu(x, w)$ agree whenever $L(x) \not \subseteq L(w)$. See Example \ref{ex easy stackable} for more about the special case in which $U$ is standard.

\begin{conjecture}
\label{cj eW-cells and dual GP copies}
\
\begin{list}{\emph{(\alph{ctr})}} {\usecounter{ctr} \setlength{\itemsep}{1pt} \setlength{\topsep}{2pt}}
\item If $P$ is a single-column tableau and $\Gamma_P \in \lrcell$, then $\sgnpa{P} = \gpda{U}{P}$ for a unique $U$ of shape $\nu$. Thus, in $\lrcell$, tableaux of shape $\nu$ are in bijection with tableaux of shape $1^n$ via $U \mapsto \sgnp(U)$.
\item The left $\eW$-cells of $\lrcell$ are in bijection with $\sgnrec{\nu}$ and are of the form
    \[ \Upsilon_Q := \{ w \in \lrcell : \sgnq(w^{-1}) = Q \},\ Q \in \sgnrec{\nu}. \]
\item The left cell $\Upsilon_Q$ decomposes into dual GP csq as
    \[ \Upsilon_Q = \bigsqcup_{\substack{\sgnq((\reading(U))^{-1}) = Q, \\ \sh(U) = \nu, \\ \partition(U) = \nu}} \csqa{U}{\sgnp(U)}. \]
\item This further gives the decomposition of $\lrcell$ into dual GP csq
    \[ \lrcell = \bigsqcup_{\substack{\sh(U) = \nu,\ \\ \partition(U) = \nu}} \csqa{U}{\sgnp(U)}.\]
\end{list}
\end{conjecture}

It is straightforward from the definition of $\lrcell$  that $\Gamma_P \in \lrcell$ implies $\sh(P) \ld \nu$, which is consistent with this conjecture.

\begin{proposition}
Conjecture \ref{cj gp dual copies} and Conjecture \ref{cj eW-cells and dual GP copies} (a) imply Conjecture \ref{cj eW-cells and dual GP copies}.
\end{proposition}
\begin{proof}
First note that (d) is an easy consequence of Conjecture \ref{cj eW-cells and dual GP copies} (b) and (c).
We know that for any AT $P$ of shape $1^{n}$, the ccp $\sgnpa{P}$ is contained in a left $\eW$-cell of $\eW$ (Proposition \ref{p connectivity left cell eW}).  Thus (c) follows from Conjecture \ref{cj eW-cells and dual GP copies} (a) and (b).

Now we prove (b) (actually its equivalent statement for right cells).
Proposition \ref{p connectivity left cell eW} and Conjecture \ref{cj eW-cells and dual GP copies} (a) imply that
\be \label{e two-sided cell decomposition}
\lrcell = \bigsqcup_{\substack{\sh(P) = 1^n,\ \\ \partition(P) = \nu}} \sgnpa{P} = \bigsqcup_{\substack{\sh(U) = \nu,\ \\ \partition(U) = \nu}} \gpda{U}{\sgnp(U)}.
\ee
Let
\[ f_{U U'} : \gpda{U}{\sgnp(U)} \to \gpda{U'}{\sgnp(U')} \]
be the isomorphisms given by Conjecture \ref{cj gp dual copies} (a) for all $U,U'$ of shape $\nu$ with $\partition(U) = \partition(U') = \nu$. By a similar argument to the proof of Proposition \ref{p sign insertion algorithm facts} (h) and (i), $\sgnq(w) = \sgnq(f_{U U'}(w))$ for all $w \in \gpda{U}{\sgnp(U)}$.

Let $Z_{1^n}^*$ be the standard tableau of shape $1^n$. It is known that $(\Upsilon_{Z_{1^n}^*})^{-1}$ is a right $\eW$-cell of $\eW$, called the canonical right cell (denoted $\Phi_{\lrcell}$ in \cite{X2}). It is also known that applying a sequence of left star operations with $* \subseteq S$ and left multiplications by $\pi$ to a right cell results in a right cell. It follows that $w$ and $f_{U U'}(w)$ belong to the same right cell for any $w \in \gpda{U}{\sgnp(U)}$. Thus by the previous paragraph, $(\Upsilon_Q)^{-1}$ is contained in a right cell.

Finally, it is known that there are $\binom{n}{\nu'_1,\nu'_2,\dots,\nu'_{\nu_1}}$ right cells in $\lrcell$ \cite{Shi}. This  is also the cardinality of $\sgnrec{\nu}$, which follows, for instance, by the standardization map. Thus since $\bigsqcup_{Q \in \sgnrec{\nu}} (\Upsilon_Q)^{-1} = \lrcell$ by (\ref{e two-sided cell decomposition}), $(\Upsilon_Q)^{-1}$ is exactly a right cell.
\end{proof}

Since $\csqa{\dual{G}_\lambda}{\dual{G}_{1^n}}$ is a cellular submodule of $\csqa{\dual{G}_\nu}{\dual{G}_{1^n}}$ for $\lambda \ld \nu$, Conjectures \ref{cj gp dual copies} and \ref{cj eW-cells and dual GP copies} would give a complete description of dual GP csq copies:

\begin{corollary} [of Conjectures \ref{cj gp dual copies} and \ref{cj eW-cells and dual GP copies}]
The copies of the csq $\csqa{\dual{G}_\lambda}{\dual{G}_{1^n}}$ are of the form $\csqa{U}{\sgnp(U)}$
for $U$ as in Definition \ref{d dual gp ccp copy} and such that $\sh(U) = \lambda$ and $\Gamma_U \in \lrcell$ and $\lambda \ld \nu$.
\end{corollary}

\subsection{}
Here we give an alternative definition of $\gpda{U}{\sgnp(U)}$ for special $U$, which is related to catabolizability. We can prove some partial results towards the conjectures in \textsection\ref{ss cells in We} in this special case. In the even more restricted case in which $U$ has shape $(n)$, we can prove these conjectures, thereby describing all objects in $\cs$ isomorphic to $\cinv$.

\begin{definition}
Given an AT  $U$, let $P$ be the filling of $1^n$ obtained by stacking the columns of $U$ on top of each other (in order, from left to right) and adding $n(c-1)$ to the entries in the $c$-th column of $U$. If $P$ is a tableau, then $U$ is said to be \emph{stackable}.
\end{definition}
For example,
\renewcommand{\duma}{15}
\renewcommand{\dumb}{13}
\renewcommand{\dumc}{14}
\renewcommand{\dumd}{23}
\renewcommand{\dume}{26}
\renewcommand{\dumf}{45}
\[ \text{if } U = {\Yboxdim10pt \tiny \young(146\duma,2\dumb)}\ , \text{ then } P = {\Yboxdim10pt \tiny\young(1,2,\dumc,\dumd,\dume,\dumf)}\ .\]
The filling  $P$ is a tableau, so $U$ is stackable.

Stackable tableaux are those for which $\sgnp$ is easy to compute. If $U$ is stackable, then it is straightforward to see that  $\sgnp(U)=\sgnp(\creading(U))$ is the tableau $P$ in the definition above. Note that the dual Garnir tableau $\dual{G}_\lambda$ is stackable and $\sgnp(\dual{G}_\lambda) = \dual{G}_{1^n}$. As we will soon see, the (conjectural) dual Gasia-Procesi ccp copies  $\gpda{U}{\sgnp(U)}$ for $U$ stackable of shape  $\lambda$ have similar combinatorics to the dual Garsia-Procesi ccp $\gpda{\dual{G}_\lambda}{\dual{G}_{1^n}}$.  These appear to be exactly the ccp copies for which catabolizability as defined in \textsection\ref{ss catabolizability definition} is sensible.

\begin{example}\label{ex easy stackable}
Let $U$ be a standard tableau of shape $\lambda \vdash n$. The tableau $U$ is certainly stackable.  Let $\Gamma_U$ be the corresponding left  $W_f$-cell of $\H = \H(\S_n)$ and define $\AA_U$ by
\be \AA_U := \eH \tsr_\H A \Gamma_U = A\{\C_w : P(\rj{w}{S}) = U\}. \ee
This equality, which shows that $\AA_U$ is a cellular subquotient, is a special case of \cite[Proposition 2.6]{B0}.

We expect that  $\sgnp(U)$ is the unique minimal degree occurrence of the sign representation in $\AA_U$.  The csq $\csqa{U}{\sgnp(U)}$ is the minimal cellular quotient of $\AA_U$ containing $\sgnp(U)$ and we expect that this is equal to the minimal quotient of $\AA_U$ containing $\sgnp(U)$.  Note that the $\csqa{U}{\sgnp(U)}$ as $U$ ranges over standard tableau of shape $\lambda$ are all isomorphic by Proposition \ref{p star ops give isomorphism}.
\end{example}

\begin{algorithm}
\label{a cat insertion}
This algorithm depends on a AT $Q$ of shape $\lambda$ such that $\sgn(Q)$ exists and integers $d_1 \leq d_2 \leq\dots \leq d_{\lambda_1}$. It takes as input an affine word $w$ and outputs true or false.
This algorithm is a variant of the sign insertion algorithm (Algorithm \ref{a sign insertion}) and we maintain the notation from its description.  The insertion tableau will be the same as for the sign insertion algorithm; the recording tableau is still a single-column tableau, but is no longer required to have distinct entries.  Let  $x''$ and $P''$ be the word and insertion tableau  after the  $c$-th step of the sign insertion algorithm.  Let $k$ be maximal such that $P''_{1^k} = \sgnp(w)_{1^k}$.  Then define  $f'$ by
\be \label{e f'}
f'(c-1, x, P, Q) = (c,x'', P'', Q''),
\ee
where $Q''$ is the tableau obtained from $Q$ by appending entries $c$ to the bottom of $Q$ until $|Q''| = k$.

This algorithm repeatedly applies $f'$ to the tuple $(0, w, \emptyset, \emptyset)$ and terminates when the word of the tuple is empty and outputs the pair of tableaux.  Step, insertion step, and corotation step are defined here as they are  for the sign insertion algorithm.

Define a sequence of integers $d_1 < d_2 <\dots < d_{r}$ inductively as follows:   $d_1 =|w|$, and $d_j = d_{j-1} + |x|$, where $x$ is the word after the $d_{j-1}$-th step. The integer $r$ is defined so that  the algorithm terminates after the $d_r$-th step. The $j$-th \emph{pass} of the algorithm consists of steps $d_{j-1}+1$ to  $d_{j}$ inclusive (define  $d_0=0$). We remark that the sequence $d_1, d_2,\dots, d_r$ can be obtained in a straightforward way from $\sgnq(w)$.
\end{algorithm}

\begin{definition}
Let $\eta$ be an  $r$-composition of  $n$ with partial sums $l_j = \sum_{i=1}^{j} \eta_i$, $j \in [r]$. An affine word $w$ is  $\eta$-\emph{word catabolizable} if  $|Q''| \geq l_j$ for all $j \in [r]$, where $Q''$ is the recording tableau of $w$ after the $j$-th pass of Algorithm \ref{a cat insertion}.
\end{definition}

\begin{example}
Let $w$ be the word from Example \ref{ex sign insertion}. Then the sequence of recording tableaux produced by Algorithm \ref{a cat insertion} is \newcommand{\eleven}{11}
\newcommand{\twelve}{12}
\[ \emptyset, \emptyset, \emptyset, {\Yboxdim10pt \tiny \young(3)}\ , {\Yboxdim10pt \tiny \young(3)}\ , {\Yboxdim10pt \tiny \young(3,5)}\ , {\Yboxdim10pt \tiny \young(3,5)}\ ,{\Yboxdim10pt \tiny \young(3,5)}\ ,{\Yboxdim10pt \tiny \young(3,5,8)}\ ,{\Yboxdim10pt \tiny \young(3,5,8,9)}\ ,{\Yboxdim10pt \tiny \young(3,5,8,9)}\ ,{\Yboxdim10pt \tiny \young(3,5,8,9,\eleven)}\ ,{\Yboxdim10pt \tiny \young(3,5,8,9,\eleven,\twelve)}.\]

The word $w$ is $(2,2,1,1)$-word catabolizable, but not $(2,2,2)$-word catabolizable.
\end{example}

\begin{lemma}
\label{l algorithm facts}
Let $x$ be the word after the  $j$-th pass of Algorithm \ref{a cat insertion}.  Then $x$ can be replaced by $\leftexp{*}{x}$, with $x \tto \leftexp{*}{x}$ a Knuth transformation, without changing the insertion tableau or the length of the recording tableau after subsequent passes.  In particular, the set of $\eta$-word catabolizable words is invariant under Knuth transformations.
\end{lemma}
\begin{proof}
Let $(d_{j+1}, x', P',Q')$ be the tuple for  $x$ and $(d_{j+1}, \leftexp{*}{x}', P'',Q'')$  the tuple for  $\leftexp{*}{x}$ after the  $j+1$-st pass. By Proposition \ref{p sign insertion algorithm facts} (f) and its proof,   $P' = P''$. Since the final insertion tableau is also unchanged by the Knuth transformation and the length of the recording tableau only depends on the current insertion tableau  and the final insertion tableau,  $|Q'|=|Q''|$.  By the proof of Proposition \ref{p sign insertion algorithm facts} (f), either $x' = \leftexp{*}{x}'$ or $x' \tto \leftexp{*}{x}'$ is a Knuth transformation, so the result  follows by induction.
\end{proof}

\begin{conjecture}
\label{cj cat equivalences}
Suppose $U$ is a stackable AT of shape $\lambda$. Then the following are equivalent for an AT $T$:
\begin{list}{\emph{(\roman{ctr})}} {\usecounter{ctr} \setlength{\itemsep}{1pt} \setlength{\topsep}{2pt}}
\item $T$ is $(\sgnp(U),\lambda')$-row catabolizable.
\item  $\sgnp(T) = \sgnp(U)$ and $\reading(T)$ is  $\lambda'$-word catabolizable.
\item $\sgnp(T) = \sgnp(U)$ and $\creading(T)$ is  $\lambda'$-word catabolizable.
\item $T$ is $(U,1^{\lambda_1})$-column catabolizable.
\item There is a sequence of Knuth transformations and corotation-edges from $w := \reading(T)$ to $\reading(\sgnp(U))$ and there is a sequence of Knuth transformations, corotation-edges, and ascent-edges from $\reading(U)$ to $w$.
\end{list}
\end{conjecture}

\begin{proposition}
\label{p cat equivalences}
Maintain the notation of Conjecture \ref{cj cat equivalences}. Properties (ii), (iii), and (iv) are equivalent, (i) implies (ii), and any of (i)-(iv) implies (v).
\end{proposition}

\begin{proof}
The equivalence of (ii) and (iii) is immediate from Lemma \ref{l algorithm facts}.

We now prove the equivalence of (iii) and (iv). Let $(d_1, x,P,Q)$ be the tuple after the first pass of Algorithm \ref{a cat insertion} run on  $\creading(T)$.  Then  $P$ is equal to the first column of $T$ implying
\be
\label{e words cat initial block}
T_{1^{\lambda'_1}} = U_{1^{\lambda'_1}} \iff P_{1^{\lambda'_1}} = U_{1^{\lambda'_1}} \iff |Q| \geq \lambda'_1.
\ee
If any (all) of these conditions fails, then (iii) and (iv) do not hold, so we may assume these conditions hold.
Then we have the following chain of equivalences
\refstepcounter{equation}
\label{e word catab equivalences}
\begin{enumerate}[label={(\theequation.\alph{*})}]
\item Property (iii) holds.
\item $\sgnp(x P_1 P_2 \dots P_{|P|-l_1}) = \sgnp(n + U_{1, \text{east}})$ and $x P_1 P_2 \dots P_{|P|-l_1}$ is $\widehat{\lambda'}$-word catabolizable.
\item $\sgnp(x P_1 P_2 \dots P_{|P|-l_1}) = \sgnp(n + U_{1, \text{east}})$ and $\creading((n + T^*_{1, \text{east}}) T^*_{1, \text{west}})$ is $\widehat{\lambda'}$-word catabolizable.
\item $(n + T^*_{1, \text{east}}) T^*_{1, \text{west}}$ is $(n + U_{1, \text{east}}, 1^{\lambda_1-1})$-column catabolizable.
\item $T^*_{1, \text{east}} (\ng n + T^*_{1, \text{west}})$ is $(U_{1, \text{east}}, 1^{\lambda_1-1})$-column catabolizable.
\item Property (iv) holds,
\end{enumerate}
where $\widehat{\lambda'} = (\lambda'_2, \lambda'_3, \dots, \lambda'_{\lambda_1})$ and $T^*$ is the skew subtableau of $T$ obtained by removing $T_{1^{l_1}}$.
Note that if we replace the tuple $(d_1, x, P, Q)$ after the first pass with  $(d_1, x P_1 P_2 \dots P_{|P|-l_1}, P_{1^{l_1}}, Q_{1^{l_1}})$, nothing changes in the remainder of the algorithm except the indexing of steps; in particular, the recording tableau after each subsequent pass does not change.
The equivalence of (\ref{e word catab equivalences}.a) and (\ref{e word catab equivalences}.b) follows from
this observation, Proposition \ref{p sign insertion algorithm facts} (j), and the assumption that $U$ is stackable.
%The application of the you is stackable assumption is slightly tricky
Statements (\ref{e word catab equivalences}.b) and (\ref{e word catab equivalences}.c) are equivalent by Lemma \ref{l algorithm facts} as  $x = \creading(n + T^*_{1, \text{east}})$ and $P_1 P_2 \dots P_{|P|-l_1} = \creading(T^*_{1, \text{west}})$. Statements (\ref{e word catab equivalences}.c) and (\ref{e word catab equivalences}.d) are equivalent by induction. The equivalence of (\ref{e word catab equivalences}.d), (\ref{e word catab equivalences}.e), and (\ref{e word catab equivalences}.f) is clear.

The proof that (i) implies (ii) is similar to the proof that (iii) and (iv) are equivalent.  We may assume that the conditions in (\ref{e words cat initial block}) hold.  There is a chain of implications similar to  the  chain of equivalences above.  The difference occurs when we know $\reading((n + T^*_{l_1, \text{south}}) T^*_{l_1, \text{north}})$ is $\widehat{\lambda'}$-word catabolizable by induction.  Let $X = T^*_{l_1, \text{north}}$ and
\[x' := \reading(n+X_{1,\text{east}})\reading(n + T^*_{l_1, \text{south}})\reading(X_{1,\text{west}}),\]
which is the word of the tuple after the first pass of the algorithm run on $\reading(T)$.  The key fact to check is that $\reading((n + T^*_{l_1, \text{south}}) T^*_{l_1, \text{north}})$ is $\widehat{\lambda'}$-word catabolizable implies the same for $x'$.  To see this, note that if $w \tto \pi w$ is a corotation-edge with $w_{n-1} < w_n$, then  $w$'s being $\widehat{\lambda'}$-word catabolizable implies the same for $\pi w$.  Since a step of Algorithm \ref{a cat insertion} can be rephrased as a composition of Knuth transformations and corotation-edges of the above form, the fact follows.

If any of (i)-(iv) holds, then (ii) holds, so there is a path of Knuth transformations and corotation-edges from $\reading(T)$ to $\reading(\sgnp(T))$.

If any of (i)-(iv) holds, then (iv) holds.  Suppose that $Q$ is a tableau that appears in the inductive verification that $T$ is $(U, 1^{\lambda_1})$-column catabolizable, i.e.  $Q_{1^{\lambda'_c}}$ is the  $c$-th column of $U$ and $\ccat_{1^{\lambda'_c}}(Q)$ is  $(U_{c,\text{east}}, 1^{\lambda_1 - c})$-column catabolizable. Since $U$ is stackable, the entries in $Q^*_{1,\text{west}}$ are larger than those in $U_{c,\text{west}}$. Thus the word
\[\creading(U_{c,\text{east}})\creading(\ccat_{1^{\lambda'_c}}(Q))\]
can be obtained from
\[\creading(U_{c-1,\text{east}})\creading(Q)\]
by following a sequence of rotation-edges, Knuth transformations, and ascent-edges in reverse.  Iterating this argument shows that there is a sequence of Knuth transformations, corotation-edges, and ascent-edges from $\reading(U)$ to $\reading(T)$, completing the proof that any of (i)-(iv) implies (v).
\end{proof}

\begin{definition}
\label{d ccp copy stackable}
If  $U$ is stackable, redefine the dual Garsia-Procesi cocyclage poset copy $\gpda{U}{\sgnp(U)}$ to be the ccp on the set of AT satisfying conditions (ii)-(iv) of Conjecture \ref{cj cat equivalences}.  This is conjecturally the same as $\gpda{U}{\sgnp(U)}$ from Definition \ref{d dual gp ccp copy} in the case $U$ is stackable.
\end{definition}

\begin{remark}
\label{r definitions gp dual agree}
For $U = \dual{G}_\lambda$, this definition agrees with  the definition of $\gpda{\dual{G}_\lambda}{\dual{G}_{1^n}}$ given in  \textsection\ref{ss atom categories}.  The only way this could fail is if
$\gpda{U}{\sgnp(U)} \not \subseteq \fsp(\cinv) = \dsw$ $(U = \dual{G}_\lambda)$.  We know this inclusion to hold however, because  if $T$ satisfies (iv) of Conjecture \ref{cj cat equivalences}, then it satisfies (v) of this conjecture, implying  $T \in \csqa{G_{(n)}}{G_{1^n}}$. But we know that this csq is equal to $\cinv$.
\end{remark}

In the case $U$ is stackable, the ``easy half'' of Conjecture \ref{cj gp dual copies} (b) is immediate from the new definition of $\gpda{U}{\sgnp(U)}$ and Proposition \ref{p cat equivalences}:

\begin{proposition}
\label{p q sgn q contains q sgn q}
If $U$ is stackable, then $\fsp(\csqa{U}{\sgnp(U)}) \supseteq \gpda{U}{\sgnp(U)}$.
\end{proposition}

\begin{example}\label{ex dual gp ccp}
Figure \ref{f q sgn q} depicts the ccp $\gpda{U}{\sgnp(U)}$ for $U = \setlength{\cellsize}{11pt} \small \tableau{1&3&4&5\\2}$.
\begin{figure}
\begin{pspicture}(0pt,0pt)(500pt,400pt){\tiny
\hoogte=9pt
\breedte=9pt
\dikte=0.2pt

\newdimen\horizcent
\newdimen\ycor
\newdimen\xcor
\newdimen\horizspace
\newdimen\temp

\horizcent=185pt
\ycor=375pt
\advance\ycor by 0pt
\horizspace=50pt
\xcor=\horizcent
\temp=50pt
\multiply \temp by 0 \divide \temp by 2
\advance\xcor by -\temp
\rput(\xcor,\ycor){\rnode{v15h1}{\begin{Young}
1&3&4&5\cr
2\cr
\end{Young}}}
\advance\xcor by \horizspace
\advance\ycor by -55pt
\horizspace=50pt
\xcor=\horizcent
\temp=50pt
\multiply \temp by 1 \divide \temp by 2
\advance\xcor by -\temp
\rput(\xcor,\ycor){\rnode{v25h2}{\begin{Young}
1&3&4\cr
2&15\cr
\end{Young}}}
\advance\xcor by \horizspace
\rput(\xcor,\ycor){\rnode{v25h3}{\begin{Young}
1&3&4\cr
2\cr
15\cr
\end{Young}}}
\advance\xcor by \horizspace
\advance\ycor by -55pt
\horizspace=50pt
\xcor=\horizcent
\temp=50pt
\multiply \temp by 2 \divide \temp by 2
\advance\xcor by -\temp
\rput(\xcor,\ycor){\rnode{v35h4}{\begin{Young}
1&3&15\cr
2&14\cr
\end{Young}}}
\advance\xcor by \horizspace
\rput(\xcor,\ycor){\rnode{v35h5}{\begin{Young}
1&3&15\cr
2\cr
14\cr
\end{Young}}}
\advance\xcor by \horizspace
\rput(\xcor,\ycor){\rnode{v35h6}{\begin{Young}
1&3\cr
2&15\cr
14\cr
\end{Young}}}
\advance\xcor by \horizspace
\advance\ycor by -55pt
\horizspace=50pt
\xcor=\horizcent
\temp=50pt
\multiply \temp by 2 \divide \temp by 2
\advance\xcor by -\temp
\rput(\xcor,\ycor){\rnode{v45h7}{\begin{Young}
1&14&15\cr
2\cr
13\cr
\end{Young}}}
\advance\xcor by \horizspace
\rput(\xcor,\ycor){\rnode{v45h8}{\begin{Young}
1&3\cr
2&14\cr
25\cr
\end{Young}}}
\advance\xcor by \horizspace
\rput(\xcor,\ycor){\rnode{v45h9}{\begin{Young}
1&3\cr
2\cr
14\cr
25\cr
\end{Young}}}
\advance\xcor by \horizspace
\advance\ycor by -65pt
\horizspace=50pt
\xcor=\horizcent
\temp=50pt
\multiply \temp by 1 \divide \temp by 2
\advance\xcor by -\temp
\rput(\xcor,\ycor){\rnode{v55h10}{\begin{Young}
1&14\cr
2&25\cr
13\cr
\end{Young}}}
\advance\xcor by \horizspace
\rput(\xcor,\ycor){\rnode{v55h11}{\begin{Young}
1&14\cr
2\cr
13\cr
25\cr
\end{Young}}}
\advance\xcor by \horizspace
\advance\ycor by -65pt
\horizspace=50pt
\xcor=\horizcent
\temp=50pt
\multiply \temp by 0 \divide \temp by 2
\advance\xcor by -\temp
\rput(\xcor,\ycor){\rnode{v65h12}{\begin{Young}
1&25\cr
2\cr
13\cr
24\cr
\end{Young}}}
\advance\xcor by \horizspace
\advance\ycor by -65pt
\horizspace=50pt
\xcor=\horizcent
\temp=50pt
\multiply \temp by 0 \divide \temp by 2
\advance\xcor by -\temp
\rput(\xcor,\ycor){\rnode{v75h13}{\begin{Young}
1\cr
2\cr
13\cr
24\cr
35\cr
\end{Young}}}
\advance\xcor by \horizspace
}
\ncline[nodesep=1pt]{->}{v15h1}{v25h3}
\ncline[nodesep=1pt]{->}{v25h2}{v35h5}
\ncline[nodesep=1pt]{->}{v25h2}{v35h6}
\ncline[nodesep=1pt]{->}{v25h3}{v35h6}
\ncline[nodesep=1pt]{->}{v35h4}{v45h7}
\ncline[nodesep=1pt]{->}{v35h4}{v45h8}
\ncline[nodesep=1pt]{->}{v35h5}{v45h9}
\ncline[nodesep=1pt]{->}{v35h6}{v45h7}
\ncline[nodesep=1pt]{->}{v45h7}{v55h11}
\ncline[nodesep=1pt]{->}{v45h8}{v55h10}
\ncline[nodesep=1pt]{->}{v45h9}{v55h11}
\ncline[nodesep=1pt]{->}{v55h10}{v65h12}
\ncline[nodesep=1pt]{->}{v55h11}{v65h12}
\ncline[nodesep=1pt]{->}{v65h12}{v75h13}
\end{pspicture}
% We tried to insert this into the caption: {\tiny \tableau{1&3&4&5\\2}}
\caption{ The ccp copy $\gpda{U}{\sgnp(U)}$ for $U$ the tableau in the top row. All cocyclage-edges are drawn.}
\label{f q sgn q}
\end{figure}
\end{example}

We are now in a position to state the generalization of Corollary \ref{c cinv} that uses the full power of Theorem \ref{t factoriz}. For $\lambda \in \pY_+$ and $u_2 \in \pW$ such that $\invlr(u_2) \in \ds$, put
\be N_{\lambda,u_2}:=A \{s_{\lambda}(\y) \C_{u_1 w_0 u_2}:u_1 \in \ds\}. \ee

\begin{theorem}
\label{t coinvariant copies}
Suppose $w \in \eW$ is maximal in its coset $W_f w$ and let $w = w_0 y^\beta u'$ with $\beta \in \pY_+$, $\invlr(u') \in \ds$ be the factorization of Proposition \ref{p two-sided primitive}. Let $U$ be the single-row tableau $P(w)$. Then
\begin{list}{\emph{(\alph{ctr})}} {\usecounter{ctr} \setlength{\itemsep}{1pt} \setlength{\topsep}{2pt}}
\item $N_{\beta,u'}$ and $A\{\C_{uw} : u \in \ds \}$ are equal and are cellular subquotients of $\pH$.
\item $N_{\beta,u'}$ is isomorphic to $\cinv$ in $\cs$ (and therefore isomorphic to any $N_{\lambda,u_2}$).
\item $\fsp(N_{\beta,u'}) = \sgnpa{\sgnp(U)}$ and is equal to $\gpda{U}{\sgnp(U)}$, with the old Definition \ref{d dual gp ccp copy}.
\item $\fsp(N_{\beta,u'})$ is equal to $\gpda{U}{\sgnp(U)}$, with the new Definition \ref{d ccp copy stackable}.
\item $N_{\beta,u'} = \csqa{U}{\sgnp(U)}$.
\end{list}
\end{theorem}
\begin{proof}
The equality of $N_{\beta,u'}$ and $A\{\C_{uw} : u \in \ds \}$ is immediate from Theorem \ref{t factoriz}. For $\lambda \in \pY_+$ and $u_2 \in \pW$ such that $\invlr(u_2) \in \ds$, define
 \[ N_{\supseteq \lambda,u_2} := \bigoplus_{\mu \supseteq \lambda} N_{\mu,u_2},\ N_{\lambda}:= \bigoplus_{\invlr(u_2) \in \ds} N_{\lambda, u_2}, \]
\[ \ N_{\supseteq \lambda} := \bigoplus_{\mu \supset \lambda} N_{\mu}, \ \text{ and } N_{\supsetneq \lambda} := \bigoplus_{\mu \supsetneq \lambda} N_{\mu}.\]
By Theorem \ref{t factoriz} and the Littlewood-Richardson rule,  $N_{\supseteq \lambda}$ and $N_{\supsetneq \lambda}$ are submodules of $\pH$.  Thus $N_{\lambda} = N_{\supseteq \lambda} / N_{\supsetneq \lambda}$ is a cellular subquotient of  $\pH$.
Moreover, as  $\pH e^+ \cong N_{\supseteq \textbf{0}, \idelm}$ is a submodule of the left $\pH$-module $\pH$, $N_{\supseteq \textbf{0},u_2} = N_{\supseteq \textbf{0},\idelm} \rC_{u_2}$ is as well, where the equality is by Theorem \ref{t factoriz}. Since the intersection of two cellular subquotients is a cellular subquotient, $N_{\lambda, u_2} = N_{\supseteq \textbf{0},u_2} \cap N_{\lambda}$ is a cellular subquotient of $\pH$. This proves (a).

For (b), define the map $f: N_{\supseteq \textbf{0}, \idelm} \to N_{\supseteq \lambda,u_2}$ by requiring $\C_{w_0} \mapsto s_\lambda(\y) \C_{w_0 u_2}$. This implies $f(\C_{u w_0})= s_\lambda(\y) \C_{u w_0 u_2}$ for all  $u \in \ds$ and  $f(N_{\supsetneq \textbf{0},\idelm}) \subseteq N_{\supsetneq \lambda,u_2}$ by Theorem \ref{t factoriz}.  Thus $f$ gives rise to the isomorphism
\be N_{\textbf{0}, \idelm} = N_{\supseteq \textbf{0}, \idelm}/N_{\supsetneq \textbf{0},\idelm}\xrightarrow{\cong} N_{\supseteq \lambda,u_2} /N_{\supsetneq \lambda,u_2} = N_{\lambda,u_2}.\ee This proves (b) as $N_{\textbf{0}, \idelm} = \cinv$.

For statement (c), first note that $\fsp(\cinv)$ is equal to $\gpda{G_{(n)}}{\sgnp(G_{(n)})}$ (old definition) by definition. By (b), $\fsp(N_{\beta,u'}) \cong \fsp(\cinv)$. Then by the same argument as in the proof of Proposition \ref{p sign insertion algorithm facts} (h) and (i), we have $\fsp(N_{\beta,u'}) = \gpda{U}{\sgnp(U)}$ (old definition). This ccp is certainly contained in $\sgnpa{\sgnp(U)}$. To see that it is equal, suppose $v \in \sgnpa{\sgnp(U)}$. Then $v \in \lrcelllong{(n)} \cap \pW$, so $v$ belongs to some $N_{\beta',u''} = (\fsp)^{-1} (\gpda{U'}{\sgnp(U')})$, where $U' = P(w_0 y^{\beta'} u'')$; but then $\sgnp(U) = \sgnp(v) = \sgnp(U')$, implying $U = U'$. Hence $v \in \fsp(N_{\beta,u'}) = \gpda{U}{\sgnp{U}}$, as desired.

We certainly have that $\gpda{U}{\sgnp(U)}$ (new definition) is contained in $\sgnpa{\sgnp(U)}$, so given (c), we must show this containment is an equality. Suppose $v \in \sgnpa{\sgnp(U)}$.  We need to show that $v$ is $1^n$-word catabolizable. Our assumption implies $v \in \lrcelllong{(n)} \cap \pW$. Now the poset $\wposet(v)$ is a total order on $[n]$, which is equal to $\wposet(v')$ for any $v'$ obtainable from $v$ by Knuth transformations and corotations. It follows that after the $j$-th pass of Algorithm \ref{a cat insertion} run on $v$, the insertion tableau contains at least the $j$ numbers whose residues correspond to the $j$ smallest elements of this total order. This shows that $v$ is $1^n$-word catabolizable, proving (d).

Since $N_{\beta,u'}$ contains the left cells $\Gamma_U$ and $\Gamma_{\sgnp(U)}$, $N_{\beta,u'} \supseteq \csqa{U}{\sgnp(U)}$. Then by (d) and Proposition \ref{p q sgn q contains q sgn q}, $\fsp(N_{\beta,u'}) = \gpda{U}{\sgnp(U)} \subseteq \fsp(\csqa{U}{\sgnp(U)})$. This implies $N_{\beta,u'} \subseteq \csqa{U}{\sgnp(U)}$, hence (e) is proved.
\end{proof}

\begin{corollary}
Conjecture \ref{cj gp dual copies} holds for $U$ of shape $(n)$ and Conjecture \ref{cj eW-cells and dual GP copies} holds for the lowest two-sided  $\eW$-cell of $\eW$ ($= \lrcelllong{(n)}$).
\end{corollary}
\subsection{}
\label{ss SSYT in PAT}
Here we show that the dual of $\ccp(\Tab(\eta))$ is strongly isomorphic to a subposet of $\ccp(AT)$ (see Definition \ref{d strongly isomorphic}).

Given an $r$-composition $\eta$ of $n$, let $\eS_\eta = \eS_{\eta_1} \times \eS_{\eta_2} \times \cdots \times \eS_{\eta_r}$. Let $l_j = \sum_{i=1}^{j-1} \eta_i$, $j \in [r+1]$ be the partial sums of $\eta$ and $B_j$ be the interval $[l_j+1,l_{j+1}]$ for $j \in [r]$. Recall that the notation $\rsd{a}$ denotes the element of $[n]$ congruent to the integer $a$ mod $n$.

Define $\alpha: \eS_n \to \eS_\eta$ by $\alpha(w) = (x^1, x^2, \dots, x^r)$, where (identifying $\eS_n$ and $\eS_\eta$ with affine words and $r$-tuples of affine words) $x^j$ is determined as follows: let $\tilde{w}$ be the word $w_1\ w_2\ \dots\ w_n$ sorted in increasing order, i.e., $\tilde{w} = w_0 (\rj{w}{S})$. Let $w^j$, $j \in [r]$ be the subword of the word of $w$ consisting of the numbers in $\{\tilde{w}_i : i \in B_j\}$; $w$ is a shuffle of its subwords $w^1, \dots, w^r$. Then $x^j$ is determined by $w^j$ and the conditions
\refstepcounter{equation}
\begin{enumerate}[label={(\theequation.\roman{*})}]
\item $\rsd{x_1^j}\ \rsd{x_2^j}\ \dots\ \rsd{x_{\eta_j}^j} \text{ has the same relative order as } \rsd{w_1^j}\ \rsd{w_2^j}\ \dots\ \rsd{w_{\eta_j}^j},$
\item $\frac{x^j_i - \rsd{x^j_i}}{\eta_j} = \frac{w^j_i - \rsd{w^j_i}}{n}, \ i \in [\eta_j].$
\end{enumerate}

\begin{example}
Suppose $n = 9$ and $\eta = (2,2,1,4)$. Then for the given $w$, $\alpha(w)$ is computed below.
\[ \begin{array}{rrcl}
&w &=& 13\ 44\ 9\ 31\ 12\ 25\ 46\ 7\ 8 \\
&(w^1, w^2, \dots, w^r) &= &(7\ 8, 9\ 12, 13, 44\ 31\ 25\ 46) \\
\alpha(w) = & (x^1, x^2, \dots, x^r) &= &(1\ 2, 2\ 11, 11, 42\ 31\ 23\ 44)
\end{array} \]
\end{example}

For $D = (D_1, D_2, \dots,  D_r)$ an ordered partition of the set $[n]$ with $|D_i| = \eta_i$, define another map $\br{\alpha}_D :\eS_n \to \eS_\eta$ by $\br{\alpha}_D(w) = (x^1, x^2, \dots, x^r)$, where $x^j$ is defined in terms of $w^j$ as above and $w^j$ is the subword of $w$ consisting of those $w_i$ such that $\rsd{w_i} \in D_j$ ($i \in [n]$).

\begin{example}
\label{ex bar alpha definition}
Suppose $n = 9$ and $\eta = (2,2,1,4)$ and $D = \{1, 5\} \sqcup \{2,9\} \sqcup \{6\} \sqcup \{3,4,7,8\}$. Then $\br{\alpha}_D$ is computed below.
\begin{align*}
w &= 13\ 44\ 9\ 31\ 12\ 25\ 46\ 7\ 8 \\
(w^1, w^2, \dots, w^r)&= (31\ 25, 9\ 12, 46, 13\ 44\ 7\ 8) \\
(x^1, x^2, \dots, x^r) &= (31\ 22, 2\ 11, 41, 11\ 42\ 3\ 4)
\end{align*}
\end{example}

Recall that $\word(\eta)$ denotes the set of words of content $\eta$. Let $\dual{\word(\eta)}$ be the set of semistandard words of content $\eta$, but with the convention that if two numbers in such a word are the same, then the one on the left is slightly bigger; let $\dual{\Tab(\eta)}$ be the set of transposed tableaux of content $\eta$, or equivalently, the set of insertion tableaux of $\dual{\word(\eta)}$; let $\ccp(\dual{\Tab(\eta)})$ be the cocyclage poset obtained from $\ccp(\Tab(\eta))$ by transposing tableaux and reversing edges. Let $W' \subseteq \eW$ consist of those $w$ such that $\alpha(w) = (x^1, \dots, x^r)$ with $x^j$ decreasing.
For $w \in W'$, denote by $\beta(w)$ the unique element of $\dual{\word(\eta)}$ such that $\lj{w}{S}$ and $\beta(w)$ have the same relative order. This defines a map $\beta: W' \to \dual{\word(\eta)}$.

For $D = (D_1, D_2, \dots,  D_r)$ an ordered partition of the set $[n]$ with $|D_i| = \eta_i$, define a map  $\br{\beta}_D: \eW \to \word(\eta)$, similar to $\br{\alpha}_D$, as follows: $\br{\beta}_D(w)$ has word $x = x_1\ x_2\ \dots\ x_n$, where $x_i = j$ if and only if $\rsd{w_i} \in D_j$.

\begin{example}
%For $w$ as in Example \ref{ex bar alpha definition},
%\[\br{\beta}_D(w) = 4\ 4\ 3\ 1\ 2\ 1\ 2\ 4\ 4. \]
If $n = 9$, $\eta = (2, 2, 1, 4)$, $D_j = [l_j + 1, l_{j+1}]$ with $l_j = \sum_{i=1}^{j-1} \eta_i$, and $w$ is as shown, then $\beta(w)$ and $\br{\beta}_D(w)$ follow
\[ \begin{array}{rrccccccccc}
&w =& 2& 23& 48& 35& 47& 1& 46& 39& 14,\\
&\lj{w}{S} =&2&4&9&5&8&1&7&6&3,\\
\beta(w) = &\br{\beta}_D(w) =& 1&2&4&3&4&1&4&4&2.
\end{array}\]
\end{example}

Let $W'_D := \{w \in W': \beta(w) = \br{\beta}_D(w)\}$. A \emph{weak corotation} of a semistandard word is the same as a corotation except we allow the number 1 to be corotated.

\begin{proposition}
\label{p beta commutes with kt and pi}
The map $\beta$ (or $\br{\beta}_D$)  restricted to $W'_D$ commutes with left-multiplication by $s \in S$, preserves left descent sets, and commutes with weak corotations.
\end{proposition}
\begin{proof}
Since $\beta(w)$ and $\lj{w}{S}$ have the same relative order, $\beta$ certainly commutes with left-multiplications by $s \in S$ and preserves left descent sets. On the other hand, $\br{\beta}_D$ certainly commutes with weak corotations.
\end{proof}

Also write $\beta$ for the map of tableau $P(w) \mapsto P(\beta(w))$, which is well-defined by the proposition.

Let $\eta_+$ be the partition obtained by sorting the parts of $\eta$.  Note that the word of  $\pi^{k}$,  $k \in \ZZ$ is decreasing and that  $A\{\pi^k:k \in [i,j]\}$,  $0 \leq i \leq j$, is a cellular subquotient of $\eH$.  Suppose  $\textbf{a} \in \ZZ^r$ and $\lceil a_j/\eta_j \rceil +1< \lceil a_{j+1}/\eta_{j+1} \rceil$ for all $j \in [r-1]$.  Next suppose an affine word $w$ satisfies
\be \label{e ssyt form w}
\alpha(w)=\br{\alpha}_D(w)  = (\pi^{a_1}, \pi^{a_2}, \dots, \pi^{a_r}).\ee
Then  $w$ is a shuffle of decreasing subwords  $w^j$ of the form
\be
\label{e a word}
a.c_k \ a.c_{k-1} \cdots a.c_1 \ (a-1).c_{\eta_j} \ (a-1).c_{\eta_{j-1}} \cdots (a-1).c_{k+1}, \ee
where $a = \lceil a_j / \eta_j \rceil$, $D_j = \{c_1,\ldots,c_{\eta_j}\}$, and $c_1 < \dots < c_{\eta_j}$.
It is then not hard to see that
\refstepcounter{equation}
\label{e ssyt tab def}
\begin{enumerate}[label={(\theequation)}]
\item there is a unique tableau  $U$ of shape $\lambda = (\eta_+)'$ such that any word inserting to $U$ satisfies  (\ref{e ssyt form w}).
\end{enumerate}
For example, with $\eta = (3,2,2,1)$,  $D_j = [l_j + 1, l_{j+1}]$, and $\textbf{a} = (0,2,4,3)$, the resulting $U$ is the  tableau in the first row and second column of Figure \ref{f part of 3 gpd atoms}.

\begin{definition}
For a cocyclage poset $\AA$, set $\word(\AA) = \{w:P(w) \in \AA\}$.
\end{definition}

\begin{lemma}
\label{p large a}
With $U$ defined in terms of $\textbf{a}$ as above and if  $\eta = \eta_+$, then  $U$ is stackable.  If, in addition, $a_1/\eta_1 << a_2/\eta_2 << \dots << a_r/\eta_r$, then $\word(\gpda{U}{\sgnp(U)}) \subseteq W'_D$.
\end{lemma}
\begin{proof}
With the present hypotheses, the column reading word of  $U$ is  $w^1 w^2 \dots w^r$, where $w^j$ is of the form (\ref{e a word}). Then  $\lceil a_j/\eta_j \rceil +1< \lceil a_{j+1}/\eta_{j+1} \rceil$ for all $j \in [r-1]$ implies  $U$ is stackable.

For the second statement observe that $\reading(\sgnp(U))$ is a shuffle of words $w^1,\ldots,w^r$, where  $w^j$ is of the form (\ref{e a word}) for  $a$ not too far from $\lceil a_j / \eta_j \rceil$.  Any word obtained from $\reading(\sgnp(U))$ by sequence of Knuth transformations and a small number of rotation-edges is also a shuffle of a similar form.  It is easy to see that such a shuffle belongs to $W'_D$.  For any $w \in \word(\gpda{U}{\sgnp(U)})$, there is a path from $\reading(\sgnp(U))$ to  $w$ consisting of Knuth transformations and at most  $\binom{n}{2}$ rotation-edges, hence the desired result.

% rom An element's $w$ being in $W'_D$ is preserved by left-multiplications by $s_i$ such that $\rsd{w_i}$, $\rsd{w_{i+1}}$ are not both in some $D_j$.
\end{proof}

\begin{theorem}
\label{t ssyt in pat}
With $U$ defined in terms of $\textbf{a}$ as in (\ref{e ssyt tab def}) and if $\word(\gpda{U}{\sgnp(U)}) \subseteq W'_D$, then there exists a section $\beta': \dual{\word(\eta)} \to W'_D$ of $\beta$ with image $\word(\gpda{U}{\sgnp(U)})$. If, in addition, $\eta$ is a partition, then $\beta: \gpda{U}{\sgnp(U)} \to \ccp(\dual{\Tab(\eta)})$  is a strong isomorphism.
\end{theorem}
%The 2nd part of theorem can probably be proved with fewer hypotheses
\begin{proof}
For any $T_1, T_2 \in \gpda{U}{\sgnp(U)}$, we will show by induction on $\deg(\sgnp(U))-\deg(T_i)$ that $\beta(T_1) = \beta(T_2)$ implies $T_1 = T_2$. The base case is $\sh(\beta(T_1)) = \sh(\beta(T_2)) = 1^n$. Since the only tableau of $\gpda{U}{\sgnp(U)}$ of shape $1^n$ is $\sgnp(U)$, we have $T_1 = \sgnp(U) = T_2$. For $\beta(T_1) = \beta(T_2)$ not of shape $1^n$, we use that every $w \in \word(\gpda{U}{\sgnp(U)})$ has a path of Knuth transformations and corotation-edges to $\reading(\sgnp(U))$.  Thus induction and Proposition \ref{p beta commutes with kt and pi} imply $T_1 = T_2$. This shows that $\beta$ restricted to $\gpda{U}{\sgnp(U)}$ is injective. Since every word of  $\dual{\word(\eta)}$ has a path of Knuth transformations and corotations to the decreasing word of  $\dual{\word(\eta)}$, a similar inductive argument shows that $\beta$ restricted to $\gpda{U}{\sgnp(U)}$ is surjective. Thus we can define $\beta'$ to be the inverse of $\beta$ restricted to $\gpda{U}{\sgnp(U)}$.

The second statement of theorem is a little tricky. The fact that $w$ and $\beta(w)$ have the same relative order together with Proposition \ref{p beta commutes with kt and pi} imply that $\beta$ is a strong isomorphism of cocyclage posets provided we check the following: (i) if $w, \pi w \in \word(\gpda{U}{\sgnp(U)})$, then $\beta(\pi w)$ is a corotation of  $\beta(w)$; (ii) if $x,x' \in \dual{\word(\eta)}$ and $x'$ is a corotation of x, then  $\beta'(x') = \pi \beta'(x)$.  Statement (i) would fail only if $\beta(w)$ ends in a 1, but this would imply $\sgnp(\pi w) \neq \sgnp(w)$, contradicting $w, \pi w \in \word(\gpda{U}{\sgnp(U)})$.  To prove (ii), we need a more explicit description of $\beta'$.

Note that a word $w \in \eW$ can be recovered uniquely from  $\br{\alpha}_D(w) \in \eS_\eta$ and  $\br{\beta}_D(w) \in \dual{\word(\eta)}$, and therefore $w \in W'_D$ can be recovered from $\br{\alpha}_D(w)$ and $\beta(w)$. For a word $x$, let $x|_{[j]}$ be the subword of $x$ obtained by removing from $w$ all numbers not in $[j]$. Also let $x^\dagger$ denote the reverse of the word $x$. We will need the notion of charge of a  semistandard word, which can be computed by the well-known circular-reading procedure (see, for instance, \cite[\textsection 3.6]{SW}).  The map $\beta'$ has the following description (which we temporarily denote by $\beta''$): given  $x \in \dual{\word(\eta)}$ define \[c_i:=\charge(x^\dagger|_{[j]}) - \charge(x^\dagger|_{[j-1]}), \ j \in [r],\] (charge of the empty word is defined to be $0$); then $w:= \beta''(x) \in  W'_D$ is determined by  $\beta(w) = x$ and  \[\br{\alpha}_D(w) = (\pi^{c_1},\ldots,\pi^{c_r}) \cdot \br{\alpha}_D(\reading(U))\] (this is a product in the group $\eS_\eta$).
%Equivalently, we may require $\br{\alpha}_D(w) = (\pi^{c_1-0},\pi^{c_2-1},\ldots,\pi^{c_r-r+1}) \br{\alpha}_D(\reading(\sgnp(U)))$.

There is a path of Knuth transformations, corotation-edges, and rotation-edges from  $\reading(U)$ to any element of $\word(\gpda{U}{\sgnp(U)})$, and by (i) and the remarks preceding it, this path is mapped by  $\beta$ to a path of Knuth transformations, corotations, and rotations in $\dual{\word(\eta)}$.  Then  $\beta' = \beta''$ is proved by showing that these maps agree on $x = \beta(\reading(U))$ and that if they agree on $x$, then they agree on $x'$, where  $x \tto x'$ is an edge in one of the paths just described.  These claims are straightforward to check.  To check it for corotations, suppose  $\beta(\pi w) = x'$ is a corotation of $\beta(w) = x$ (and is an edge in one of the paths just described) and $x_n = j$. Then
\[
\charge({x'}^\dagger|_{[i]}) - \charge(x^\dagger|_{[i]}) =
\begin{cases}
1 & \text{ if } i \geq j, \\
$0$ & \text{ if } i < j.
\end{cases}\]
On the other hand,
\be \label{e beta' beta''}
\br{\alpha}_D(\pi w) = (\idelm,\ldots,\pi,\ldots,\idelm) \cdot \br{\alpha}_D(w),\ee where  $\pi$ occurs in the $j$-th position.  Thus  $\beta'(x') = \pi \beta'(x)$ and  $\beta''(x')$ both map to the right hand side of (\ref{e beta' beta''}) under $\br{\alpha}_D$, hence  $\beta'(x') = \beta''(x')$.

The main point is that the path in $\word(\gpda{U}{\sgnp(U)})$ from $w$ to  $\reading(U)$ consists of corotation-edges and rotation-edges of a special form (those corresponding to corotation steps in the sign insertion algorithm), however by the check for corotations in the previous paragraph,  $\pi \beta''(x) = \beta''(x')$ for \emph{any} corotation $x \tto x'$.  This proves (ii).
\end{proof}

\begin{corollary}\label{c ssyt in pat}
If $\textbf{a}$ is as in Lemma \ref{p large a}, $U$ is defined in terms of  $\textbf{a}$ as in (\ref{e ssyt tab def}), and  $\eta = \eta_+$, then  \[ \beta: \gpda{U}{\sgnp(U)} \to \ccp(\dual{\Tab(\eta)}) \] is a strong isomorphism of cocyclage posets.
\end{corollary}

\begin{conjecture}\label{cj ssyt in pat}
Corollary \ref{c ssyt in pat} holds for arbitrary $\eta$.
\end{conjecture}

%This would actually not follow from Conjecture isomorphism of cocyclage poset, because the definition of the cocyclage poset for semistandard tableau has no substance -- it would have to be shown that reflection operators commute with certain cocyclages -- the ones that keep you in the atom

\begin{example}
\label{ex ssyt in PAT}
Suppose $n = 8$, $\eta = (3,2,2,1)$ and $D_j = [l_j + 1, l_{j+1}]$. Let $U, \tilde{U}, \beta(U)$ be the three tableaux in the first row of Figure \ref{f part of 3 gpd atoms} and let $T^{r,c}$ be the tableau in the $r$-th row and $c$-th column. In each column is a selection of tableaux from the isomorphic cocyclage posets $\gpda{U}{\sgnp(U)}$, $\gpda{\tilde{U}}{\sgnp(\tilde{U})}$, $\ccp(\dual{\Tab(\eta)})$ such that $T^{r,1} \leftrightarrow T^{r,2} \leftrightarrow T^{r,3}$ under the isomorphisms between these posets. In the first column, we see that $\beta_D = \beta$ on the second row but not on the third and there is a cocyclage-edge from $T^{2,1}$ to $T^{3,1}$. The $a_i$ are large enough so that $\beta(T^{r,2}) = \beta_D(T^{r,2}) = T^{r,3}$ for all $r$; this is in contrast to $\beta(T^{5,1}) \neq \beta_D(T^{5,1}) = T^{5,3}$. Also the cocyclage-edges $T^{4,1} \cc T^{5,1}$ and $T^{4,3} \cc T^{5,3}$ demonstrate that $\gpda{U}{\sgnp(U)}$ and $\ccp(\dual{\Tab(\eta)})$ are not strongly isomorphic.
\begin{figure}
\[
\footnotesize
% hoogte = height
% breedte = width
% dikte = linewidth
\hoogte=11pt
\breedte=12pt
\dikte=0.2pt
\begin{array}{lll}
\begin{Young}
1&4&6&8\cr
2&5&7\cr
3\cr
\end{Young}&
\begin{Young}
1&14&26&38\cr
2&15&27\cr
3\cr
\end{Young}&
\begin{Young}
1&2&3&4\cr
1&2&3\cr
1\cr
\end{Young}\\
%%%%%%%%%%%%%%%%%%%%%
\begin{Young}
1&4&7\cr
2&5\cr
3&18\cr
16\cr
\end{Young}&
\begin{Young}
1&14&27\cr
2&15\cr
3&48\cr
36\cr
\end{Young}&
\begin{Young}
1&2&3\cr
1&2\cr
1&4\cr
3\cr
\end{Young}\\
%%%%%%%%%%%%%%%%%%%%%%%%%%%%%%%%%%%%
\begin{Young}
1&5&7\cr
2&16&18\cr
3\cr
14\cr
\end{Young}&
\begin{Young}
1&15&27\cr
2&36&48\cr
3\cr
24\cr
\end{Young}&
\begin{Young}
1&2&3\cr
1&3&4\cr
1\cr
2\cr
\end{Young}\\
%%%%%%%%%%%%%%%%%%%%%%%%%%%%%%%%%%
\begin{Young}
1&5&7\cr
2\cr
3\cr
14\cr
18\cr
26\cr
\end{Young}&
\begin{Young}
1&15&36\cr
2\cr
3\cr
24\cr
37\cr
48\cr
\end{Young}&
\begin{Young}
1&2&3\cr
1\cr
1\cr
2\cr
3\cr
4\cr
\end{Young}\\
%%%%%%%%%%%%%%%%%%%%%%%%%%%%%%%%%%
\begin{Young}
1&5\cr
2&18\cr
3\cr
14\cr
17\cr
26\cr
\end{Young}&
\begin{Young}
1&15\cr
2&48\cr
3\cr
24\cr
37\cr
46\cr
\end{Young}&
\begin{Young}
1&2\cr
1&4\cr
1\cr
2\cr
3\cr
3\cr
\end{Young}
\end{array}
\]
\caption{A selection of tableaux from three isomorphic cocyclage posets.}
\label{f part of 3 gpd atoms}
\end{figure}
\end{example}

\begin{remark}
\label{r new charge computation}
The sign insertion algorithm and  Corollary \ref{c ssyt in pat} give an algorithm for computing charge of a semistandard word  $x$ of partition content. If Conjecture \ref{cj ssyt in pat} holds, then the sign insertion algorithm would give a way of computing charge for semistandard words not of partition content that avoids reflection operators.
%
%This algorithm can be adapted to a computation only involving the affine word $\beta'(x)$ (thought of as an infinite word) of a similar flavor to the well-known circular reading that computes charge (see \cite{SW}).  However, it turns out that it is genuinely different from this usual circular reading computation.
\end{remark}

\subsection{}
\label{ss sw atoms}
Here we discuss Shimozono-Weyman atoms in more detail. Recall from \textsection\ref{ss atom categories} that the SW ccp $\swra{G_\lambda}{\eta}$ (resp. $\swca{\dual{G}_\lambda}{\eta}$) is the cocyclage poset on the set of $(G_\lambda,\eta)$-row (resp. $(\dual{G}_\lambda,\eta)$-column) catabolizable tableaux. Semistandard versions of these sets of tableaux first appear in \cite{SW}, where they are conjectured to give a combinatorial description of certain generalized Hall-Littlewood polynomials.  These polynomials are symmetric polynomials with coefficients in $\CC[t]$ and are defined as the formal characters of the Euler characteristics of certain $\CC[\gl_n]$-modules supported in nilpotent conjugacy class closures. The generalized Hall-Littlewood polynomials are known to be $t$-analogues of the character of certain  $\S_n$-modules induced from a parabolic subgroup of $\S_n$, however as far as we know, it is yet to be proved that SW ccp also satisfy this property.

We now state precisely a conjecture mentioned in \textsection\ref{ss atom categories}.
\begin{conjecture}
The SW csq $(\fsp)^{-1}(\swra{G_\lambda}{\eta})$ and $(\fsp)^{-1}(\swca{\dual{G}_\lambda}{\eta})$ exist.
\end{conjecture}
This has been checked in Magma for $n$ up to 6.

It seems that catabolizability combinatorics can be extended to the dual GP ccp copy $\gpda{U}{\sgnp(U)}$ for  $U$ stackable, but not to copies for non-stackable $U$.  For $U$ stackable, let  $\swca{U}{\eta}$ be the cocyclage poset on the set of $(U,\eta)$-column catabolizable tableaux.

\begin{conjecture}
\label{cj init block}
Suppose that  $U$ and  $U'$ are stackable AT of shape $\lambda$.  Then $\swca{U}{\eta}$ is a sub-cocyclage poset of $\gpda{U}{\sgnp(U)}$ (and $\swca{U'}{\eta}$ is a sub-cocyclage poset of $\gpda{U'}{\sgnp(U')}$).  Further suppose that $f: \gpda{U}{\sgnp(U)} \cong \gpda{U'}{\sgnp(U')}$ is an isomorphism in Cocyclage Posets.   Then  $f$ restricts to an isomorphism $\swca{U}{\eta} \cong \swca{U'}{\eta}.$
\end{conjecture}

%\begin{conjecture}
%\label{cj init block}
%Suppose that  $U$ and  $U'$ are stackable AT of shape $\lambda$ and $f: \gpda{U}{\sgnp(U)} \cong \gpda{U'}{\sgnp(U')}$ is an isomorphism in Cocyclage Posets.   Then for all $T \in \gpda{U}{\sgnp(U)}$ and  $c \in [\lambda_1]$,  $T_{C_1} = U_{C_1}$ if and only if $f(T)_{C_1} = U'_{C_1}$, where $C_1:=(\lambda'_1,\ldots,\lambda'_c)'.$
%\end{conjecture}

Let $U$ be the tableau $\beta'(\dual{Z}_{\lambda'})$, where $\dual{Z}_{\lambda'}$ is the transpose of the superstandard tableau of shape and content $\lambda'$ and  $\beta'$ is as in Theorem \ref{t ssyt in pat}.  As remarked in \textsection\ref{ss catabolizability definition}, the SW ccp copy $\swca{U}{\eta}$ agrees with the original set of semistandard tableaux $CT(\lambda; R)$ as defined in \cite{SW} under the correspondence $\beta$.

Conjecture \ref{cj init block} has been extensively tested in the case  $f$ is the standardization map (or, to be consistent with the notation in the conjecture, in the case $U$ is as in the previous paragraph and $U' = \dual{G}_\lambda$).

\subsection{}
\label{ss chen atom}
Here we give the definition of Chen ccp as the intersection of certain SW ccp \cite{Ch}.

For a skew shape $\theta = \Theta / \nu$, define $\row(\theta)$ (resp. $\col(\theta)$) to be the composition given by the row (resp. column) lengths of $\theta$, i.e., $\row(\theta)_i = \Theta_i-\nu_i$ and $\col(\theta)_i = \Theta'_i - \nu'_i$.

\begin{definition}
If $\theta = \Theta / \nu$ is a skew shape with $|\theta| = n$ and $\lambda, \mu \vdash n$, then $\lambda$ is \emph{skew-linked to} $\mu$ \emph{by} $\theta$, written $\skl{\lambda}{\theta}{\mu}$, if $\row(\theta) = \lambda$ and $\col(\theta) = \mu'$. If $\lambda$ is skew-linked to $\mu$ by $\theta$ for some $\lambda, \mu \vdash n$, then we say that $\theta$ is a \emph{skew linking} shape. See Figure \ref{f skew-link}.
\end{definition}

\begin{figure}
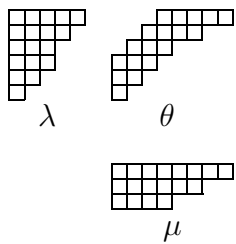

\[ \Yboxdim6pt \begin{array}{cc} \yng(5,4,3,3,2,1) & \young(:::\hfil\hfil\hfil\hfil\hfil,::\hfil\hfil\hfil\hfil,:\hfil\hfil\hfil,\hfil\hfil\hfil,\hfil\hfil,\hfil)  \\
{\small \lambda} & {\small \theta} \vspace{10pt}\  \\
& \yng(8,6,4)\\  &{\small \mu}\end{array}\]
 \caption{$\lambda$ skew-linked to $\mu$ by $\theta$.}
 \label{f skew-link}
\end{figure}

Let $V_\lambda$ be the Specht module of shape $\lambda$, $J_\mu \subseteq S$ be as in \textsection\ref{ss BZ category}, $\S_\mu$ the parabolic subgroup ${\S_n}_{J_\mu}$, and $e^+_\mu := \lj{e^+}{J_\mu}$ the trivial module for $\CC\S_\mu$.

\begin{proposition}[{\cite{Ch}}]
\label{p chen skew-link}
The following are equivalent.
\begin{list}{\emph{(\alph{ctr})}} {\usecounter{ctr} \setlength{\itemsep}{1pt} \setlength{\topsep}{2pt}}
\item There is a skew shape $\theta$ such that $\skl{\lambda}{\theta}{\mu}$,
\item There exists a non-negative integer  $d(\mu,\lambda)$ such that in the $\S_n$-module $R \tsr_\CC (\ind^{\S_n}_{\S_\mu} e^+_\mu)$,  $V_\lambda$ occurs with multiplicity 1 in degree $d(\mu,\lambda)$ and this is the unique occurrence of any $V_\nu$ with $\nu \ld \lambda$ in degree less than or equal to $d(\mu,\lambda)$.
\end{list}
\end{proposition}

\begin{definition}
For $\lambda, \mu$ satisfying any (all) of the conditions in Proposition \ref{p chen skew-link}, let $\moda{\mu}{\lambda}$ be the minimal quotient of $R \tsr_\CC (\ind^{\S_n}_{\S_\mu} e^+_\mu)$ (in $R\star W_f$-$\Mod$) whose restriction to  $\CC \S_n$ contains $V_\lambda$. By the proposition, this determines a unique $R \star W_f$-module.
\end{definition}

\begin{conjecture}
\label{cj skew linking}
Suppose $\skl{\lambda}{\theta}{\mu}$ and $M \in \cs$ is a copy of a cellular subquotient of $\cinv$. The following are equivalent for an AT $Q \in \fsp(M)$ of shape $\mu$ and an AT $P \in \fsp(M)$ of shape $\lambda$:
\begin{list}{\emph{(\alph{ctr})}} {\usecounter{ctr} \setlength{\itemsep}{1pt} \setlength{\topsep}{2pt}}
\item In the minimal cellular quotient of $M$ containing $\Gamma_P$, the AT  $Q$ is the unique tableau of shape $\gd \mu$ in degree $\geq \deg(P) - d(\mu,\lambda)$.
\item In the minimal cellular submodule of $M$ containing $\Gamma_Q$, the AT $P$ is the unique tableau of shape $\ld \lambda$ in degree $\leq \deg(Q) + d(\mu,\lambda)$.
\end{list}
\end{conjecture}

\begin{definition}
If  $P$ and $Q$ satisfy (a) and (b) of Conjecture \ref{cj skew linking} for some $M$ satisfying the stated condition, then $P$ is \emph{skew-linked to} $Q$ \emph{by} $\theta$, or  $P$ and  $Q$ are \emph{skew-linked}.
\end{definition}

With the notation of Conjecture \ref{cj skew linking}, let $P = G_\lambda$ and $M = \R_\lambda$.  Then by Theorem \ref{t main theorem gp} and a dual version of Proposition \ref{p chen skew-link}, there is a unique $U$ of shape $\mu$ in  $\fsp(\R_\lambda)$ such that  $G_\lambda$ is skew-linked to $U$ by  $\theta$.

For a skew shape $\theta = \Theta / \nu$ with $\skl{\lambda}{\theta}{\mu}$, define the intervals $[b_r, d_r]$ for $r \in [\ell(\lambda)]$ as follows: let $c = \nu_r$. If $\nu_r \neq 0$, then $b_i = \mu'_{c+1}$, $d_i = \mu'_c$; if $\nu_r = 0$, then $b_r = d_r = \ell(\lambda) + 1 - r$ (which is $\leq \mu'_1$).

\begin{example}
Let $\theta$ be as shown.
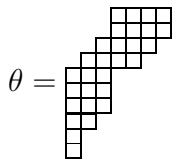
\[ \Yboxdim6pt
\theta = \young(:::\hfil\hfil\hfil\hfil,:::\hfil\hfil\hfil\hfil,::\hfil\hfil\hfil\hfil,:\hfil\hfil\hfil\hfil,\hfil\hfil\hfil,\hfil\hfil\hfil,\hfil\hfil\hfil,\hfil\hfil,\hfil,\hfil) \]
The intervals $[b_1, d_1] \dots [b_{10},d_{10}]$ are
\[ [4,5]\ [4,5]\ [5,5]\ [5,6]\ [6,6]\ [5,5]\ [4,4]\ [3,3]\ [2,2]\ [1,1]. \]
\end{example}

%For a (skew) shape $\theta$, let $\theta_{<r}$ be the shape obtained by removing the first $r$ rows of $\theta$.

\begin{definition}
Let $\skl{\lambda}{\theta}{\mu}$ and $b_r, d_r$ be as above. Suppose $G_\lambda$ is skew-linked to  $U$ by  $\theta$. The \emph{Li-Chung Chen ccp} $\chena{U}{G_\lambda}$ is the intersection of $\swra{G_\lambda}{\eta}$ over those $\ell(\lambda)$-compositions $\eta$ such that $\eta_{i} \leq d_{l_i+1}$ for $i \in [\ell(\lambda)]$, where $l_j = \sum_{i=1}^{j-1} \eta_i$.

%??? to understand what $U_<r$ is
%
%Equivalently, $\chena{U}{P}$ can be defined inductively: define $\lambda_{<r}$ and $\mu_{<r}$ by $\skl{\lambda_{<r}}{\theta_{<r}}{\mu_{<r}}$. Then $\chena{U}{P}$ is the set of AT $T$ such that $\cat_{R_1}(T) \in \chena{U_{<r}}{P_{<r}}$ for  $R_1:=\lambda_{\geq r}$, $r \in [d_1]$.
\end{definition}

\renewcommand{\duma}{16}
\renewcommand{\dumb}{15}
\begin{example}
\label{ex chen atoms}
The leftmost cocyclage poset in the top row of Figure \ref{f atomcopies} is the Chen ccp $\chena{U}{G_{(3,2,1)}}$,  $U = {\tiny\Yboxdim9pt\young(1234\duma,\dumb)}$, corresponding to the skew linking shape $\Yboxdim7pt \young(::\hfil\hfil\hfil,\hfil\hfil,\hfil)$.
\end{example}

Assuming Conjecture \ref{cj init block}, then $\chena{U}{G_\lambda}$ is isomorphic to Chen's original definition of these ccp in terms of semistandard tableaux \cite{Ch}.

Also, if we take $U = \beta'(\dual{Z}_\mu)$ with $\dual{Z}_\mu$ the transpose of the superstandard tableau of shape and content $\mu$ and  $\beta'$ as in Theorem \ref{t ssyt in pat}, then there is a unique $P$ of shape $\lambda$ in $\gpda{U}{\sgnp(U)}$ in degree $\deg(U)+d(\mu,\lambda)$. Let $\chena{U}{P}$ be defined analogously to $\chena{U}{G_\lambda}$ with $\swca{U}{\eta}$ in place of $\swra{G_\lambda}{\eta}$.  Then $\chena{U}{P}$ is the original definition of Chen ccp under the correspondence $\beta$.

\begin{conjecture}\label{cj chen atoms}
Suppose $\skl{\lambda}{\theta}{\mu}$ and $G_\lambda$ is skew-linked to $U$ by $\theta$. Then
\begin{list}{\emph{(\alph{ctr})}} {\usecounter{ctr} \setlength{\itemsep}{1pt} \setlength{\topsep}{2pt}}
\item $\mathscr{F}(\moda{\mu}{\lambda})t^{\deg{U}} = \mathscr{F}(\chena{U}{G_\lambda})$.
\item $\fmod(\csqa{U}{G_\lambda}) = \moda{\mu}{\lambda}$.
\item $\fsp(\csqa{U}{G_\lambda}) = \chena{U}{G_\lambda}$.
\item if $P$ is skew-linked to $Q$ by $\theta$, then $\csqa{Q}{P} \cong \csqa{U}{G_\lambda}$.
\end{list}
\end{conjecture}

Note that since $\fmod(\csqa{U}{G_\lambda}) \supseteq \moda{\mu}{\lambda}$, (a) and (c) imply (b).
Chen has verified statement (a) of this conjecture up to $n=7$. Computing in Magma, we verified Conjecture \ref{cj skew linking} and (c) and (d) of the above conjecture (and therefore (b) as well) for $n$ up to 6.
%Computing with Magma, we found all skew linking pairs in $\cinvn{6}$.  For each skew linking pair $\skl{\lambda}{\theta}{\mu}$, we computed the Chen ccp $\chena{U}{G_\lambda}$ and checked that it is equal to $\fsp(\csqa{U}{G_\lambda})$, thus verifying Conjecture \ref{cj chen atoms} (a) for $n$ = 6. Then for each $P$ skew-linked to $Q$ by $\theta$, we verified that $\csqa{Q}{P}$ is isomorphic to $\csqa{U}{G_\lambda}$.  Furthermore, the pairs such that $P$ is skew-linked to $Q$ by $\theta$ can also be found by looking for all sub-ccp in  $\fsp(\cinvn{6})$ isomorphic to $\chena{U}{G_\lambda}$.
%This is a substantial data set as there are over 50 types of edges in the coinvariants, excluding cocyclage-edges

\begin{example}
Continuing Example \ref{ex chen atoms}, the top row of Figure \ref{f atomcopies} contains all cellular subquotients of $\cinv$ isomorphic to $\csqa{U}{G_{(3,2,1)}}$. The bottom row contains two other cellular subquotients of $\eH$ isomorphic to $\csqa{U}{G_{( 3, 2, 1)}}$.  By Example \ref{ex dke}, the two subquotients on the bottom row are isomorphic.
\end{example}

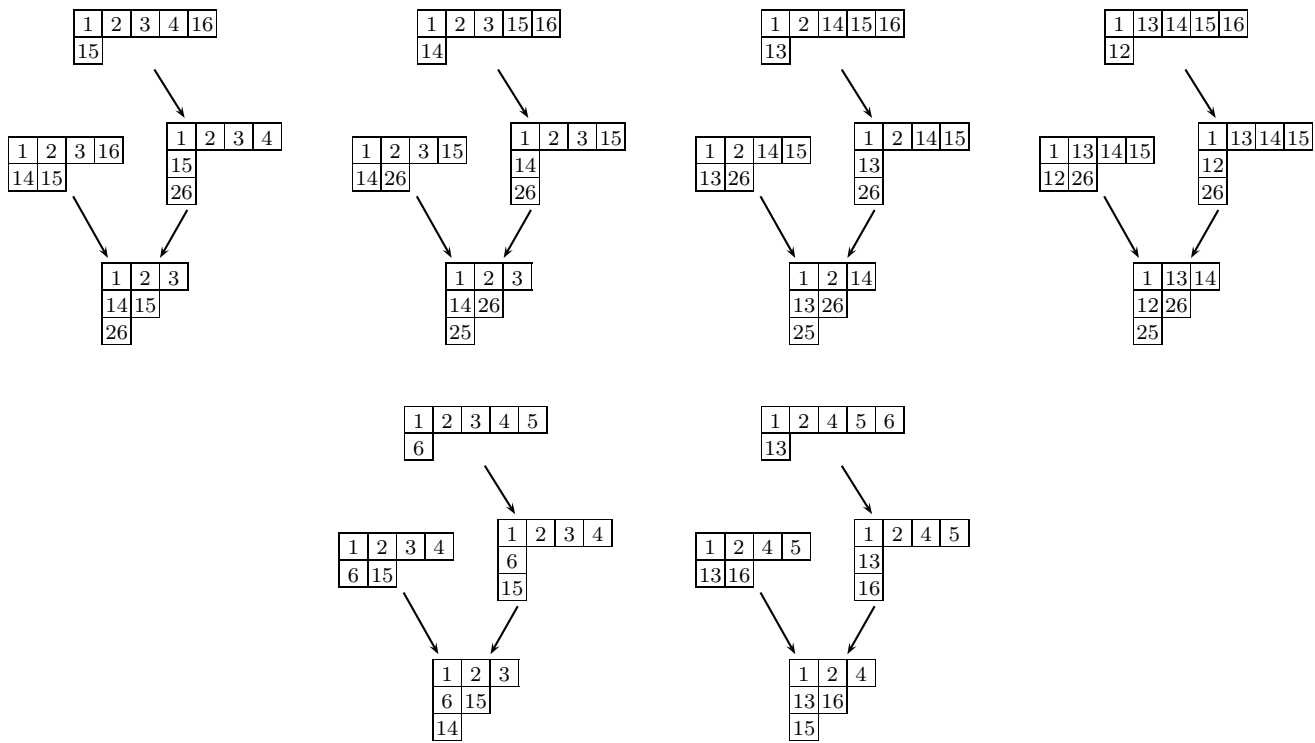
\begin{figure}[htp]
\begin{pspicture}(0pt,0pt)(15,300pt){\tiny

\hoogte=10pt
\breedte=11pt
\dikte=0.2pt

%\newdimen\horizcent
%\newdimen\ycor
%\newdimen\xcor
%\newdimen\horizspace
%\newdimen\temp
\newdimen\rowoneh
\rowoneh=275pt
\horizcent=370pt

\ycor=\rowoneh
\advance\ycor by 0pt
\horizspace=48pt
\xcor=\horizcent
\temp=48pt
\multiply \temp by 0 \divide \temp by 2
\advance\xcor by -\temp
\rput(\xcor,\ycor){\rnode{v71h4}{\begin{Young}
1&13&14&15&16\cr
12\cr
\end{Young}}}
\advance\xcor by \horizspace
\advance\ycor by -48pt
\horizspace=60pt
\xcor=\horizcent
\temp=60pt
\multiply \temp by 1 \divide \temp by 2
\advance\xcor by -\temp
\rput(\xcor,\ycor){\rnode{v81h2}{\begin{Young}
1&13&14&15\cr
12&26\cr
\end{Young}}}
\advance\xcor by \horizspace
\rput(\xcor,\ycor){\rnode{v81h3}{\begin{Young}
1&13&14&15\cr
12\cr
26\cr
\end{Young}}}
\advance\xcor by \horizspace
\advance\ycor by -53pt
\horizspace=48pt
\xcor=\horizcent
\temp=48pt
\multiply \temp by 0 \divide \temp by 2
\advance\xcor by -\temp
\rput(\xcor,\ycor){\rnode{v91h1}{\begin{Young}
1&13&14\cr
12&26\cr
25\cr
\end{Young}}}
\advance\xcor by \horizspace
%\ncline[nodesep=2pt]{->}{v91h1}{v81h2}
\ncline[nodesep=2pt]{->}{v81h2}{v91h1}
\ncline[nodesep=2pt]{->}{v81h3}{v91h1}
%\ncline[nodesep=2pt]{->}{v81h3}{v71h4}
\ncline[nodesep=2pt]{->}{v71h4}{v81h3}

\horizcent=240pt
\ycor=\rowoneh
\advance\ycor by 0pt
\horizspace=48pt
\xcor=\horizcent
\temp=48pt
\multiply \temp by 0 \divide \temp by 2
\advance\xcor by -\temp
\rput(\xcor,\ycor){\rnode{v61h4}{\begin{Young}
1&2&14&15&16\cr
13\cr
\end{Young}}}
\advance\xcor by \horizspace
\advance\ycor by -48pt
\horizspace=60pt
\xcor=\horizcent
\temp=60pt
\multiply \temp by 1 \divide \temp by 2
\advance\xcor by -\temp
\rput(\xcor,\ycor){\rnode{v71h2}{\begin{Young}
1&2&14&15\cr
13&26\cr
\end{Young}}}
\advance\xcor by \horizspace
\rput(\xcor,\ycor){\rnode{v71h3}{\begin{Young}
1&2&14&15\cr
13\cr
26\cr
\end{Young}}}
\advance\xcor by \horizspace
\advance\ycor by -53pt
\horizspace=48pt
\xcor=\horizcent
\temp=48pt
\multiply \temp by 0 \divide \temp by 2
\advance\xcor by -\temp
\rput(\xcor,\ycor){\rnode{v81h1}{\begin{Young}
1&2&14\cr
13&26\cr
25\cr
\end{Young}}}
\advance\xcor by \horizspace
%\ncline[nodesep=2pt]{->}{v81h1}{v71h2}
\ncline[nodesep=2pt]{->}{v71h2}{v81h1}
\ncline[nodesep=2pt]{->}{v71h3}{v81h1}
%\ncline[nodesep=2pt]{->}{v71h3}{v61h4}
\ncline[nodesep=2pt]{->}{v61h4}{v71h3}

\horizcent=110pt
\ycor=\rowoneh
\advance\ycor by 0pt
\horizspace=48pt
\xcor=\horizcent
\temp=48pt
\multiply \temp by 0 \divide \temp by 2
\advance\xcor by -\temp
\rput(\xcor,\ycor){\rnode{v51h4}{\begin{Young}
1&2&3&15&16\cr
14\cr
\end{Young}}}
\advance\xcor by \horizspace
\advance\ycor by -48pt
\horizspace=60pt
\xcor=\horizcent
\temp=60pt
\multiply \temp by 1 \divide \temp by 2
\advance\xcor by -\temp
\rput(\xcor,\ycor){\rnode{v61h3}{\begin{Young}
1&2&3&15\cr
14&26\cr
\end{Young}}}
\advance\xcor by \horizspace
\rput(\xcor,\ycor){\rnode{v61h2}{\begin{Young}
1&2&3&15\cr
14\cr
26\cr
\end{Young}}}
\advance\xcor by \horizspace
\advance\ycor by -53pt
\horizspace=60pt
\xcor=\horizcent
\temp=60pt
\multiply \temp by 0 \divide \temp by 2
\advance\xcor by -\temp
\rput(\xcor,\ycor){\rnode{v71h1}{\begin{Young}
1&2&3\cr
14&26\cr
25\cr
\end{Young}}}
\advance\xcor by \horizspace
%\ncline[nodesep=2pt]{->}{v71h1}{v61h3}
\ncline[nodesep=2pt]{->}{v61h2}{v71h1}
%\ncline[nodesep=2pt]{->}{v61h2}{v51h4}
\ncline[nodesep=2pt]{->}{v61h3}{v71h1}
\ncline[nodesep=2pt]{->}{v51h4}{v61h2}

\horizcent=-20pt
\ycor=\rowoneh
\advance\ycor by 0pt
\horizspace=60pt
\xcor=\horizcent
\temp=48pt
\multiply \temp by 0 \divide \temp by 2
\advance\xcor by -\temp
\rput(\xcor,\ycor){\rnode{v41h4}{\begin{Young}
1&2&3&4&16\cr
15\cr
\end{Young}}}
\advance\xcor by \horizspace
\advance\ycor by -48pt
\horizspace=60pt
\xcor=\horizcent
\temp=60pt
\multiply \temp by 1 \divide \temp by 2
\advance\xcor by -\temp
\rput(\xcor,\ycor){\rnode{v51h2}{\begin{Young}
1&2&3&16\cr
14&15\cr
\end{Young}}}
\advance\xcor by \horizspace
\rput(\xcor,\ycor){\rnode{v51h3}{\begin{Young}
1&2&3&4\cr
15\cr
26\cr
\end{Young}}}
\advance\xcor by \horizspace
\advance\ycor by -53pt
\horizspace=48pt
\xcor=\horizcent
\temp=48pt
\multiply \temp by 0 \divide \temp by 2
\advance\xcor by -\temp
\rput(\xcor,\ycor){\rnode{v61h1}{\begin{Young}
1&2&3\cr
14&15\cr
26\cr
\end{Young}}}
\advance\xcor by \horizspace
%\ncline[nodesep=2pt]{->}{v61h1}{v51h2}
\ncline[nodesep=2pt]{->}{v51h2}{v61h1}
\ncline[nodesep=2pt]{->}{v51h3}{v61h1}
%\ncline[nodesep=2pt]{->}{v51h3}{v41h4}
\ncline[nodesep=2pt]{->}{v41h4}{v51h3}

%%%%%%%%%%%%%%%%%%%%%%%%%%%%%%%%%%%%%%%%%%%%%%%

\newdimen\rowtwoh
\rowtwoh=125pt

\horizcent=105pt
\ycor=\rowtwoh
\advance\ycor by 0pt
\horizspace=48pt
\xcor=\horizcent
\temp=48pt
\multiply \temp by 0 \divide \temp by 2
\advance\xcor by -\temp
\rput(\xcor,\ycor){\rnode{v51h4}{\begin{Young}
1&2&3&4&5\cr
6\cr
\end{Young}}}
\advance\xcor by \horizspace
\advance\ycor by -48pt
\horizspace=60pt
\xcor=\horizcent
\temp=60pt
\multiply \temp by 1 \divide \temp by 2
\advance\xcor by -\temp
\rput(\xcor,\ycor){\rnode{v61h3}{\begin{Young}
1&2&3&4\cr
6&15\cr
\end{Young}}}
\advance\xcor by \horizspace
\rput(\xcor,\ycor){\rnode{v61h2}{\begin{Young}
1&2&3&4\cr
6\cr
15\cr
\end{Young}}}
\advance\xcor by \horizspace
\advance\ycor by -53pt
\horizspace=60pt
\xcor=\horizcent
\temp=60pt
\multiply \temp by 0 \divide \temp by 2
\advance\xcor by -\temp
\rput(\xcor,\ycor){\rnode{v71h1}{\begin{Young}
1&2&3\cr
6&15\cr
14\cr
\end{Young}}}
\advance\xcor by \horizspace
%\ncline[nodesep=2pt]{->}{v71h1}{v61h3}
\ncline[nodesep=2pt]{->}{v61h2}{v71h1}
%\ncline[nodesep=2pt]{->}{v61h2}{v51h4}
\ncline[nodesep=2pt]{->}{v61h3}{v71h1}
\ncline[nodesep=2pt]{->}{v51h4}{v61h2}

\horizcent=240pt
\ycor=\rowtwoh
\advance\ycor by 0pt
\horizspace=60pt
\xcor=\horizcent
\temp=48pt
\multiply \temp by 0 \divide \temp by 2
\advance\xcor by -\temp
\rput(\xcor,\ycor){\rnode{v41h4}{\begin{Young}
1&2&4&5&6\cr
13\cr
\end{Young}}}
\advance\xcor by \horizspace
\advance\ycor by -48pt
\horizspace=60pt
\xcor=\horizcent
\temp=60pt
\multiply \temp by 1 \divide \temp by 2
\advance\xcor by -\temp
\rput(\xcor,\ycor){\rnode{v51h2}{\begin{Young}
1&2&4&5\cr
13&16\cr
\end{Young}}}
\advance\xcor by \horizspace
\rput(\xcor,\ycor){\rnode{v51h3}{\begin{Young}
1&2&4&5\cr
13\cr
16\cr
\end{Young}}}
\advance\xcor by \horizspace
\advance\ycor by -53pt
\horizspace=48pt
\xcor=\horizcent
\temp=48pt
\multiply \temp by 0 \divide \temp by 2
\advance\xcor by -\temp
\rput(\xcor,\ycor){\rnode{v61h1}{\begin{Young}
1&2&4\cr
13&16\cr
15\cr
\end{Young}}}
\advance\xcor by \horizspace
%\ncline[nodesep=2pt]{->}{v61h1}{v51h2}
\ncline[nodesep=2pt]{->}{v51h2}{v61h1}
\ncline[nodesep=2pt]{->}{v51h3}{v61h1}
%\ncline[nodesep=2pt]{->}{v51h3}{v41h4}
\ncline[nodesep=2pt]{->}{v41h4}{v51h3}

}
\end{pspicture}
\caption{Copies of Chen csq.}
\label{f atomcopies}
\end{figure}

\begin{remark}
\label{r copies important}
The main reason we think that finding nice combinatorics for describing copies of atoms is important is that in $\csqa{U}{G_\lambda}$, with $G_\lambda$ and $U$ skew-linked, there is a path of ascent-edges, corotation-edges, and Knuth transformations from any word to $\reading(G_\lambda)$, but there is no obvious path from a word to $\reading(U)$ in general. However, letting $\mu = \sh(U)$ and assuming Conjecture \ref{cj flip}, there is a  $P$ of shape $\lambda$ such that  $P$ is skew-linked to  $\dual{G}_\mu$.  We expect that  $\csqa{U}{G_\lambda} \cong \csqa{\dual{G}_\mu}{P}$, and there is a path of ascent-edges, corotation-edges and Knuth transformations from $\reading(\dual{G}_\mu)$ to any word in  $\csqa{\dual{G}_\mu}{P}$.  Thus understanding both these copies simultaneously would give a very nice description of the atom.
\end{remark}

\begin{remark}
The cocyclage poset $\chena{U}{G_\lambda}$ is not necessarily connected. The smallest $n$ for which this occurs is $n = 6$ and the only skew linking shape corresponding to a disconnected ccp is \Yboxdim6pt $\young(:::\hfil\hfil,::\hfil,:\hfil,\hfil,\hfil)$. It may be the case that all LLM ccp are connected (see the next subsection).
\end{remark}
\subsection{}
\label{ss LLM atom}
The Lascoux-Lapointe-Morse super atoms of \cite{LLM} are conjecturally a special case of Chen ccp. Let us see how this comes about.

A partition $\lambda$ is $k$-\emph{bounded} if its parts have length $\leq k$. For a $k$-bounded partition $\lambda$ with $r$ parts, the skew shape $\theta^k(\lambda) = \Theta / \nu$ is defined uniquely by the condition $\row(\Theta / \nu) = \lambda$ and the inductive conditions: $\Theta_r = \lambda_r$, $\nu_r = 0$; $\widehat{\theta^k(\lambda)} = \theta^k(\widehat{\lambda})$ and $\nu_1$ is the smallest non-negative integer such that the hook lengths of $\theta^k(\lambda)$ are $\leq k$ ($\widehat{\lambda}$ denotes the partition obtained by removing the first part of $\lambda$). See (\ref{e k-conjugate}).

\be
\label{e k-conjugate}
\Yboxdim6pt
\begin{array}{cc}
\yng(4,3,2,2,2,1) &\young(:::::\hfil\hfil\hfil\hfil,::\hfil\hfil\hfil,:\hfil\hfil,\hfil\hfil,\hfil\hfil,\hfil) \\
{\small \lambda} & {\small \theta^4(\lambda)}
\end{array} \ee

\begin{proposition}[{\cite[Property 33]{LLM}}]
\label{p LLM k conjugate}
For a $k$-bounded partition $\lambda$, $\skl{\lambda}{\theta^k}{\mu}$ for some partition $\mu$ whose conjugate is also $k$-bounded.
\end{proposition}
In the language of \cite{LLM}, the $k$-conjugate of $\lambda$ is $\mu'$ in the proposition.

\begin{definition}
For $\mu$, $\lambda$ as in Proposition \ref{p LLM k conjugate} and tableaux $G_\lambda, U$ that are skew-linked by $\theta^k$, the \emph{LLM atom} of $G_\lambda, U$ is $\chena{U}{G_\lambda}$.
\end{definition}

\section*{Acknowledgments}
I am deeply grateful to my advisor Mark Haiman for his wonderful ideas,  his generous advice, and our many detailed discussions that made this paper possible. I am grateful to Ryo Masuda, John Wood, Michael Phillips, and Kristofer Henriksson for help typing and typesetting figures.

 \end{document}